\documentclass[a4paper, 11pt, final]{amsart}    
\usepackage{latexsym, amsmath, amsthm, amssymb, setspace, verbatim, bbm}
\usepackage[all]{xy}
\usepackage[utf8]{inputenc} \usepackage[T1]{fontenc} \usepackage{lmodern}
\usepackage{hyperref, aliascnt}
\usepackage{enumitem}
\usepackage{color}
\usepackage{longtable}

\title[The dynamical Kirchberg--Phillips theorem]{The dynamical Kirchberg--Phillips theorem}
\author{James Gabe}
\author{Gábor Szabó}
\address{Department of Mathematics and Computer Science, University of Southern\linebreak
\phantom{-}\hspace{2mm} Denmark, Campusvej 55, 5230 Odense, Denmark}
\email{gabe@imada.sdu.dk}
\address{Department of Mathematics, KU Leuven, Celestijnenlaan 200b, box 2400\linebreak
\phantom{-}\hspace{2mm} B-3001 Leuven, Belgium}
\email{gabor.szabo@kuleuven.be}
\dedicatory{In memory of Eberhard Kirchberg.}


\numberwithin{equation}{section}

\begin{document}

% math
\renewcommand\matrix[1]{\left(\begin{array}{*{10}{c}} #1 \end{array}\right)}  % Matrix
\newcommand\set[1]{\left\{#1\right\}}  % Menge

%% Besondere Variablen
%Zahlmengen-Stil
\newcommand{\IA}[0]{\mathbb{A}} \newcommand{\IB}[0]{\mathbb{B}}
\newcommand{\IC}[0]{\mathbb{C}} \newcommand{\ID}[0]{\mathbb{D}}
\newcommand{\IE}[0]{\mathbb{E}} \newcommand{\IF}[0]{\mathbb{F}}
\newcommand{\IG}[0]{\mathbb{G}} \newcommand{\IH}[0]{\mathbb{H}}
\newcommand{\II}[0]{\mathbb{I}} \renewcommand{\IJ}[0]{\mathbb{J}}
\newcommand{\IK}[0]{\mathbb{K}} \newcommand{\IL}[0]{\mathbb{L}}
\newcommand{\IM}[0]{\mathbb{M}} \newcommand{\IN}[0]{\mathbb{N}}
\newcommand{\IO}[0]{\mathbb{O}} \newcommand{\IP}[0]{\mathbb{P}}
\newcommand{\IQ}[0]{\mathbb{Q}} \newcommand{\IR}[0]{\mathbb{R}}
\newcommand{\IS}[0]{\mathbb{S}} \newcommand{\IT}[0]{\mathbb{T}}
\newcommand{\IU}[0]{\mathbb{U}} \newcommand{\IV}[0]{\mathbb{V}}
\newcommand{\IW}[0]{\mathbb{W}} \newcommand{\IX}[0]{\mathbb{X}}
\newcommand{\IY}[0]{\mathbb{Y}} \newcommand{\IZ}[0]{\mathbb{Z}}

\newcommand{\Ia}[0]{\mathbbmss{a}} \newcommand{\Ib}[0]{\mathbbmss{b}}
\newcommand{\Ic}[0]{\mathbbmss{c}} \newcommand{\Id}[0]{\mathbbmss{d}}
\newcommand{\Ie}[0]{\mathbbmss{e}} \newcommand{\If}[0]{\mathbbmss{f}}
\newcommand{\Ig}[0]{\mathbbmss{g}} \newcommand{\Ih}[0]{\mathbbmss{h}}
\newcommand{\Ii}[0]{\mathbbmss{i}} \newcommand{\Ij}[0]{\mathbbmss{j}}
\newcommand{\Ik}[0]{\mathbbmss{k}} \newcommand{\Il}[0]{\mathbbmss{l}}
\renewcommand{\Im}[0]{\mathbbmss{m}} \newcommand{\In}[0]{\mathbbmss{n}}
\newcommand{\Io}[0]{\mathbbmss{o}} \newcommand{\Ip}[0]{\mathbbmss{p}}
\newcommand{\Iq}[0]{\mathbbmss{q}} \newcommand{\Ir}[0]{\mathbbmss{r}}
\newcommand{\Is}[0]{\mathbbmss{s}} \newcommand{\It}[0]{\mathbbmss{t}}
\newcommand{\Iu}[0]{\mathbbmss{u}} \newcommand{\Iv}[0]{\mathbbmss{v}}
\newcommand{\Iw}[0]{\mathbbmss{w}} \newcommand{\Ix}[0]{\mathbbmss{x}}
\newcommand{\Iy}[0]{\mathbbmss{y}} \newcommand{\Iz}[0]{\mathbbmss{z}}

%Geschwungener Stil
\newcommand{\CA}[0]{\mathcal{A}} \newcommand{\CB}[0]{\mathcal{B}}
\newcommand{\CC}[0]{\mathcal{C}} \newcommand{\CD}[0]{\mathcal{D}}
\newcommand{\CE}[0]{\mathcal{E}} \newcommand{\CF}[0]{\mathcal{F}}
\newcommand{\CG}[0]{\mathcal{G}} \newcommand{\CH}[0]{\mathcal{H}}
\newcommand{\CI}[0]{\mathcal{I}} \newcommand{\CJ}[0]{\mathcal{J}}
\newcommand{\CK}[0]{\mathcal{K}} \newcommand{\CL}[0]{\mathcal{L}}
\newcommand{\CM}[0]{\mathcal{M}} \newcommand{\CN}[0]{\mathcal{N}}
\newcommand{\CO}[0]{\mathcal{O}} \newcommand{\CP}[0]{\mathcal{P}}
\newcommand{\CQ}[0]{\mathcal{Q}} \newcommand{\CR}[0]{\mathcal{R}}
\newcommand{\CS}[0]{\mathcal{S}} \newcommand{\CT}[0]{\mathcal{T}}
\newcommand{\CU}[0]{\mathcal{U}} \newcommand{\CV}[0]{\mathcal{V}}
\newcommand{\CW}[0]{\mathcal{W}} \newcommand{\CX}[0]{\mathcal{X}}
\newcommand{\CY}[0]{\mathcal{Y}} \newcommand{\CZ}[0]{\mathcal{Z}}

%Script Stil
\newcommand{\FA}[0]{\mathfrak{A}} \newcommand{\FB}[0]{\mathfrak{B}}
\newcommand{\FC}[0]{\mathfrak{C}} \newcommand{\FD}[0]{\mathfrak{D}}
\newcommand{\FE}[0]{\mathfrak{E}} \newcommand{\FF}[0]{\mathfrak{F}}
\newcommand{\FG}[0]{\mathfrak{G}} \newcommand{\FH}[0]{\mathfrak{H}}
\newcommand{\FI}[0]{\mathfrak{I}} \newcommand{\FJ}[0]{\mathfrak{J}}
\newcommand{\FK}[0]{\mathfrak{K}} \newcommand{\FL}[0]{\mathfrak{L}}
\newcommand{\FM}[0]{\mathfrak{M}} \newcommand{\FN}[0]{\mathfrak{N}}
\newcommand{\FO}[0]{\mathfrak{O}} \newcommand{\FP}[0]{\mathfrak{P}}
\newcommand{\FQ}[0]{\mathfrak{Q}} \newcommand{\FR}[0]{\mathfrak{R}}
\newcommand{\FS}[0]{\mathfrak{S}} \newcommand{\FT}[0]{\mathfrak{T}}
\newcommand{\FU}[0]{\mathfrak{U}} \newcommand{\FV}[0]{\mathfrak{V}}
\newcommand{\FW}[0]{\mathfrak{W}} \newcommand{\FX}[0]{\mathfrak{X}}
\newcommand{\FY}[0]{\mathfrak{Y}} \newcommand{\FZ}[0]{\mathfrak{Z}}

\newcommand{\Fa}[0]{\mathfrak{a}} \newcommand{\Fb}[0]{\mathfrak{b}}
\newcommand{\Fc}[0]{\mathfrak{c}} \newcommand{\Fd}[0]{\mathfrak{d}}
\newcommand{\Fe}[0]{\mathfrak{e}} \newcommand{\Ff}[0]{\mathfrak{f}}
\newcommand{\Fg}[0]{\mathfrak{g}} \newcommand{\Fh}[0]{\mathfrak{h}}
\newcommand{\Fi}[0]{\mathfrak{i}} \newcommand{\Fj}[0]{\mathfrak{j}}
\newcommand{\Fk}[0]{\mathfrak{k}} \newcommand{\Fl}[0]{\mathfrak{l}}
\newcommand{\Fm}[0]{\mathfrak{m}} \newcommand{\Fn}[0]{\mathfrak{n}}
\newcommand{\Fo}[0]{\mathfrak{o}} \newcommand{\Fp}[0]{\mathfrak{p}}
\newcommand{\Fq}[0]{\mathfrak{q}} \newcommand{\Fr}[0]{\mathfrak{r}}
\newcommand{\Fs}[0]{\mathfrak{s}} \newcommand{\Ft}[0]{\mathfrak{t}}
\newcommand{\Fu}[0]{\mathfrak{u}} \newcommand{\Fv}[0]{\mathfrak{v}}
\newcommand{\Fw}[0]{\mathfrak{w}} \newcommand{\Fx}[0]{\mathfrak{x}}
\newcommand{\Fy}[0]{\mathfrak{y}} \newcommand{\Fz}[0]{\mathfrak{z}}

%Modifikation der Variablen
\renewcommand{\phi}[0]{\varphi}
\newcommand{\eps}[0]{\varepsilon}

%zusätzliche Features
\newcommand{\id}[0]{\operatorname{id}}		% Identität
\newcommand{\eins}[0]{\mathbf{1}}			% Eine Eins in allgemeinerem Kontext, z.B. in einem Ring
\newcommand{\diag}[0]{\operatorname{diag}}
\newcommand{\ad}[0]{\operatorname{Ad}}
\newcommand{\ev}[0]{\operatorname{ev}}
\newcommand{\fin}[0]{{\subset\!\!\!\subset}}
\newcommand{\Aut}[0]{\operatorname{Aut}}
\newcommand{\dst}[0]{\displaystyle}
\newcommand{\cstar}[0]{\ensuremath{\mathrm{C}^*}}
\newcommand{\dist}[0]{\operatorname{dist}}
\newcommand{\cc}[0]{\simeq_{\mathrm{cc}}}
\newcommand{\ue}[0]{{~\approx_{\mathrm{u}}}~}
\newcommand{\prim}[0]{\ensuremath{\mathrm{Prim}}}
\newcommand{\GL}[0]{\operatorname{GL}}
\newcommand{\Hom}[0]{\operatorname{Hom}}
\newcommand{\sep}[0]{\ensuremath{\mathrm{sep}}}
\newcommand{\wc}[0]{\preccurlyeq}
\newcommand{\supp}[0]{\operatorname{supp}}
\renewcommand{\asymp}[0]{\sim_{\mathrm{asymp}}}
\newcommand{\asuc}[0]{\precsim_{\mathrm{as.u.}}}

% theorems
\newtheorem{satz}{Satz}[section]		% <--- optional, zählt so mit den Abschnitten

\newaliascnt{corCT}{satz}
\newtheorem{cor}[corCT]{Corollary}
\aliascntresetthe{corCT}
\providecommand*{\corCTautorefname}{Corollary}
\newaliascnt{lemmaCT}{satz}
\newtheorem{lemma}[lemmaCT]{Lemma}
\aliascntresetthe{lemmaCT}
\providecommand*{\lemmaCTautorefname}{Lemma}
\newaliascnt{propCT}{satz}
\newtheorem{prop}[propCT]{Proposition}
\aliascntresetthe{propCT}
\providecommand*{\propCTautorefname}{Proposition}
\newaliascnt{theoremCT}{satz}
\newtheorem{theorem}[theoremCT]{Theorem}
\aliascntresetthe{theoremCT}
\providecommand*{\theoremCTautorefname}{Theorem}
\newtheorem*{theoreme}{Theorem}

\theoremstyle{definition}

\newaliascnt{conjectureCT}{satz}
\newtheorem{conjecture}[conjectureCT]{Conjecture}
\aliascntresetthe{conjectureCT}
\providecommand*{\conjectureCTautorefname}{Conjecture}
\newaliascnt{defiCT}{satz}
\newtheorem{defi}[defiCT]{Definition}
\aliascntresetthe{defiCT}
\providecommand*{\defiCTautorefname}{Definition}
\newtheorem*{defie}{Definition}
\newaliascnt{notaCT}{satz}
\newtheorem{nota}[notaCT]{Notation}
\aliascntresetthe{notaCT}
\providecommand*{\notaCTautorefname}{Notation}
\newtheorem*{notae}{Notation}
\newaliascnt{remCT}{satz}
\newtheorem{rem}[remCT]{Remark}
\aliascntresetthe{remCT}
\providecommand*{\remCTautorefname}{Remark}
\newtheorem*{reme}{Remark}
\newaliascnt{exampleCT}{satz}
\newtheorem{example}[exampleCT]{Example}
\aliascntresetthe{exampleCT}
\providecommand*{\exampleCTautorefname}{Example}
\newaliascnt{questionCT}{satz}
\newtheorem{question}[questionCT]{Question}
\aliascntresetthe{questionCT}
\providecommand*{\questionCTautorefname}{Question}
\newtheorem*{questione}{Question}

\newcounter{theoremintro}
\renewcommand*{\thetheoremintro}{\Alph{theoremintro}}
\newaliascnt{theoremiCT}{theoremintro}
\newtheorem{theoremi}[theoremiCT]{Theorem}
\aliascntresetthe{theoremiCT}
\providecommand*{\theoremiCTautorefname}{Theorem}
\newaliascnt{defiiCT}{theoremintro}
\newtheorem{defii}[defiiCT]{Definition}
\aliascntresetthe{defiiCT}
\providecommand*{\defiiCTautorefname}{Definition}
\newaliascnt{coriCT}{theoremintro}
\newtheorem{cori}[coriCT]{Corollary}
\aliascntresetthe{coriCT}
\providecommand*{\coriCTautorefname}{Corollary}

%%%%%%%%%%%%%%%%%%%%%%%%%%%%%%%%%%%%%%%%%%%%

\begin{abstract} 
Let $G$ be a second-countable, locally compact group.
In this article we study amenable $G$-actions on Kirchberg algebras that admit an approximately central embedding of a canonical quasi-free action on the Cuntz algebra $\CO_\infty$.
If $G$ is discrete, this coincides with the class of amenable and outer $G$-actions on Kirchberg algebras.
We show that the resulting $G$-\cstar-dynamical systems are classified by equivariant Kasparov theory up to cocycle conjugacy.
This is the first classification theory of its kind applicable to actions of arbitrary locally compact groups.
Among various applications, our main result solves a conjecture of Izumi for actions of discrete amenable torsion-free groups, and recovers the main results of recent work by Izumi--Matui for actions of poly-$\IZ$ groups.
%For actions of $\IR$, we also obtain an alternative proof of Kishimoto's conjecture that every Kirchberg algebra carries a unique Rokhlin flow, and in fact the approach immediately yields the higher-dimensional generalization of this statement.
\end{abstract}

\maketitle

\setcounter{tocdepth}{1}
\tableofcontents

%%%%%%%%%%%%%%%%%%%%%%%%%%%

\section*{Introduction}

The present paper breaks new ground in the classification of group actions on \cstar-algebras up to cocycle conjugacy.
The development of such a classification theory represents the overarching goal behind a lot of past and present research in the area of operator algebras and can be argued to form a branch of its own.
Noncommutative dynamics is deeply rooted and ubiquitous within the subject of operator algebras, arguably because of the interesting ways in which groups can act on noncommutative structures.
Whether one wishes to mention $\IR$-actions in Tomita--Takesaki theory \cite{Takesaki70, Connes73}, the role of single automorphisms within Connes--Haagerup's classification of injective factors \cite{Connes76, Haagerup87} or later subsequent developments in Jones' subfactor theory \cite{Jones83, Popa90, Popa94, Popa10} or Popa's deformation/rigidity theory \cite{Popa07, Vaes10, Ioana18}, it is evident that dynamical ideas have been taking center stage on the side of von Neumann algebras for a long time.
The von Neumann algebraic result most pertinent to the context of this paper is the complete classification of actions of discrete amenable groups on injective factors \cite{Connes75, Connes77, Jones80, Ocneanu85, KawahigashiSutherlandTakesaki92, KatayamaSutherlandTakesaki98, Masuda13}, whose famous type II subcase was pioneered by Connes, Jones and Ocneanu.

A special case of our main result \autoref{intro:main-result} is the following \cstar-algebraic analogue of the above, though we wish to emphasise that our main theorem is applicable to all second countable, locally compact groups.
The general statement merely requires setting up some additional terminology.

\begin{theoremi} \label{intro:main-result-discrete}
Let $G$ be a countable discrete amenable group and let $\alpha: G\curvearrowright A$ and $\beta: G\curvearrowright B$ be pointwise outer actions on stable Kirchberg algebras.\footnote{These are the separable, nuclear, purely infinite, simple \cstar-algebras classified in \cite{KirchbergC, Phillips00}.}
Then $(A,\alpha)$ and $(B,\beta)$ are cocycle conjugate if and only if they are $KK^G$-equivalent.
\end{theoremi}

A classification theory of this scope has long been sought-after for \cstar-algebras and has in fact motivated most of the interesting developments found in the literature.
Yet our understanding of \cstar-dynamics on simple \cstar-algebras has mostly remained  underdeveloped in direct comparison to von Neumann algebras, which in hindsight is not surprising given the extra challenge posed by $K$-theoretical obstructions.
This is perhaps most convincingly demonstrated in the recent article \cite{Meyer21} of Meyer.
It would be an unpractically herculean task to thoroughly review all the important developments found in the literature in this regard (especially relating to finite \cstar-algebras), hence we will be somewhat selective in what we mention, in particular considering past articles that have given more thorough reviews (at least in part) already.
An obvious quintessential reference is Izumi's survey article \cite{Izumi10}, though the reader may also wish to consult the introduction of \cite{Szabo21cc} and the references therein with an eye to more recent works of the past decade.

At the level of methodology, almost all classification results for group actions rely on some kind of Rokhlin property, which in one way or another is made to work in conjunction with the Evans--Kishimoto intertwining method \cite{EvansKishimoto97}.
This kind of modus operandi still underpins the majority of the state-of-the-art concerning actions of non-compact groups, be it for flows \cite{Szabo21R} or poly-$\IZ$ groups \cite{IzumiMatui21, IzumiMatui21_2}.
The drawback of this approach is that the actual implementation of the Evans--Kishimoto intertwining technique becomes considerably less realistic without having full control over the precise structure of the acting group, which is due to the fact that the technical obstacles one faces  become increasingly opaque in $K$-theoretic terms.
In this paper we hope to achieve a full paradigm shift by promoting an approach that instead mirrors much more closely the methodology of the Elliott program \cite{Elliott94, Kirchberg95, Winter17, White23} to classify simple nuclear \cstar-algebras.
More concretely, we classify group actions following the conceptual approach suggested by the second author in \cite{Szabo21cc}, which relies on the prevalence of so-called \emph{existence and uniqueness theorems} in conjunction with Elliott's intertwining machinery applied to the \emph{(proper) cocycle category} of \cstar-dynamics.

\begin{defii}
Let $G$ be any locally compact group.
Let $\alpha: G\curvearrowright A$ and $\beta: G\curvearrowright B$ be two actions on \cstar-algebras.
A proper cocycle morphism is a pair $(\phi,\Iu): (A,\alpha)\to (B,\beta)$, where $\phi: A\to B$ is a $*$-homomorphism, $\Iu: G\to\CU(\eins+B)$ is a norm-continuous $\beta$-cocycle, and one has the equation $\ad(\Iu_g)\circ\beta_g\circ\phi=\phi\circ\alpha_g$ for all $g\in G$.
\end{defii}

As is fleshed out in \cite{Szabo21cc}, the proper cocycle morphisms form the arrows in a category, which provides a fair deal of additional flexibility in comparison to just working with equivariant maps.
In particular it is natural in this framework to start from a proper cocycle morphism as above and \emph{perturb} it with any given unitary $v\in\CU(\eins+B)$ to arrive at another proper cocycle morphism $\ad(v)\circ(\phi,\Iu)=(\ad(v)\circ\phi, v\Iu_\bullet\beta_\bullet(v)^*)$.
This provides the concept of unitary equivalence among proper cocycle morphisms, and in complete analogy to what one does for $*$-homomorphisms, one can generalize this to get a suitable notion of (proper) approximate or asymptotic unitary equivalence.
This turns out to fit into Elliott's intertwining machinery \cite{Elliott10}, which provides a full-fledged analog of the fundamental methodology underpinning the Elliott program.
The latter has enjoyed enormous success in recent years in the context of finite \cstar-algebras, first driven by breakthroughs in the approach related to tracial approximation \cite{GongLinNiu20, GongLinNiu20-2, ElliottGongLinNiu15, ElliottGongLinNiu20, ElliottGongLinNiu20-2, GongLin20, GongLin21, GongLin22}, and more recently gaining momentum through a more refined understanding of how ultrapowers, traces, and $K$-theory interact \cite{TikuisisWhiteWinter17, Schafhauser20, CGSTW23}.

However, it cannot be emphasized enough how influential and groundbreaking the much earlier work of Kirchberg and Phillips \cite{KirchbergC, Phillips00, KirchbergPhillips00} has been, which classified the traceless algebras within Elliott's program, now commonly referred to as the class of Kirchberg algebras.
Their classification result was arguably the first classification of \cstar-algebras that was truly abstract. Using ideas of Rørdam from \cite{Rordam95}, they exploit Kasparov's bivariant $K$-theory \cite{Kasparov88} for the classification of $*$-homomorphisms. 
To summarize their main result, one has for any two stable Kirchberg algebras $A$ and $B$ that any invertible element in $KK(A,B)$ lifts to an isomorphism $A\cong B$.
Determining whether this is true becomes more tractable if one assumes these algebras to satisfy the \emph{universal coefficient theorem} (UCT) \cite{RosenbergSchochet87}, whereby a $KK$-equivalence can be obtained from an isomorphism between the ordinary $K$-groups of $A$ and $B$.

For the reasons stated above, the class of Kirchberg algebras plays a special role within the Elliott program, to the point where many problems pertaining to general classifiable \cstar-algebras are first considered for Kirchberg algebras as a supposedly easier test case.
The desire to unravel the structure of group actions is no exception, hence the case of Kirchberg algebras has long been considered as the natural starting point in this regard \cite{Nakamura00, Izumi04, Izumi04II, IzumiMatui10}. 
In fact they were still the focal point in Izumi--Matui's recent groundbreaking work \cite{IzumiMatui21, IzumiMatui20, IzumiMatui21_2}.
In view of these past developments, a few researchers have raised suspicions over the years that there ought to be a dynamical analog of the Kirchberg--Phillips theorem for outer actions of discrete amenable groups.

Phillips himself has notably promoted this viewpoint for actions of finite groups in conference talks dating back at least a decade, and it should be mentioned that he has had an unpublished draft proving a number of results overlapping with our present work for actions of finite groups, although the detailed methods differ significantly.
For actions of torsion-free groups, Izumi conjectured in \cite{Izumi10} that  outer actions are classified up to cocycle conjugacy via isomorphism classes of certain bundles, and his recent work with Matui is the positive solution to his conjecture for poly-$\IZ$ groups.
Meyer has discovered in \cite{Meyer21} that the Baum--Connes machinery can be employed to reformulate Izumi's conjecture in terms of equivariant $KK$-theory.
As we make more precise in the sixth section, we can use Meyer's observation and our main result \autoref{intro:main-result-discrete} to obtain a positive solution to Izumi's conjecture; see \autoref{thm:Izumi-conjecture}.
While the things mentioned above have been going on, a part of the second author's work was driven by the goal to get closer to some kind of dynamical Kirchberg--Phillips theorem, which is most obvious in \cite{Szabo18kp}.

Apart from the aforementioned difficulties related to the implementation of the Evans--Kishimoto method, another major obstacle to cover larger classes of acting groups has been a lack of any means to systematically employ equivariant $KK$-theory as an invariant.
If we look back at more recent milestones in \cstar-algebra classification, one can see upon close examination that the systematic use of $KK$-theory is most commonly achieved in the Cuntz picture \cite{Cuntz83, Cuntz84, Higson87} via the stable uniqueness theorem of Lin and Dadarlat--Eilers \cite{Lin02, DadarlatEilers01, DadarlatEilers02}.
Although something similar had been demonstrated in Dadarlat--Eilers' original work, the recent work of the first author \cite{Gabe21} includes a new proof of the original Kirchberg--Phillips theorem by exploiting the stable uniqueness theorem to its fullest.
Motivated by the importance of such methods in the Elliott program, the authors developed in \cite{GabeSzabo22} a dynamical generalization of the stable uniqueness theorem in the context of equivariant $KK$-theory for arbitrary locally compact groups.

In the present work, we take the next major step and prove the desired dynamical Kirchberg--Phillips theorem in the highest generality possible.
Although we originally set out to classify outer actions of discrete amenable groups, we can in fact cover actions of arbitrary (second-countable) locally compact groups on Kirchberg algebras under appropriate dynamical assumptions.
The first important assumption is that the involved actions ought to be amenable (instead of the acting groups), the theory of which was only fleshed out recently in full generality \cite{BussEchterhoffWillett20, BussEchterhoffWillett23, Suzuki19, OzawaSuzuki21}.
The basic idea that the amenability of actions can be used for classification even for non-amenable acting groups goes back to Suzuki \cite{Suzuki21}.
We note that this stands out as a feature that seems unique to \cstar-algebras, as it is for instance well-known that non-amenable groups cannot act amenably on factors.
The second important assumption is inspired by a known result for finite group actions \cite{GoldsteinIzumi11} and gives rise to a class of actions that we call \emph{isometrically shift-absorbing}:

\begin{defii}
Let $\mu$ be a left-invariant Haar measure on $G$ and $\lambda: G\to\CU(L^2(G,\mu))$ the left-regular representation.
Let $\beta: G\curvearrowright B$ be an action on a separable unital \cstar-algebra.\footnote{We note that unitality is only assumed for ease of notation here. Our definition in the main body of the paper, \autoref{defi:the-condition}, covers arbitrary separable \cstar-algebras.}
We say that $\beta$ is isometrically shift-absorbing, if there exists a linear map $\Fs: L^2(G,\mu)\to B_\infty\cap B'$ that is equivariant in the sense that $\beta_{\infty,g}\circ\Fs=\Fs\circ\lambda_g$ holds for all $g\in G$, and such that one has $\Fs(\xi)^*\Fs(\eta)=\langle\xi\mid\eta\rangle\cdot\eins_B$ for all $\xi,\eta\in L^2(G,\mu)$.\footnote{In order to be in line with common terminology related to Hilbert modules over \cstar-algebras, we will always assume the first component of an inner product to be anti-linear.}
\end{defii}

Although we have not yet succeeded in finding a characterization of isometrically shift-absorbing actions on Kirchberg algebras in more familiar terms for arbitrary locally compact groups, one can easily obtain one for discrete groups with the available literature.
In what essentially boils down to an observation of Izumi--Matui in \cite{IzumiMatui10}, an action of a countable discrete group on a Kirchberg algebra is isometrically shift-absorbing if and only if it is outer.
For topological groups, outerness is known to be a significantly weaker propery, however.
For instance, we prove that for actions of $\IR^k$ on $\CO_\infty$-absorbing \cstar-algebras, being isometrically shift-absorbing is equivalent to the Rokhlin property \cite{Kishimoto96_R}, although the analogous statement is not true for many other groups like compact ones.
As one can observe with the help of \cite[Theorem 6.1]{OzawaSuzuki21}, \emph{every} second-countable locally compact group admits actions on Kirchberg algebras that are amenable and isometrically shift-absorbing.
One might adopt the viewpoint that isometric shift-absorption is some general manifestation of a Rokhlin-type property, but it has several major benefits in comparison to any earlier candidates given in the literature.
On the one hand, it can be formulated for all groups, and on the other hand, it turns out to pose no $K$-theoretical obstruction as a property of actions.
The latter is for instance the primary drawback of the Rokhlin property for compact group actions \cite{Izumi04II, HirshbergPhillips15, AranoKubota17, Gardella19, Gardella22}.
In fact we can observe the following, which is \autoref{thm:range} in the main body of the article.

\begin{theoremi}[cf.~Pimsner, Kumjian, Meyer, Ozawa--Suzuki] \label{thmi:range}
Let $\alpha: G\curvearrowright A$ be an amenable action on a separable nuclear \cstar-algebra.
Then there exists an amenable and isometrically shift-absorbing action $\beta: G\curvearrowright B$ on a stable Kirchberg algebra and an equivariant embedding $(A,\alpha)\to (B, \beta)$ that induces a $KK^G$-equivalence.
By \autoref{intro:main-result} it follows that $(B, \beta)$ is unique up to cocycle conjugacy.
\end{theoremi}

In this article we prove a classification result for actions on Kirchberg algebras that are both amenable and isometrically shift-absorbing.
Our theory is enabled by the observation that our two dynamical assumptions work as a very powerful tool in conjunction.
In a nutshell, an action $\beta: G\curvearrowright B$ can be seen to be isometrically shift-absorbing if $B$ is locally approximated by $L^2(G,B)$ as an equivariant $B$-bimodule (\autoref{prop:the-condition}), whereas $\beta$ is amenable precisely when $L^2(G,B)$ admits a suitable net of approximate fixed points.
This grants us access to the kind of averaging arguments that one usually only has with some kind of Rokhlin property, and becomes the key ingredient that allows us to apply our stable uniqueness theorem \cite{GabeSzabo22} in a fruitful way.
All of this culminates in suitable existence and uniqueness theorems for the group actions under consideration, which we shall state here in an oversimplified form for the sake of readability.
For the precise statements, we refer to \autoref{thm:existence} and \autoref{thm:uniqueness}.

\begin{theoremi} \label{intro:existence-uniqueness}
Let $\alpha: G\curvearrowright A$ and $\beta: G\curvearrowright B$ be actions on Kirchberg algebras.
Suppose $\alpha$ and $\beta$ are amenable and isometrically shift-absorbing.
Assume that $\beta$ tensorially absorbs the trivial $G$-action on $\CK$ up to conjugacy.
Then the assignment $(\phi,\Iu)\mapsto KK^G(\phi,\Iu)$ induces a bijection between proper cocycle embeddings $(A,\alpha)\to (B,\beta)$ which are anchored (see \autoref{rem:KKGEG}), modulo strong asymptotic unitary equivalence, and the group $KK^G(\alpha,\beta)$.
\end{theoremi}

There is of course a unital version of the above theorem as well (see \autoref{thm:unital-existence} and \autoref{thm:unital-uniqueness}), but it should be noted that the level of generality is then in truth strictly lower because only exact groups can act amenably on unital \cstar-algebras; see \cite[Corollary 3.6]{OzawaSuzuki21}.
The above ends up being the optimal generalization of the known existence and uniqueness theorems in the context of the Kirchberg--Phillips theorem, whereby $*$-homomorphisms between stable Kirchberg algebras are classified by $KK$-theory up to asymptotic unitary equivalence.
Finally, the desired classification result for the actions can be deduced from the above theorem in conjunction with the Elliott intertwining machinery.
We shall again state it here in its most abridged form, and refer to \autoref{thm:main-theorem} for the detailed statement.

\begin{theoremi} \label{intro:main-result}
Let $\alpha: G\curvearrowright A$ and $\beta: G\curvearrowright B$ be actions on stable Kirchberg algebras.
Suppose $\alpha$ and $\beta$ are amenable and isometrically shift-absorbing.
Then $(A,\alpha)$ and $(B,\beta)$ are cocycle conjugate if and only if they are $KK^G$-equivalent.
\end{theoremi}

In light of an earlier remark, we can restrict to $G$ being discrete and conclude that we have classified all amenable and outer $G$-actions on Kirchberg algebras up to cocycle conjugacy.
This represents the first abstract classification result up to cocycle conjugacy by $K$-theoretical invariants that covers all (discrete) amenable groups.
One may of course argue that the main result of \cite{Szabo18kp} went in this direction, but the scope and methods therein were very narrow in comparison and relied on much more than just $K$-theoretical information.
Even if we were to count this and hence also Suzuki's work \cite{Suzuki21} as a sort of classification theory, \autoref{intro:main-result} is certainly the first result of its kind covering actions of arbitrary locally compact groups.

In order to come full circle with the analogy regarding the known Kirch\-berg--Phillips theorem, we remark that the results achieved in this article ought to be viewed as the completion of the \emph{analytical} work needed to arrive at a fully satisfactory classification theory.
In analogy to Rosenberg--Schochet's work on the UCT \cite{RosenbergSchochet87}, it is still a largely unsolved and independent problem to determine the existence of a $KK$-equivalence between two actions in the \emph{equivariant bootstrap class} \cite{DAEmersonMeyer14} in terms of isomorphism of more manageable local invariants.
It should be noted that Rosenberg--Schochet recognized this challenge early on and developed methods to tackle this problem for certain compact connected Lie groups \cite{RosenbergSchochet86}, but there was a subsequent long period without any follow-up.
In our opinion, the ideas in Köhler's PhD thesis \cite{KoehlerPHD} related to the equivariant UCT seem promising to build on, but this needs to involve methods related to algebraic topology and homological algebra. 
Therefore we support the viewpoint propagated by Meyer \cite{Meyer21} to view that as the \emph{algebraic} side of the classification problem.

The article is organized as follows.
Conceptually speaking, the content is deliberately ordered so that the level of generality of the involved techniques decreases with each section.
The introduction of basic concepts is done in the first section, covering the cocycle category framework, equivariant $KK$-theory, sequence algebras, and amenability for actions.
In the second section, we study the concept of approximate domination between cocycle representations, which is a stronger version of weak containment.
In particular we obtain a sufficient condition for a proper cocycle morphism to absorb another one in the sense of Cuntz addition up to asymptotic unitary equivalence, which is inspired by an important technical step in the known Kirchberg--Phillips theorem; see \cite[Lemma 2.3.6]{Phillips00}.

In the third section we introduce and study isometrically shift-absorbing actions, with a focus on those that are also amenable.
The key lemma of this section (\autoref{lem:key-lemma}) serves as the technical backbone of our theory and represents a reduction principle: If one considers a pair of cocycle representations $(\phi,\Iu), (\psi,\Iv): (A,\alpha)\to (B,\beta)$ with $\beta$ being amenable and isometrically shift-absorbing, then $(\phi,\Iu)$ approximately dominates $(\psi,\Iv)$ if and only if $\phi$ approximately dominates $\psi$ as an ordinary $*$-homomorphism.
In other words, the dynamical assumptions of the codomain action allow one to solve an a priori difficult dynamical problem by solving a much easier and more familiar \cstar-algebra problem one encounters in the literature.
An important consequence of this reduction principle is that it allows us to obtain explicit descriptions of examples of absorbing cocycle representations, which are highly relevant in light of the stable uniqueness theorem for equivariant $KK$-theory.
We note here that in various places in both the second and third section, many of the most non-trivial technical ingredients are imported as consequences of observations made in the context of studying general $KK^G$-groups in our earlier article \cite{GabeSzabo22}.

In the fourth section we prove a dynamical version of the famous $\CO_2$-embedding theorem \cite{KirchbergPhillips00}, which is a technical centerpiece towards the `onto' part of \autoref{intro:existence-uniqueness}.
Since one may see this as a result of independent interest generalizing also significant parts of \cite{Szabo18kp, Suzuki21}, let us state it here; see also \autoref{thm:equi-O2-embedding}.

\begin{theoremi} \label{thmi:O2-embedding}
Let $\alpha: G\curvearrowright A$ be an amenable action on a separable exact \cstar-algebra.
Let $\beta: G\curvearrowright B$ be an isometrically shift-absorbing action on a Kirchberg algebra.
Then there exists a proper cocycle embedding $(A,\alpha)\to (B\otimes\CO_2,\beta\otimes\id_{\CO_2})$.
\end{theoremi}

The main conclusions of our work are then coming together in the fifth section.
Here we prove the aforementioned existence and uniqueness theorems, i.e., the more detailed version of \autoref{intro:existence-uniqueness}.
Conceptually speaking, this part of our approach runs in perfect parallel with techniques suggested by the first author in \cite[Section 7]{Gabe21} as a more direct way to prove the Kirchberg--Phillips theorem that also applied to the classification of non-simple \cstar-algebras absorbing the Cuntz algebra $\CO_\infty$.

In the sixth and final section, we prove our main classification result, i.e., the more detailed version of \autoref{intro:main-result}.
In the rest of the section, we deduce a number of consequences of our main result, including:
\begin{enumerate}[leftmargin=*,label=\textup{\textbf{(\alph*)}}]
\item A refined version of our main result classifying actions of compact groups up to genuine conjugacy; see \autoref{cor:main-compact}.
\item The positive solution to a conjecture of Izumi \cite{Izumi10}, a special case of which was recently the focus of \cite{IzumiMatui21, IzumiMatui20, IzumiMatui21_2}; see \autoref{thm:Izumi-conjecture}.
\item An alternative proof of \cite[Theorem A]{Szabo21R}, asserting that every Kirchberg algebra admits a unique Rokhlin flow \cite{Kishimoto96_R} up to cocycle conjugacy, which started as a conjecture of Kishimoto.
We can in fact obtain the same rigidity theorem for actions of $\IR^k$ for all $k\geq 1$; see \autoref{cor:Rokhlin-multiflows}.
\item If $G$ is a discrete amenable group, then all faithful quasi-free actions of $G$ on $\CO_\infty$ are mutually cocycle conjugate; see \autoref{thm:Izumi-other-conjecture}.
This generalizes the main result of \cite{GoldsteinIzumi11} and confirms another conjecture of Izumi.
\item For all exact groups with the Haagerup property, we prove the existence of an amenable $G$-action on $\CO_\infty$ which is canonically $KK^G$-equivalent to $\IC$; see \autoref{thm:exact-HP}.
Aside from amenable groups, the existence of such actions has only been known for discrete free groups due to Suzuki.
(We remark that shortly before submission of our article, Suzuki has independently constructed such examples by a different method \cite{Suzuki23}.)
\item If $G$ is a discrete exact torsion-free group with the Haagerup property and $\CD$ a strongly self-absorbing Kirchberg algebra, then there exists a unique amenable outer $G$-action on $\CD$ up to cocycle conjugacy, which in fact must be a strongly self-absorbing action \cite{Szabo18ssa}; see \autoref{cor:ssa-conj}.
This confirms and generalizes the traceless case of \cite[Conjecture A]{Szabo19ssa4}.
\end{enumerate}

%%%%%%%%%%%%%%%%%%%%%%%%%%%%%%%%%%%%%%%%%%%%%%%%%%%%%

\section{Preliminaries}

\begin{nota} \label{basic-notation}
Throughout, $G$ will denote a second-countable, locally compact group unless specified otherwise.
Normal capital letters like $A,B,C$ will denote \cstar-algebras.
The multiplier algebra of $A$ is denoted as $\CM(A)$, whereas $A^\dagger$ denotes the proper unitization of $A$, i.e., one adds a new unit even if $A$ was unital.
We sometimes denote the closed unital ball of $A$ by $A_{\leq 1}$.
We write $\CU(\eins+A)$ for the set of all unitaries in $A^\dagger$ whose scalar part is $\eins$, which can be canonically identified with the unitary group of $A$ if it was already unital.
Throughout the article, the symbol $\CK$ denotes the \cstar-algebra of compact operators on a separable infinite-dimensional Hilbert space $\CH$, and we write $\CK(\CH)$ when a specific description of $\CH$ is relevant to the matter at hand.
A \cstar-algebra $A$ is called stable when $A\cong A\otimes\CK$.
Let $e_{k,\ell}\in \CK$ for $k,\ell\geq 1$ denote a collection of generating matrix units.
Greek letters such as $\alpha,\beta,\gamma$ are used for point-norm continuous maps $G\to\Aut(A)$, most often group actions.
Depending on the situation, we may denote $\id_A$ either for the identity map on $A$ or the trivial $G$-action on $A$.
We will denote by $A^\alpha$ or $\CM(A)^\alpha$ the \cstar-subalgebra of fixed points (in $A$ or $\CM(A)$) with respect to $\alpha$.
Normal alphabetical letters such as $u,v,U,V$ are used for unitary elements in some \cstar-algebra $A$.
If either $u\in\CU(\CM(A))$ or $u\in\CU(\eins+A)$, we denote by $\ad(u)$ the induced inner automorphism of $A$ given by $a\mapsto uau^*$.
Double-struck letters such as $\Iu, \Iv, \IU, \IV$ are used for strictly continuous maps $G\to\CU(\CM(A))$, most often (1-)cocycles with respect to an action $\alpha: G\curvearrowright A$, which for the map $\Iu$ would mean that it satisfies the cocycle identity $\Iu_{gh}=\Iu_g\alpha_g(\Iu_h)$ for all $g,h\in G$.
Under this assumption, one obtains a new (cocycle perturbed) action $\alpha^\Iu: G\curvearrowright A$ via $\alpha^\Iu_g=\ad(\Iu_g)\circ\alpha_g$.

Although we will introduce equivariant $KK$-theory in this section via the Cuntz--Thomsen picture, we will implicitly assume that the reader has some existing passing familiarity with it; see for example \cite{BlaKK}.
In particular we make use of the Kasparov product and freely use its known properties. We apply the common tensor product notation $x\otimes y\in KK^G(\alpha,\gamma)$ for the product of two elements $x\in KK^G(\alpha,\beta)$ and $y\in KK^G(\beta,\gamma)$.
\end{nota}

\begin{defi}[see {\cite[Section 1]{Szabo21cc}}] \label{def:cocycle-morphism}
Let $\alpha: G\curvearrowright A$ and $\beta: G\curvearrowright B$ be two actions on \cstar-algebras.
\begin{enumerate}[label=\textup{(\roman*)}]
\item A \emph{cocycle representation} $(\phi,\Iu): (A,\alpha)\to (\CM(B),\beta)$ consists of a $*$-homomorphism $\phi: A\to\CM(B)$ and a strictly continuous $\beta$-cocycle $\Iu: G\to\CU(\CM(B))$ satisfying $\ad(\Iu_g)\circ\beta_g\circ\phi=\phi\circ\alpha_g$ for all $g\in G$.
\item If additionally $\phi(A)\subseteq B$, then the pair $(\phi,\Iu)$ is called a \emph{cocycle morphism}, and we denote $(\phi,\Iu): (A,\alpha)\to (B,\beta)$.
\item If $\phi(A)\subseteq B$ and furthermore $\Iu$ takes values in $\CU(\eins+B)$, then the pair $(\phi,\Iu)$ is called a \emph{proper cocycle morphism}.
Note that by \cite[Proposition 6.9(ii)]{Szabo21cc}, this assumption implies that $\Iu$ is automatically a norm-continuous map.
\end{enumerate}
We will later use the convention that a \emph{(proper) cocycle embedding} is a (proper) cocycle morphism $(\phi,\Iu)$ with the property that $\phi$ is injective.
Moreover a \emph{(proper) cocycle conjugacy} is a (proper) cocycle morphism $(\phi,\Iu)$ with the property that $\phi$ is an isomorphism. We write $\cc$ for the relation of cocycle conjugacy. 
\end{defi}

\begin{nota} \label{nota:strongly-stable}
We say that an action $\beta: G\curvearrowright B$ is \emph{strongly stable} if $(B,\beta)$ is (genuinely) conjugate to $(B\otimes\CK,\beta\otimes\id_\CK)$.
\end{nota}

As is easily observed (see \cite[Remark 1.4]{GabeSzabo22}), an action $\beta: G\curvearrowright B$ is strongly stable if and only if there is a sequence of isometries $r_n\in\CM(B)^\beta$ such that $\eins=\sum_{n=1}^\infty r_nr_n^*$ holds in the strict topology.
After exploring the literature it would appear that, at least to the best of our knowledge, the following simple observation is so far unknown.\footnote{There seems to be evidence suggesting that the analogous statement in the setting of von Neumann algebras is known; see for example the proof of \cite[Theorem 6.1(2)]{MasudaTomatsu16}.
Our method of proof would appear to suitably translate back to that context to give an elegant approach to that analogous result.}
We note that this is the one and only spot in the entire article that refers to twisted $G$-actions instead of genuine actions, so we deliberately stated \autoref{def:cocycle-morphism} above only for genuine actions.
Before delving into the proof below the reader may hence either choose to check out \cite[Definitions 1.1 and 1.10]{Szabo21cc} or pretend like one has $\Fw=\eins$ below.

\begin{prop} \label{prop:cc-stability}
Let $B$ be a stable \cstar-algebra.
Then every twisted action $(\beta,\Fw): G\curvearrowright B$ is cocycle conjugate to $(\beta\otimes\id_\CK,\Fw\otimes\eins): G\curvearrowright B\otimes\CK$. In particular, every action on a stable \cstar-algebra is cocycle conjugate to a strongly stable action.
\end{prop}
\begin{proof}
Since $B$ is stable, we may find a sequence of isometries $r_n\in\CM(B)$ such that $\sum_{n=1}^\infty r_nr_n^*=\eins$ holds in the strict topology.
Let $\set{e_{j,k}}_{j,k\geq 1}$ be a set of matrix units generating the compacts $\CK$.
We consider the isomorphism
\[
\Lambda: B\otimes\CK\to B,\quad \Lambda(b\otimes e_{j,k})=r_j b r_k^*.
\]
We moreover consider the strictly continuous map 
\[
\IU: G\to\CU(\CM(B)),\quad \IU_g = \sum_{n=1}^\infty r_n\beta_g(r_n)^*.
\]
We claim that $(\Lambda,\IU)$ is a cocycle conjugacy $(B\otimes\CK,\beta\otimes\id_\CK,\Fw\otimes\eins)\cc (B,\beta,\Fw)$.
Since $\Lambda$ is an isomorphism, we only need to verify the equivariance condition and the cocycle condition.
The first follows because we can compute for all $j,k\geq 1$, $b\in B$ and $g\in G$ that
\[
\ad(\IU_g)\circ\beta_g\circ\Lambda(b\otimes e_{j,k})=\IU_g \beta_g(r_j b r_k^*) \IU_g^* = r_j\beta_g(b)r_k^* = \Lambda\circ(\beta_g\otimes\id_\CK)(b\otimes e_{j,k}).
\]
The cocycle condition follows as we compute for all $g,h\in G$ that
\[
\begin{array}{ccl}
\IU_g \beta_g(\IU_h) \Fw_{g,h} \IU_{gh}^* &=& \dst \sum_{n,m=1}^\infty r_n\beta_g(r_n)^* \beta_g(r_m) (\beta_g\circ\beta_h)(r_m)^* \Fw_{g,h} \IU_{gh}^* \\
&=& \dst \sum_{n=1}^\infty r_n \Fw_{g,h} \beta_{gh}(r_n)^* \IU_{gh}^* \\
&=& \dst \sum_{n=1}^\infty r_n \Fw_{g,h} r_n^* \\
&=& \dst \Lambda(\Fw_{g,h}\otimes\eins_{\CM(\CK)}).
\end{array}
\]
This finishes our proof.
\end{proof}

\begin{rem}
By the Packer--Raeburn stabilization technique \cite[Section 3]{PackerRaeburn89}, it follows as a consequence of the above that when $B$ is a stable \cstar-algebra, then every twisted action on $B$ is exterior equivalent to a genuine action.
As far as we are aware, this has not been observed before.
\end{rem}

We shall now recall some necessary background on equivariant $KK$-theory.
Throughout the paper the main focus lies on the Cuntz--Thomsen picture \cite{Cuntz83, Cuntz84, Higson87, Thomsen98}.

\begin{defi}[cf.\ {\cite[Section 3]{Thomsen98}}] \label{def:equi-Cuntz-pair}
Let $\alpha: G\curvearrowright A$ and $\beta: G\curvearrowright B$ be two actions on \cstar-algebras such that $A$ is separable and $B$ is $\sigma$-unital.
An \emph{$(\alpha,\beta)$-Cuntz pair} is a pair of cocycle representations
\[
(\phi,\Iu), (\psi,\Iv): (A,\alpha) \to (\CM(B\otimes\CK),\beta\otimes\id_\CK),
\]
such that the pointwise differences $\phi-\psi$ and $\Iu-\Iv$ take values in $B\otimes\CK$.\footnote{In Thomsen's article it was also assumed that the map $\Iu-\Iv$ is norm-continuous. This turns out to be redundant, see \cite[Proposition 6.9]{Szabo21cc}.}
Whenever $\beta$ is assumed to be strongly stable, we also allow $(B,\beta)$ in place of $(B\otimes\CK,\beta\otimes\id_\CK)$ appearing in the definition of an $(\alpha,\beta)$-Cuntz pair.
\end{defi}

\begin{defi}[cf.\ {\cite[Lemma 3.4]{Thomsen98}}]
Let $\beta: G\curvearrowright B$ be an action on a \cstar-algebra.
Suppose that there exists a unital inclusion $\CO_2\subseteq\CM(B)^\beta$.
For two isometries $t_1,t_2\in\CM(B)^\beta$ with $t_1t_1^*+t_2t_2^*=\eins$, we may consider the $\beta$-equivariant $*$-homomorphism
\[
\CM(B)\oplus\CM(B)\to\CM(B),\quad b_1\oplus b_2\mapsto b_1\oplus_{t_1,t_2} b_2 := t_1b_1t_1^*+t_2b_2t_2^*.
\]
Up to unitary equivalence with a unitary in $\CM(B)^\beta$, this $*$-homomorphism does not depend on the choice of $t_1$ and $t_2$.\footnote{If $v_1, v_2\in\CM(B)^\beta$ are two other isometries with $v_1v_1^*+v_2v_2^*=\eins$, then the unitary equivalence between ``$\oplus_{t_1,t_2}$'' and ``$\oplus_{v_1,v_2}$'' is implemented by $w=t_1v_1^*+t_2v_2^*\in\CM(B)^\beta$.}
One refers to the element $b_1\oplus_{t_1,t_2} b_2$ as the \emph{Cuntz sum} of the two elements $b_1$ and $b_2$ (with respect to $t_1$ and $t_2$).
Now let $\alpha: G\curvearrowright A$ be another action on a \cstar-algebra, and $(\phi,\Iu), (\psi,\Iv): (A,\alpha)\to (\CM(B),\beta)$ two cocycle representations.
We likewise define the (pointwise) Cuntz sum
\[
(\phi,\Iu)\oplus_{t_1,t_2} (\psi,\Iv) = (\phi\oplus_{t_1,t_2}\psi, \Iu\oplus_{t_1,t_2}\Iv): (A,\alpha)\to (\CM(B),\beta),
\]
which is easily seen to be another cocycle representation.
Since its unitary equivalence class does not depend on the choice of $t_1$ and $t_2$, we will often omit $t_1$ and $t_2$ from the notation if it is clear from context that a given statement is invariant under said equivalence.
\end{defi}

\begin{nota}
Given a \cstar-algebra $B$, we denote $B[0,1]=\CC[0,1]\otimes B$.
If one has an action $\beta: G\curvearrowright B$, we consider the obvious $G$-action on $B[0,1]$ given by $\beta[0,1]=\id_{\CC[0,1]}\otimes\beta$.
\end{nota}

\begin{defi}[see {\cite[Section 3]{Thomsen98} and \cite[Section 1]{GabeSzabo22}}] \label{def:KKG-Thomsen}
Let $A$ be a separable \cstar-algebra and $B$ a $\sigma$-unital \cstar-algebra.
For two actions $\alpha: G\curvearrowright A$ and $\beta: G\curvearrowright B$,
let $\IE^G(\alpha,\beta)$ denote the set of all $(\alpha,\beta)$-Cuntz pairs.

Two elements $\big( (\phi^0,\Iu^0), (\psi^0,\Iv^0) \big)$ and $\big( (\phi^1,\Iu^1), (\psi^1,\Iv^1) \big)$ in $\IE^G(\alpha,\beta)$ are called \emph{homotopic}, abbreviated $\big( (\phi^0,\Iu^0), (\psi^0,\Iv^0) \big)\sim_h \big( (\phi^1,\Iu^1), (\psi^1,\Iv^1) \big)$, if there exists an $(\alpha,\beta[0,1])$-Cuntz pair that restricts to $\big( (\phi^0,\Iu^0), (\psi^0,\Iv^0) \big)$ upon evaluation at $0\in [0,1]$, and restricts to $\big( (\phi^1,\Iu^1), (\psi^1,\Iv^1) \big)$ upon evaluation at $1\in [0,1]$.
An $(\alpha,\beta)$-Cuntz pair of the form $\big( (\phi,\Iu), (\psi,\Iv) \big)$ with $\phi=\psi=0$ is called a \emph{cocycle pair} and is denoted by $(\Iu,\Iv)$ with slight abuse of notation.
We define $\IE_0^G(\alpha,\beta)$ as the set of all \emph{anchored} $(\alpha,\beta)$-Cuntz pairs, i.e., those $\big( (\phi,\Iu), (\psi,\Iv) \big)\in\IE^G(\alpha,\beta)$ such that $(\Iu,\Iv)\sim_h (\eins,\eins)$.

For any unital inclusion $\CO_2\subseteq\CM(B\otimes\CK)^{\beta\otimes\id_\CK}$ with generating isometries $t_1,t_2$, one can perform the Cuntz addition for two $(\alpha,\beta)$-Cuntz pairs as
\[
\begin{array}{cl}
\multicolumn{2}{l}{
\big( (\phi^0,\Iu^0), (\psi^0,\Iv^0) \big) \oplus_{t_1,t_2} \big( (\phi^1,\Iu^1), (\psi^1,\Iv^1) \big) } \\
=& \big( (\phi^0,\Iu^0)\oplus_{t_1,t_2}(\phi^1,\Iu^1), (\psi^0,\Iv^0)\oplus_{t_1,t_2} (\psi^1,\Iv^1) \big).
\end{array}
\]
This Cuntz pair is independent of the choice of $t_1,t_2$ up to homotopy; see \cite[Lemma 3.4]{Thomsen00}.
\end{defi}

\begin{rem}[see {\cite[Proposition 1.12]{GabeSzabo22}}] \label{def:KKG}
The quotient $\IE^G(\alpha,\beta)/{\sim_h}$ becomes an abelian group with Cuntz addition.
The homotopy classes of cocycle pairs form a subgroup $H_\beta$.
It was proved by Thomsen in \cite[Theorem 3.5]{Thomsen98} that the group quotient of $\IE^G(\alpha,\beta)/{\sim_h}$ modulo $H_\beta$ is naturally isomorphic to $KK^G(\alpha,\beta)$.
For an $(\alpha,\beta)$-Cuntz pair consisting of $(\phi,\Iu)$ and $(\psi,\Iv)$, we denote its associated equivalence class in $KK^G(\alpha,\beta)$ by $[(\phi,\Iu),(\psi,\Iv)]$.
Under this identification, one has that the inclusion map $\IE^G_0(\alpha,\beta)\subseteq\IE^G(\alpha,\beta)$ also induces a natural isomorphism of abelian groups $\IE^G_0(\alpha,\beta)/{\sim_h}\cong KK^G(\alpha,\beta)$.
In other words, $KK^G(\alpha,\beta)$ may be defined as the abelian group of homotopy classes of anchored $(\alpha,\beta)$-Cuntz pairs.
\end{rem}

\begin{rem}
Thomsen explicitly constructs the aforementioned isomorphism (described just before \cite[Theorem 3.5]{Thomsen98}). 
In \cite[Section 6]{Szabo21cc} this description was revisited in detail and used to show that under certain conditions, compositions of cocycle morphisms are compatible with the Kasparov product.
In this paper we need that this holds for arbitrary compositions of proper cocycle morphisms, so let us briefly recall again the functoriality of $KK^G$ both in the Kasparov picture and the Cuntz--Thomsen picture.
Given an action $\beta: G\curvearrowright B$ and a $\beta$-cocycle $\Iu: G\to\CU(\CM(B))$, we denote by $B^\Iu$ the Hilbert $(B,\beta)$-module that is equal to $B$ as an ordinary Hilbert right-$B$-module, but is equipped with the continuous linear $G$-action given by $g\cdot b := \Iu_g\beta_g(b)$ for all $g\in G$ and $b\in B$.
For any cocycle morphism $(\phi,\Iu): (A,\alpha)\to (B,\beta)$, one has that $(B^\Iu,\phi,0)$ is a Kasparov triple, and one defines $KK^G(\phi,\Iu)\in KK^G(\alpha,\beta)$ as its equivalence class; see \cite[Definition 6.3]{Szabo21cc}.
Under Thomsen's identification above, this element is the class associated to the $(\alpha,\beta)$-Cuntz pair 
\[
\big( (\phi\otimes e_{1,1},\Iu\otimes e_{1,1}+(\eins-e_{1,1})), (0,\Iu\otimes e_{1,1}+(\eins-e_{1,1}) \big).\]
If $(\phi,\Iu)$ happens to be proper, then the second cocycle above can also be chosen to be trivial; see \cite[Proposition 6.14]{Szabo21cc}. Note that these two different ways of associating a Cuntz pair to a proper cocycle morphism might give different classes in $\IE^G(\alpha, \beta)/{\sim_h}$, but they agree in $KK^G(\alpha, \beta)$.
\end{rem}

\begin{prop} \label{prop:KKG-simplified}
Let $(\phi, \Iu) : (A, \alpha) \to (B, \beta)$ be a cocycle morphism.
Then $KK^G(\phi,\Iu)$ is represented by the Kasparov triple $( \overline{\phi(A) B^\Iu},\phi,0)$.
\end{prop}
\begin{proof}
By the above remark, it suffices to show that $( \overline{\phi(A) B^\Iu},\phi,0)$ and $(B^\Iu, \phi, 0)$ are homotopic in the sense of \cite[Definition 17.2.2]{BlaKK}.
Let $D = \{ f\in C([0,1], B) : f(0) \in \overline{\phi(A)B} \}$ be the equivariant Hilbert $(B[0,1],\beta[0,1])$-module with $G$-action $(g\cdot f)(t) = \Iu_g \beta_g(f(t))$.
Let $\Phi : A \to \mathcal \IK(D)$ be the representation where $\Phi(a)$ is given by pointwise multiplication from the left by $\phi(a)$.
Then $(D, \Phi, 0)$ is the desired homotopy.
\end{proof}

\begin{cor}[cf.\ {\cite[Proposition 6.5]{Szabo21cc}}]
Let $\alpha: G\curvearrowright A$, $\beta: G\curvearrowright B$ and $\gamma: G\curvearrowright C$ be actions on separable \cstar-algebras.
Let $(\phi, \Iu) : (A, \alpha) \to (B, \beta)$ and $(\psi, \Iv) : (B, \beta) \to (C, \gamma)$ be proper cocycle morphisms.
Then
\[
KK^G(\phi,\Iu)\otimes KK^G(\psi,\Iv) = KK^G\big( (\psi,\Iv)\circ(\phi,\Iu) \big) \ \in \ KK^G(\alpha,\gamma).
\]
\end{cor}
\begin{proof}
Let $\Iw = \phi^\dagger(\Iu) \Iv: G\to\CU(\eins+C)$ be the induced $\gamma$-cocycle so that $(\psi , \Iv) \circ (\phi , \Iu) = (\psi\circ \phi, \Iw)$.
Since the $KK$-elements are in this instance represented by Kasparov triples of a specific form, their Kasparov product can be described as a balanced tensor product as in \cite[Example 18.4.2]{BlaKK}.
In light of \autoref{prop:KKG-simplified}, the Kasparov product $KK^G(\phi,\Iu)\otimes KK^G(\psi,\Iv)$ is represented by the triple $(E,\kappa,0)$, where $E=\overline{\phi(A)B^\Iu}\otimes_{\psi}\overline{\psi(B)C^\Iv}$ is equipped with the tensor product $G$-action, and $\kappa=\phi\otimes\eins$.
We have $E\cong\overline{(\psi\circ\phi)(A)C^{\Iw}}$ via $b\otimes c\mapsto \psi(b)c$ as Hilbert right-$C$-modules, and under this identification the map $\kappa$ corresponds to $\psi\circ\phi$.
This map is actually equivariant because an element of the form
\[
E\ \ni \ g\cdot (b\otimes c)=\Iu_g\beta_g(b)\otimes \Iv_g\gamma_g(c)
\]
is mapped to the product
\[
\psi(\Iu_g\beta_g(b))\Iv_g\gamma_g(c)=\psi^\dagger(\Iu_g) \underbrace{ \psi(\beta_g(b)) \Iv_g }_{=\Iv_g\gamma_g(\psi(b))} \gamma_g(c) = \Iw_g\gamma_g(\psi(b)c) = g\cdot\big( \psi(b)c \big).
\]
So appealing to \autoref{prop:KKG-simplified} again, we have found an identification with a Kasparov triple representing the composition $(\psi,\Iv)\circ(\phi,\Iu)$.
\end{proof}

\begin{defi}
Let $(\phi,\Iu), (\psi,\Iv): (A,\alpha) \to (B,\beta)$ be two proper cocycle morphisms.
We say that $(\phi,\Iu)$ and $(\psi,\Iv)$ are \emph{properly unitarily equivalent},
if there exists a unitary $u\in\CU(\eins+B)$ with $\psi=\ad(u)\circ\phi$ and $\Iv_g=u\Iu_g\beta_g(u)^*$ for all $g\in G$.
%
\begin{comment}
We say that they are properly approximately unitarily equivalent,
if there exists a sequence $u_n\in\CU(\eins+B)$ such that
\[
\lim_{n\to\infty} \|\psi(a)-u_n\phi(a)u_n^*\|=0
\]
for all $a\in A$ and
\[
\lim_{n\to\infty} \max_{g\in K}\ \| \Iv_g - u_n\Iu_g\beta_g(u_n)^* \| =0
\] 
for all compact sets $K\subseteq G$.
\end{comment}
%
\end{defi}

\begin{defi} \label{def:various-equivalences}
Let $(\phi,\Iu), (\psi,\Iv): (A,\alpha)\to (\CM(B),\beta)$ be two cocycle representations.
We write $(\phi,\Iu)\asymp (\psi,\Iv)$, if there exists a norm-continuous path $u: [0,\infty)\to\CU(\CM(B))$ such that
	\begin{itemize}
  	\item $\psi(a) = \dst\lim_{t\to\infty} u_t\phi(a)u_t^*$ for all $a\in A$;
  	\item $\psi(a)-u_t\phi(a)u_t^* \in B$ for all $a\in A$ and $t\geq 0$;
  	\item $\dst\lim_{t\to\infty} \max_{g\in K}\ \| \Iv_g-u_t\Iu_g\beta_g(u_t)^* \| =0$ for all compact sets $K\subseteq G$;
  	\item $\Iv_g-u_t\Iu_g\beta_g(u_t)^* \in B$ for all $t\geq 0$ and $g\in G$.
  \end{itemize} 
If it is possible to choose $u$ to have its range in $\CU(\eins+B)$, then $(\phi,\Iu)$ and $(\psi,\Iv)$ are called \emph{properly asymptotically unitarily equivalent}.
If it is additionally possible to arrange $u_0=\eins$, then we say that $(\phi,\Iu)$ and $(\psi,\Iv)$ are \emph{strongly asymptotically unitarily equivalent}.
If $B$ is unital to begin with, we usually omit the word ``properly'' above.
\end{defi}

Since we will appeal to the following technical result in various places, we shall state it here for subsequent use:

\begin{theorem}[see {\cite[Theorem 5.6]{Szabo21cc}}] \label{thm:strong-ssa-abs}
Let $\alpha: G\curvearrowright A$ be an action on a separable \cstar-algebra, and let $\delta: G\curvearrowright\CD$ be a strongly self-absorbing action\footnote{Here we use the definition given in \cite[Section 5]{Szabo21cc}.} on a (necessarily strongly self-absorbing \cite{TomsWinter07}) \cstar-algebra.
Suppose that $\alpha$ or $\delta$ is equivariantly Jiang--Su stable.\footnote{In particular this is the case if $\alpha\cc\alpha\otimes\id_{\CO_\infty}$ or $\delta\cc\delta\otimes\id_{\CO_\infty}$.}
Then $\alpha\cc\alpha\otimes\delta$ if and only if the equivariant first-factor embedding
\[
\id_A\otimes\eins_\CD: (A,\alpha)\to (A\otimes\CD,\alpha\otimes\delta)
\]
is strongly asymptotically unitarily equivalent to a proper cocycle conjugacy.
\end{theorem}

\begin{defi}
Suppose that there exists a unital inclusion $\CO_2\subseteq\CM(B)^\beta$.
Let $(\phi,\Iu), (\psi,\Iv): (A,\alpha)\to (\CM(B),\beta)$ be two cocycle representations.
We say that $(\phi,\Iu)$ \emph{absorbs} $(\psi,\Iv)$ if $(\phi\oplus\psi,\Iu\oplus\Iv)\asymp (\phi,\Iu)$.
A cocycle representation is called \emph{absorbing}, if it absorbs every cocycle representation.
\end{defi}

We will in some instances make use of the following useful fact due to Kasparov concerning quasicentral approximate units coming from ideals invariant under a group action.

\begin{lemma}[see {\cite[Lemma 1.4]{Kasparov88}} and its proof] \label{lem:Kasparov}
Let $\beta: G\curvearrowright B$ be an action on a $\sigma$-unital \cstar-algebra and let $e_n\in B$ be any countable increasing approximate unit.
Then for any separable \cstar-subalgebra $D\subseteq\CM(B)$, there exists a countable increasing approximate unit of positive contractions $h_n\in B$ belonging to the convex hull of $\{e_n\mid n\geq 1\}$
satisfying 
\[
\lim_{n\to\infty} \|[h_n,d]\|=0 \quad\text{and}\quad \lim_{n\to\infty} \max_{g\in K} \| h_n-\beta_g(h_n) \| = 0
\] 
for all $d\in D$ and compact sets $K\subseteq G$.
\end{lemma}

Similarly to how we formed Cuntz sums of two elements earlier, we may also form countably infinite sums by a similar method if the underlying action is strongly stable.

\begin{defi}
Suppose that $\beta$ is strongly stable.
Let $t_n\in \CM(B)^\beta$ be any sequence of isometries such that $\sum_{n=1}^\infty t_nt_n^* = \eins$ in the strict topology.
Then we have a $\beta$-equivariant $*$-homomorphism
\[
\ell^\infty(\IN, \CM(B)) \to \CM(B),\quad (b_n)_{n\geq 1} \mapsto \sum_{n=1}^\infty t_nb_nt_n^*,
\]
which does not depend on the choice of $t_n$ up to unitary equivalence with a unitary in $\CM(B)^\beta$.\footnote{Similarly as before, if $v_n\in\CM(B)^\beta$ is another sequence of isometries satisfying the same relation, then the unitary $w=\sum_{n=1}^\infty t_nv_n^*$ implements this equivalence.}
For any sequence of cocycle representations $(\phi^{(n)},\Iu^{(n)}): (A,\alpha)\to (\CM(B),\beta)$, we may hence define the countable sum
\[
(\Phi,\IU)=\bigoplus_{n=1}^\infty (\phi^{(n)},\Iu^{(n)}): (A,\alpha)\to (\CM(B),\beta)
\]
via the pointwise strict limits
\[
\Phi(a)=\sum_{n=1}^\infty t_n\phi^{(n)}(a)t_n^*,\quad \IU_g = \sum_{n=1}^\infty t_n\Iu^{(n)}_g t_n^*.
\]
Up to equivalence with a unitary in $\CM(B)^\beta$, this cocycle representation does not depend on the choice of $(t_n)_n$.
In particular, in the special case that $(\phi^{(n)},\Iu^{(n)})=(\phi,\Iu)$ for all $n$, we denote the resulting countable sum by $(\phi^\infty, \Iu^\infty)$ and call it the \emph{infinite repeat} of $(\phi,\Iu)$.
\end{defi}

\begin{defi} \label{def:sequence-algebras}
Let $\beta: G\curvearrowright B$ be an action on a \cstar-algebra.
We denote by $\ell^\infty_\beta(\IN,B)$ the \cstar-algebra of those $B$-valued bounded sequences $(b_n)$ such that the map $[g\mapsto (\beta_g(b_n))_{n}]$ is continuous, and consider the \emph{($\beta$-continuous) sequence algebra} 
\[
B_{\infty,\beta}=\ell^\infty_\beta(\IN,B)/c_0(\IN,B),
\] 
which carries an induced continuous action $\beta_\infty: G\curvearrowright B$ given by pointwise application of $\beta$.
Noting that $B$ canonically embeds as (equivalence classes of) constant sequences, the \emph{($\beta$-continuous) central sequence algebra} is
\[
F_{\infty,\beta}(B)=(B_{\infty,\beta}\cap B')/(B_{\infty,\beta}\cap B^\perp),
\]
which carries an induced continuous action $\tilde{\beta}_\infty: G\curvearrowright F_{\infty,\beta}(B)$.
Note that if $B$ is $\sigma$-unital, then $F_{\infty,\beta}(B)$ is unital, and the unit is represented by any sequential approximate unit of $B$.

The better-known object $B_\infty$ is recovered upon choosing $\beta=\id$.
Note that any action $\beta$ as above induces an algebraic $G$-action on $B_\infty$, the equicontinuous elements of which can be identified with the elements in $B_{\infty,\beta}$; see \cite[Theorem 2]{Brown00}.
\end{defi}

\begin{lemma} \label{lem:local-stability}
Let $\beta: G\curvearrowright B$ be a strongly stable action on a \cstar-algebra.
Then for every $\sigma$-unital $\beta_\infty$-invariant \cstar-subalgebra $D\subseteq B_{\infty,\beta}$, there exists an equivariant $*$-homomorphism $\iota: (D\otimes\CK,\beta_\infty\otimes\id_\CK)\to(B_{\infty,\beta},\beta_\infty)$ such that $\iota(d\otimes e_{11})=d$ for all $d\in D$.
\end{lemma}
\begin{proof}
Let $r_n\in\CM(B)^\beta$ be a sequence of isometries such that $\eins=\sum_{n=1}^\infty r_nr_n^*$.
For each $k\geq 0$, we define another such sequence via 
\[
r_1^{(k)}=\sum_{j=1}^{k} r_jr_j^*+ r_{k+1}\sum_{\ell=1}^{\infty} r_\ell r_{\ell+k}^*,\quad r_n^{(k)}=r_{n+k},\quad n\geq 2.
\]
Then we have $r_1^{(k)}\to\eins$ in the strict topology.
By using that $D$ has a strictly positive element, we may find an increasing sequence of numbers $\ell_k$ such that if we consider $s_n=[(r_n^{(\ell_k)})_k]\in (\CM(B)^\beta)_\infty$, we have $s_1d=ds_1=d$ for all $d\in D$.
Using this fact, we can proceed as in the proof of \autoref{prop:cc-stability} and define $\iota: D\otimes\CK\to B_{\infty,\beta}$ via $\iota(d\otimes e_{k,\ell})=s_kds_\ell^*$, which satisfies the desired property.
\end{proof}

We end this preliminary section with a brief discussion of amenability for actions.
For actions on von Neumann algebras this concept has long been around due to work of Anantharaman-Delaroche \cite{Delaroche79}, which provided a rather straightforward candidate of amenability for \cstar-dynamics over discrete groups.
After Suzuki recently demonstrated that this theory applies to interesting examples on simple \cstar-algebras \cite{Suzuki19} (contrary to the case of factors), some attention was dedicated to flesh out the correct concept for general \cstar-dynamics by Buss--Echterhoff--Willett \cite{BussEchterhoffWillett20, BussEchterhoffWillett23}.
The theory was subsequently enriched further by work of Suzuki \cite{Suzuki21, Suzuki23} and Ozawa--Suzuki \cite{OzawaSuzuki21}.
By now it is clear that amenability can be defined in various equivalent ways, and we use the version most useful to us, which is also known as the quasicentral approximation property.

\begin{defi}[cf.\ {\cite[Definition 2.11]{OzawaSuzuki21}}]
Let $\alpha: G\curvearrowright A$ be an action on a \cstar-algebra.
Let us consider $\CC_c(G,A)$ equipped with the action $\bar{\alpha}: G\curvearrowright\CC_c(G,A)$ via $\bar{\alpha}_g(f)(h)=\alpha_g(f(g^{-1}h))$.
Given a Haar measure $\mu$ on $G$, let us equip $\CC_c(G,A)$ with the $A$-valued inner product given by $\langle f\mid g\rangle = \int_G f^*g~d\mu$.
The resulting norm on $\CC_c(G,A)$ shall be denoted as $\|\cdot\|_2$.
We say that $\alpha$ is \emph{amenable}, if there exists a net $\zeta_i\in\CC_c(G,A)$ such that
\[
\|\zeta_i\|_2\leq 1,\quad \langle\zeta_i\mid\zeta_i\rangle a \to a,\quad \|a\zeta_i-\zeta_i a\|_2\to 0,\quad \max_{g\in K} \|(\zeta_i-\bar{\alpha}_g(\zeta_i))a\|_2 \to 0
\]
for all $a\in A$ and compact sets $K\subseteq G$.
\end{defi}

\begin{rem}
In the context of working with nets such as above coming from amenability of a given action, we will subsequently use without further mention the following simple observation.
The two properties
\[
\eins \geq \langle\zeta_i\mid\zeta_i\rangle \stackrel{\text{\tiny strictly}}{\longrightarrow} \eins \quad\text{and}\quad \|a\zeta_i-\zeta_i a\|_2\to 0
\]
imply in conjunction that $\langle\zeta_i\mid a\zeta_i\rangle\to a$ for all $a\in A$.
Furthermore we can also conclude $\langle \zeta_i\mid d\zeta_i\rangle \stackrel{\text{\tiny strictly}}{\longrightarrow} d$ for all $d\in\CM(A)$, as follows from considering that for all $a\in A$ and large enough $i$, one has
\[
\langle\zeta_i\mid d\zeta_i\rangle a = \langle\zeta_i\mid d\zeta_i a\rangle \approx \langle\zeta_i\mid da\zeta_i\rangle \approx da
\]
and analogously $a\langle\zeta_i\mid d\zeta_i\rangle \approx ad$.
Finally, if $\Iu$ is an $\alpha$-cocycle, it is easy to check $\max_{g\in K} \| (\zeta_i - \overline \alpha_g^{\Iu}(\zeta_i)) a\|_2 \to 0$ for all $a\in A$ and compact sets $K\subseteq G$, and thus the net $\zeta_i$ also witnesses amenability of the perturbed action $\alpha^\Iu$. In particular, amenability is preserved under cocycle conjugacy.
\end{rem}

\begin{rem}
We will use several times without reference that for any $f,g\in \CC_c(G, A)$ and every norm-bounded continuous function $h : G \to A$, one has $\| \langle f, hg\rangle\| \leq \sup_{t\in \overline{\supp} f \cap \overline{\supp} g}\| h(t)\| \| f\|_2 \| g\|_2$ by the Cauchy--Schwarz inequality.
Here the product $hg$ is to be understood as pointwise.   
\end{rem}

%%%%%%%%%%%%%%%%%%%%%%%%%%%%%%%%%%%%%%%%%%%%

\section{Approximate 1-domination of cocycle representations}

Going forward, we let $A$ be a separable \cstar-algebra and $B$ a $\sigma$-unital \cstar-algebra unless specified otherwise.

\begin{defi}[cf.\ {\cite[Definition 3.3]{GabeSzabo22}}] \label{def:wc}
Let 
\[
(\phi,\Iu), (\psi,\Iv): (A,\alpha) \to (\CM(B),\beta)
\] 
be two cocycle representations.
We say that $(\psi,\Iv)$ is \emph{weakly contained} in $(\phi,\Iu)$, if the following is true.

For all compact sets $K\subseteq G$, $\CF\subset A$, every $\eps> 0$, and every contraction $b\in B$ there exists a collection of elements $\{c_{k} \mid k=1,\dots,N\}\subset B$ satisfying
\begin{equation} \label{eq:wc:1}
\max_{g\in K} \Big\| b^*\Iv_g\beta_g(b)-\sum_{k=1}^{N} c_{k}^*\Iu_g\beta_g(c_{k}) \Big\| \leq \eps
\end{equation}
and
\begin{equation} \label{eq:wc:2}
\max_{a\in\CF} \Big\| b^*\psi(a)b - \sum_{k=1}^N c_{k}^*\phi(a)c_{k} \Big\| \leq \eps.
\end{equation}
If it is always possible to choose $N=1$, then we say that $(\phi,\Iu)$ \emph{approximately 1-dominates} $(\psi,\Iv)$.\footnote{We may somtimes omit the ``1-'' for convenience, since this causes no notational conflicts within this article.
We note however, that ``approximately dominates'' may in general conflict with definitions in other sources.}
If $(\phi,\Iu)$ and $(\psi,\Iv)$ weakly contain each other, we say that they are \emph{weakly equivalent}.
\end{defi}

\begin{rem} \label{rem:ordinary-dom}
For an application later, we record a known result in the context of the above definition when $G=\set{1}$:
Suppose $A$ and $B$ are separable \cstar-algebras with $B\cong B\otimes\CO_\infty\otimes\CK$, and $\phi,\psi: A\to \CM(B)$ are two $*$-homomorphisms with $\phi(A)\subseteq B$.
If $\phi$ is full and $\psi$ is weakly nuclear\footnote{This means that for all contractions $b\in B$, the c.c.p.\ map $b^*\psi(\cdot)b$ is nuclear.}, then for every contraction $b\in B$ one can find $c\in B$ satisfying condition \eqref{eq:wc:2} (with $N=1$) by \cite[Theorem 7.21]{KirchbergRordam02} and \cite[Proposition 3.12]{Gabe21}.
By \cite[Lemma 7.4]{KirchbergRordam02} one gets that $(\phi, \eins)$ approximately 1-dominates $(\psi, \eins)$ in the case $G=\set{1}$. 
\end{rem}

\begin{defi}[cf.\ {\cite[Notation 2.2]{GabeSzabo22}}]
Let $\beta: G\curvearrowright B$ be an action on a \cstar-algebra.
We denote by $\CM^\beta(B)$ the \cstar-subalgebra of those elements $x\in\CM(B)$ where $\{x-\beta_g(x) \mid g\in G\}\subset B$.\footnote{This is not be confused with the algebra $\CM(B)^\beta$ of genuine fixed points; see \autoref{basic-notation}.}
The canonical extension of $\beta$ to an action on this \cstar-algebra is in fact point-norm continuous.
\end{defi}

\begin{lemma} \label{lem:absorption-step-1}
Let $\alpha: G\curvearrowright A$ and $\beta: G\curvearrowright B$ be two actions on \cstar-algebras, where $A$ is separable.
Suppose that $\beta$ is strongly stable.
Let
\[
(\phi,\Iu), (\psi,\Iv): (A,\alpha)\to (B,\beta)
\]
be proper cocycle morphisms.
Suppose there exists a sequence of contractions $s_n\in B$ such that for every $a\in A$ and compact set $K\subseteq G$, one has
\begin{enumerate}[leftmargin=*,label=\textup{(\roman*)}]
\item $\|s_n^*\phi(a)s_n - \psi(a)\| \to 0$; \label{AS-1:1}
\item $\dst \max_{g\in K} \| s_n^*\Iu_g\beta_g(s_n) - \Iv_gs_n^*s_n \| \to 0$; \label{AS-1:2}
\item $\dst \max_{g\in K} \| s_n^*s_n-\beta_g(s_n^*s_n) \| \to 0$; \label{AS-1:3}
\item $\dst \max_{g\in K} \| (\eins-s_n^*s_n)(\Iv_g-\eins) \| \to 0$. \label{AS-1:4}
\end{enumerate}
Then there exists a sequence of isometries $S_n\in\CM^\beta(B)$ such that
\[
\lim_{n\to\infty} \|\phi(a)S_n - S_n\psi(a)\|,\quad \lim_{n\to\infty}\ \max_{g\in K} \| \Iu_g\beta_g(S_n) - S_n\Iv_g\|=0
\]
for all $a\in A$ and every compact set $K\subseteq G$.
\end{lemma}
\begin{proof}
We claim that
\begin{equation} \label{eq:AS-1:1}
\|\phi(a)s_n - s_n\psi(a)\| \to 0 \quad\text{and}\quad \max_{g\in K} \| \Iu_g\beta_g(s_n) - s_n\Iv_g \| \to 0
\end{equation}
holds in place of the first two properties. The first property follows from applying \ref{AS-1:1} and \cite[Lemma 3.8]{Gabe20} to the case $D = B_\infty$ and $v = [(s_n)_{n\in \mathbb N}]$. This lemma also implies that $s_n^\ast s_n$ acts like an approximate unit on $\psi(A)$. 
For the second property, we observe for every compact set $K\subseteq G$ that
\[
\begin{array}{cl}
\multicolumn{2}{l}{ \dst\limsup_{n\to\infty} \max_{g\in K} \|s_n\Iv_g-\Iu_g\beta_{g}(s_n)\|^2 } \\
=& \dst\limsup_{n\to\infty} \max_{g\in K} \|(s_n\Iv_g-\Iu_g\beta_{g}(s_n))^*(s_n\Iv_g-\Iu_g\beta_{g}(s_n))\|  \\
=& \dst\limsup_{n\to\infty} \max_{g\in K} \|\Iv_g^*s_n^*s_n\Iv_g- \beta_{g}(s_n)^*\Iu_g^*s_n\Iv_g - \Iv_g^*s_n^*\Iu_g\beta_{g}(s_n)+\beta_{g}(s_n^*s_n)\| \\
\stackrel{\ref{AS-1:2}, \ref{AS-1:3}}{=}& \dst\limsup_{n\to\infty} \max_{g\in K} \|\Iv_g^*s_n^*s_n\Iv_g - s_n^*s_n \|  \\
 = & \dst\limsup_{n\to\infty} \max_{g\in K} \| (\Iv_g^* - \eins) s_n^* s_n + \Iv_g^* s_n^* s_n (\Iv_g - \eins) \| \\
 \stackrel{\ref{AS-1:4}}{=} & \dst\max_{g\in K} \| \Iv_g^* - \eins + \Iv_g^*(\Iv_g - \eins) \| \ = \ 0.
\end{array}
\]
Since $\beta$ is strongly stable, we can find two sequences of isometries $r_{1,n}, r_{2,n}\in\CM(B)^\beta$ such that $r_{1,n}r_{1,n}^*+r_{2,n}r_{2,n}^*=\eins$ and $r_{1,n}\to\eins$ in the strict topology; cf.\ the proof of \autoref{lem:local-stability}.
By passing to a subsequence of $r_{1,n}$ and $r_{2,n}$, if necessary, let us additionally assume that $\|(\eins-r_{1,n})s_n\|\to 0$.
We consider
\begin{equation}\label{eq:tnr1nsn}
S_n = r_{1,n}s_n + r_{2,n}(\eins-s_n^*s_n)^{1/2} \in \CM^\beta(B).
\end{equation}
Then $S_n$ is an isometry.
Clearly $b(S_n-s_n)\to 0$ for all $b\in B$.
Since the sequence $s_n^*s_n$ acts like an approximate unit on $\psi(A)$, we also have $(S_n-s_n)\psi(a)\to 0$ for all $a\in A$.
We can in particular observe for every $a\in A$ and large enough $n$ that
\[
\phi(a)s_n-s_n\psi(a) \approx \phi(a)S_n - S_n\psi(a).
\]
For large enough $n$, property \ref{AS-1:3} implies $(\eins-s_n^*s_n)^{1/2}\approx (\eins-\beta_g(s_n^*s_n))^{1/2}$ and therefore by \eqref{eq:tnr1nsn}
\[
S_n-\beta_g(S_n) \approx s_n-\beta_g(s_n).
\]
Using that $\Iu$ takes values in $\CU(\eins+B)$, we have $(\Iu_g-\eins)\beta_g(s_n) \approx (\Iu_g-\eins)\beta_g(S_n)$ uniformly over compact sets if $n$ is sufficiently large.
Given property \ref{AS-1:4}, we also have $s_n(\Iv_g-\eins)\approx S_n(\Iv_g-\eins)$.
Thus we see for large enough $n$ that
\[
\begin{array}{ccl}
s_n\Iv_g-\Iu_g\beta_g(s_n) &=& s_n-\beta_g(s_n) + s_n(\Iv_g-\eins)-(\Iu_g-\eins)\beta_g(s_n) \\
&\approx& S_n-\beta_g(S_n) + S_n(\Iv_g-\eins)-(\Iu_g-\eins)\beta_g(S_n) \\
&=& S_n\Iv_g - \Iu_g\beta_g(S_n).
\end{array}
\]
Note that these approximations are uniform over compact subsets in $G$.
The claim follows with \eqref{eq:AS-1:1}.
\end{proof}

\begin{lemma} \label{lem:absorption-step-2}
Let $\alpha: G\curvearrowright A$ and $\beta: G\curvearrowright B$ be two actions on \cstar-algebras, where $A$ is separable and $B$ is $\sigma$-unital.
Let
\[
(\phi,\Iu), (\psi,\Iv): (A,\alpha)\to (B,\beta)
\]
be two cocycle morphisms such that $(\phi,\Iu)$ approximately 1-dominates $(\psi,\Iv)$.
Then there exists a sequence of contractions $s_n\in B$ such that $s_n^*s_n$ is a (not necessarily increasing) approximately $\beta$-invariant approximate unit, and moreover
\[
\lim_{n\to\infty} \|s_n^*\phi(a)s_n-\psi(a)\|=0,\quad \lim_{n\to\infty}\max_{g\in K} \|s_n^*\Iu_g\beta_g(s_n)-\Iv_gs_n^*s_n\|=0
\]
for all $a\in A$ and every compact set $K\subseteq G$.
\end{lemma}
\begin{proof}
Let $e_n\in B$ be any countable increasing approximate unit.
Then $(e_n,e_n)\in B\oplus B$ is also a countable increasing approximate unit.
If we equip $B\oplus B$ with the action $\beta\oplus\beta^\Iv$, then we can apply \autoref{lem:Kasparov} and conclude that there exists a countable approximate unit $b_n\in B$ that is both asymptotically $\beta$-invariant and asymptotically $\beta^\Iv$-invariant.
In particular this implies that $b_n$ approximately commutes with the values of $\Iv$, uniformly over compact sets, as $n\to\infty$.

We can now apply \autoref{def:wc} to $b=b_n$ and increasing $K,\CF$ and decreasing $\eps$ to find sequences of contractions $s_n\in B$ such that
\begin{equation} \label{eq:AS-2:1}
\lim_{n\to\infty} s_n^*\phi(a)s_n = \lim_{n\to\infty} b_n\psi(a)b_n=\psi(a),\quad a\in A,
\end{equation}
and
\begin{equation} \label{eq:AS-2:2}
\lim_{n\to\infty} \max_{g\in K} \|s_n^*\Iu_g\beta_{g}(s_n)-b_n\Iv_g\beta_{g}(b_n)\| = 0 
\end{equation}
for every compact set $K\subseteq G$.
Applying this relation to the particular case $g=1_G$ and keeping in mind that $b_n$ approximately commutes with the values of $\Iv$, we see that $s_n^*s_n$ approximately equals $b_n^2$ and therefore has the stated properties.
\end{proof}

\begin{cor} \label{cor:approx-dom}
Let $\alpha: G\curvearrowright A$ and $\beta: G\curvearrowright B$ be two actions on \cstar-algebras, where $A$ is separable and $B$ is $\sigma$-unital.
Suppose that $\beta$ is strongly stable.
Let
\[
(\phi,\Iu), (\psi,\Iv): (A,\alpha)\to (B,\beta)
\]
be two proper cocycle morphisms such that $(\phi,\Iu)$ approximately 1-dominates $(\psi,\Iv)$.
Then there exists a sequence of isometries $S_n\in\CM^\beta(B)$ such that
\[
\lim_{n\to\infty} \|\phi(a)S_n - S_n\psi(a)\|,\quad \lim_{n\to\infty}\ \max_{g\in K} \| \Iu_g\beta_g(S_n) - S_n\Iv_g\|=0
\]
for all $a\in A$ and every compact set $K\subseteq G$.
\end{cor}
\begin{proof}
Combine \autoref{lem:absorption-step-1} and \autoref{lem:absorption-step-2}.
\end{proof}

\begin{lemma} \label{lem:removing-multipliers}
Let $\beta: G\curvearrowright B$ be a strongly stable action on a $\sigma$-unital \cstar-algebra.
Let $W\in\CU(\CM^\beta(B))$ be a unitary such that there exists a norm-continuous path $U: [0,\infty)\to\CU(\CM^\beta(B))$ with $U_0\in\CM(B)^\beta$ and 
\[
\lim_{t\to\infty} \max_{g\in K} \|U_t\beta_g(U_t)^*-W\beta_g(W)^*\|=0
\]
for every compact set $K\subseteq G$.
Then there exists a norm-continuous unitary path $v: [0,\infty)\to\CU(\eins+B)$ with $v_0=\eins$ such that
\[
\lim_{t\to\infty} \max_{g\in K} \|v_t\beta_g(v_t)^*-W\beta_g(W)^*\|=0
\]
for every compact set $K\subseteq G$, and $\dst\lim_{t\to\infty} v_t=W$ holds in the strict topology.

In particular, for every $V\in \CU(\CM(B)^\beta)$ the equivariant automorphism $\ad(V)$ is strongly asymptotically unitarily equivalent to $\id_B$.
\end{lemma}
\begin{proof}
Since $\beta$ is strongly stable, we choose (cf.\ \cite[Proposition 1.9]{GabeSzabo22}) strictly continuous paths of isometries $r_1, r_2: [0,\infty)\to\CM(B)^\beta$ satisfying the relation $r_1^{(t)}r_1^{(t)*}+r_2^{(t)}r_2^{(t)*}=\eins$ for all $t\geq 0$, as well as $r_1^{(t)}\to\eins$ strictly as $t\to\infty$.
Consider the strictly continuous path of unitaries $z: [0,\infty)\to\CU(\CM(B)^\beta)$ given by $z_t=r_1^{(t)}r_1^{(0)*}+r_2^{(t)}r_2^{(0)*}$.
We observe for all $b\in B$ that
\begin{equation} \label{eq:remove-multipliers:1}
\lim_{t\to\infty} z_t r_1^{(0)}br_1^{(0)*} z_t^* = \lim_{t\to\infty} b\oplus_{r_1^{(t)},r_2^{(t)}} 0 = b.
\end{equation}
We consider the norm-continuous path $X: [0,\infty)\to\CU(\CM^\beta(B))$ given by $X_t=W \oplus_{r_1^{(0)},r_2^{(0)}} U_0^*W^*U_t$.
Then $X_0=W \oplus_{r_1^{(0)},r_2^{(0)}} U_0^* W^* U_0$ is norm-homotopic to the unit inside $\CU(\CM^\beta(B))$ via
\[
\begin{array}{cll}
X_0 &=& (\eins\oplus_{r_1^{(0)},r_2^{(0)}} U_0^*)(W\oplus_{r_1^{(0)},r_2^{(0)}} W^*U_0) \\
&\sim_h& (U_0^*\oplus_{r_1^{(0)},r_2^{(0)}} \eins)(W\oplus_{r_1^{(0)},r_2^{(0)}} W^*U_0) \\
&=& (U_0^*W\oplus_{r_1^{(0)},r_2^{(0)}} \eins)(\eins\oplus_{r_1^{(0)},r_2^{(0)}} W^*U_0) \\
&\sim_h& (U_0^*W\oplus_{r_1^{(0)},r_2^{(0)}} \eins)(W^*U_0\oplus_{r_1^{(0)},r_2^{(0)}} \eins) \ = \ \eins.
\end{array}
\]
Hence we may apply \cite[Lemma 4.3]{GabeSzabo22} (with $D=0$) and see that there is a norm-continuous path $y: [0,\infty)\to\CU(\eins+B)$ with $y_0=\eins$ such that 
\begin{equation} \label{eq:remove-multipliers:2}
\lim_{t\to\infty} \max_{g\in K} \|y_t\beta_g(y_t)^*-X_t\beta_g(X_t)^*\|=0
\end{equation}
for every compact set $K\subseteq G$, and 
\begin{equation} \label{eq:remove-multipliers:3}
y_t-X_t \stackrel{t\to\infty}{\longrightarrow} 0 \quad\text{strictly}.
\end{equation}
We claim that the path $v_t=z_ty_tz_t^*$ does the job.
First we compute for all $b\in B$ that
\[
\begin{array}{ccl}
\dst \lim_{t\to\infty} v_tb &=& \dst\lim_{t\to\infty} z_ty_tz_t^*b \\
&\stackrel{\eqref{eq:remove-multipliers:1}}{=}& \dst\lim_{t\to\infty} z_ty_t \big( r_1^{(0)} b r_1^{(0)*} \big) z_t^* \\
&\stackrel{\eqref{eq:remove-multipliers:3}}{=}& \dst\lim_{t\to\infty} z_t X_t r_1^{(0)} b  r_1^{(0)*} z_t^* \\
&=& \dst\lim_{t\to\infty} z_t r_1^{(0)} W b r_1^{(0)*} z_t^* \\
&\stackrel{\eqref{eq:remove-multipliers:1}}{=}& Wb.
\end{array}
\]
Similarly, $\lim_{t\to \infty} b v_t = bW$, so $v_t \to W$ strictly. 
Using that $W\beta_g(W)^*\in\CU(\eins+B)$ for all $g\in G$, we observe for every compact set $K\subseteq G$ that
\[
\begin{array}{cl}
\multicolumn{2}{l}{ \dst \lim_{t\to\infty} \max_{g\in K} \|v_t\beta_g(v_t)^*-W\beta_g(W)^*\| } \\
\stackrel{\eqref{eq:remove-multipliers:1}}{=}& \dst \lim_{t\to\infty} \max_{g\in K} \|v_t\beta_g(v_t)^*-z_t(W\beta_g(W)^*\oplus_{r_1^{(0)},r_2^{(0)}}\eins)z_t^*\| \\
\stackrel{z_t\in\CM(B)^\beta}{=}& \dst \lim_{t\to\infty} \max_{g\in K} \|y_t\beta_g(y_t)^*-(W\beta_g(W)^*\oplus_{r_1^{(0)},r_2^{(0)}}\eins)\| \\
\stackrel{\eqref{eq:remove-multipliers:2}}{=}& \dst \lim_{t\to\infty} \max_{g\in K} \|X_t\beta_g(X_t)^*-(W\beta_g(W)^*\oplus_{r_1^{(0)},r_2^{(0)}}\eins)\| \\
=& \dst \lim_{t\to\infty} \max_{g\in K} \|U_0^*W^*U_t\beta_g(U_t^*WU_0)-\eins\| \\
\stackrel{U_0\in\CM(B)^\beta}{=}& \dst \lim_{t\to\infty} \max_{g\in K} \|W^*U_t\beta_g(U_t^*W)-\eins\| \ = \ 0.
\end{array}
\]
The ``in particular'' part follows by applying the result to $V=W=U_t$. 
\end{proof}

The following abstract absorption principle resembles and generalizes analogous technical results appearing in the known proofs of the classical Kirchberg--Phillips theorem.
As it was the case there, this will become a quintessential ingredient in our uniqueness theorem later.

\begin{lemma} \label{lem:strong-sum-absorption}
Let $\alpha: G\curvearrowright A$ and $\beta: G\curvearrowright B$ be two actions on \cstar-algebras, where $A$ is separable and $B$ is $\sigma$-unital.
Suppose that $\beta$ is strongly stable.
Let
\[
(\phi,\Iu), (\theta,\Iy): (A,\alpha)\to (B,\beta)
\]
be two proper cocycle morphisms such that $(\phi,\Iu)$ approximately 1-dominates $(\theta,\Iy)$.
Suppose there exists a unital embedding $\CO_2\to\CM(B)^\beta$ commuting with the range of $\theta$ and $\Iy$, and a unital embedding $\CO_\infty\to\CM(B)^\beta$ commuting with the range of $\phi$ and $\Iu$.
Then $(\phi,\Iu)$ and $(\phi\oplus\theta,\Iu\oplus\Iy)$ are strongly asymptotically unitarily equivalent.\footnote{Here ``$\oplus$'' denotes the Cuntz sum. It is worth noticing right away that the choice of the isometries used to form this sum has no effect on the resulting strong asymptotic unitary equivalence class. This is a consequence of \autoref{lem:removing-multipliers}, since Cuntz sums with two different choices of isometries are unitarily equivalent via an element in $\CM(B)^\beta$.}
\end{lemma}
\begin{proof}
Because of \autoref{cor:approx-dom} we can find isometries $S_n\in\CM^\beta(B)$ with
\[
\lim_{n\to\infty} \|\phi(a)S_n - S_n\theta(a)\|=0,\quad \lim_{n\to\infty}\max_{g\in K} \|S_n\Iy_g-\Iu_g\beta_g(S_n)\|=0
\]
for all $a\in A$ and every compact set $K\subseteq G$.
Since both $(\phi,\Iu)$ and $(\theta,\Iy)$ are proper cocycle morphisms, we have automatically that
\[
\phi(a)S_n-S_n\theta(a) \in B \quad\text{and}\quad S_n\Iy_g-\Iu_g\beta_g(S_n)\in B.
\]
By \cite[Lemma 3.9, Remark 3.10]{GabeSzabo22}, it follows that there exists a norm-continuous path of unitaries $U: [0,\infty)\to\CU(\CM^\beta(B))$ with $U_0\in\CM(B)^\beta$ such that
\[
\lim_{t\to\infty} \|\phi(a)-U_t(\phi\oplus\theta)(a)U_t^*\|=0,\quad \lim_{t\to\infty}\max_{g\in K} \|\Iu_g-U_t(\Iu\oplus\Iy)_g\beta_g(U_t)^*\|=0
\]
for all $a\in A$ and every compact set $K\subseteq G$.
Applying \cite[Corollary 4.4]{GabeSzabo22} to the unitary path $(U_t U_0^*)_t$, one has that $(\phi, \Iu)$ is strongly asymptotically unitarily equivalent to $\ad(U_0) \circ (\phi \oplus \theta, \Iu \oplus \Iy)$, and by \autoref{lem:removing-multipliers} the proof is complete.
\end{proof}

%%%%%%%%%%%%%%%%%%%%%%%%%%%%%%%%%%%%%%%%%%%%%%%%%%%%%%

\section{Isometrically shift-absorbing actions}

\begin{rem} \label{rem:quasifree-actions}
Let $\CH$ be an infinite-dimensional separable Hilbert space.
Recall \cite{Evans80} that the Cuntz algebra $\CO_\infty$ is isomorphic to $\CO_\CH$, the universal unital \cstar-algebra generated by the range of a linear map $\Fs: \CH\to\CO_\CH$ subject to the relation $\Fs(\xi)^* \Fs(\eta)=\langle \xi\mid\eta\rangle\cdot\eins$ for all $\xi,\eta\in\CH$.\footnote{Here we follow the convention that an inner product on a Hilbert space is linear in the second variable instead of the first.
This ensures that in the common language of right Hilbert modules over \cstar-algebras, every Hilbert space is a right $\IC$-Hilbert module in the obvious sense.}
As a consequence, every unitary $U$ on $\CH$ gives rise to a unique automorphism on $\CO_\CH$ such that $\Fs(\xi)$ is sent to $\Fs(U\xi)$ for all $\xi\in\CH$.
The resulting assignment $\CU(\CH)\to\Aut(\CO_\CH)$ is a group homomorphism which is continuous with respect to the strong operator topology on the left and the point-norm topology on the right.
Any group action on $\CO_\infty$ that is conjugate to one factoring through this homomorphism is said to be \emph{quasi-free}.

Let $\CH_F=\bigoplus_{n=0}^\infty \CH^{\otimes n}$ be the Fock space and consider the linear map $\Fs_F: \CH\to\CB(\CH_F)$ given by $\Fs_F(\xi)(x)=\xi\otimes x$ for all $\xi\in\CH$, $x\in\CH^{\otimes n}$ and all $n\geq 0$.
Then $\Fs_F$ gives rise to the Fock representation $\pi: \CO_\CH\to\CB(\CH_F)$, which is easily seen (and well-known) to be irreducible. The vacuum state on $\CO_\CH$ is the vector state given by any fixed unit vector in the one-dimensional subspace $\CH^{\otimes 0}$ of $\CH_F$.
\end{rem}

Although the following is well-known, we provide the full justificiation for the reader's convenience.

\begin{prop} \label{prop:quasifree-outer}
Every  automorphism on $\CO_\CH$ induced by a unitary $U\in\CU(\CH)\setminus\set{\eins}$ is outer, and the vacuum state is invariant.
\end{prop}
\begin{proof}
To see this one can consider the Fock representation of $\CO_\CH$ as above.
For notional convenience we also denote $\Fs_F(\xi)=\hat{\xi}$ in this proof.
We see that the unitary $U_F=\bigoplus_{n=0}^\infty U^{\otimes n} \in\CU(\CH_F)$ satisfies $\ad(U_F)\circ \Fs_F(\xi)=\Fs_F(U\xi)$ for all $\xi\in\CH$. By convention, $U^{\otimes 0}$ is the identity map on the one-dimensional subspace $\CH^{\otimes 0}$ and therefore the vacuum state is invariant. 
Note in particular that the gauge action $\beta: \IT\curvearrowright\pi(\CO_\CH)$ is implemented by the unitaries $(z\cdot\eins)_F$ for $z\in\IT$.

Now let us assume that the automorphism $\delta\in\Aut(\pi(\CO_\CH))$ given by $\hat{\xi}\mapsto \widehat{U\xi}$ for all $\xi\in\CH$ is inner, say implemented by $v\in\pi(\CO_\CH)$.
Then $U_Fv^*$ commutes with $\pi(\CO_\CH)$, so by irreducibility of $\pi$, we may assume $v=U_F\in\pi(\CO_\CH)$.
We will lead this to a contradiction.
We first observe for all $z\in\IT$ that $U_F$ commutes with $(z\eins)_F$, and hence
$\beta_z(U_F)=(zUz^{-1})_F=U_F$.
By classical methods of Cuntz from \cite{Cuntz77} one observes that
\[
\pi(\CO_\CH)^\beta = \IC\eins+\overline{\operatorname{span}}\set{ \hat{\xi}_1\hat{\xi}_2\cdots\hat{\xi}_n\hat{\eta}_n^*\cdots\hat{\eta}_1^* \mid n\geq 1,\ \xi_1,\dots,\xi_n,\eta_1,\dots,\eta_n\in\CH}.
\]
Note that for any $\eta\in\CH$ one has $\hat{\eta}^*|_{\CH^{\otimes 0}}=0$ and 
\[
\hat{\eta}^*(x_1\otimes\dots\otimes x_n)=\langle \eta\mid x_1\rangle (x_2\otimes\dots\otimes x_n)
\] 
for all $n\geq 1$ and elementary tensors $x_1\otimes\dots\otimes x_n\in\CH^{\otimes n}$.
By assumption, we may find some natural number $N$, a scalar $\lambda_0\in\IC$ and finitely many vectors 
\[
\set{\xi_{\ell,m,n} \mid n=1,\dots,N,\ m=1,\dots, k_n,\ \ell=1,\dots,n} \subset\CH
\]
and
\[
\set{\eta_{\ell,m,n} \mid n=1,\dots,N,\ m=1,\dots, k_n,\ \ell=1,\dots,n} \subset\CH
\]
such that the element
\[
a=\lambda_0\eins+\sum_{n=1}^N \sum_{m=1}^{k_n} \hat{\xi}_{1,m,n}\hat{\xi}_{2,m,n}\cdots\hat{\xi}_{n,m,n}\hat{\eta}_{n,m,n}^*\cdots\hat{\eta}_{1,m,n}^*
\]
satisfies $\eps:=\|a-U_F\|< \frac{\|U-\eins\|}{2}$.
We may thus conclude for all elementary tensors of unit vectors $x_1\otimes\dots\otimes x_N\in\CH^{\otimes N}$ the approximate equality
\[
\begin{array}{cl}
 \multicolumn{2}{l}{ Ux_1\otimes\dots\otimes Ux_N } \\
 =_\eps & \dst \lambda_0(x_1\otimes\dots\otimes x_N)+ \\
 & \dst + \sum_{m=1}^{k_N} \langle \eta_{1,m,N}\mid x_1 \rangle \cdots \langle \eta_{N,m,N}\mid x_n \rangle \cdot (\xi_{1,m,N}\otimes\cdots\otimes\xi_{N,m,N})
\end{array}
\]
Now let $x_{N+1}\in\CH$ be a specific unit vector such that $\|Ux_{N+1}-x_{N+1}\|>2\eps$.
Then we observe the approximate equalities
\[
\begin{array}{cl}
 \multicolumn{2}{l}{ Ux_1\otimes\dots\otimes Ux_N \otimes Ux_{N+1} } \\
 =_\eps & \dst \lambda_0(x_1\otimes\dots\otimes x_N\otimes x_{N+1})+ \\
 & \dst + \sum_{m=1}^{k_N} \langle \eta_{1,m,N}\mid x_1 \rangle \cdots \langle \eta_{N,m,N} \mid x_N \rangle \cdot (\xi_{1,m,N}\otimes\cdots\otimes\xi_{N,m,N}\otimes x_{N+1}) \\
 =_\eps & Ux_1\otimes\dots\otimes Ux_N\otimes x_{N+1}
\end{array}
\]
Since all the vectors $x_j\in\CH$ are unit vectors, this leads to the inequality $\|Ux_{N+1}-x_{N+1}\|\leq 2\eps$, a contradiction.
We may thus finally conclude that $\delta$ is outer.
\end{proof}

\begin{rem} \label{rem:KK-class-quasifree}
Let $\CH$ be a separable infinite-dimensional Hilbert space. The construction of $\mathcal O_{\CH} \cong \mathcal O_\infty$  is a special case of the construction of Pimsner from \cite{Pimsner97}. With this in mind the quasi-free actions $\delta : G \curvearrowright \CO_{\CH}$ are exactly the actions considered in \cite[Corollay 4.5, Remark 4.10(2)]{Pimsner97}, and therefore the unital inclusion $(\mathbb C, \id_\IC) \to (\CO_\infty, \delta)$ is a $KK^G$-equivalence for every quasi-free action $\delta$.  
\end{rem}

\begin{defi} \label{def:the-model}
Let $G$ be a second-countable, locally compact group.
Let us choose a left-invariant Haar measure $\mu$ for $G$.
We denote $\CH_G=L^2(G,\mu)$ and $\CH_G^\infty=\ell^2(\IN)\hat{\otimes}\CH_G$, both of which are separable Hilbert spaces, and at least the latter is infinite-dimensional.
Then the left-regular representation $\lambda: G\to\CU(\CH_G)$ is given by $\lambda_g(\xi)(h) = \xi(g^{-1}h)$ for all $g,h\in G$.
The infinite repeat is denoted $\lambda^\infty: G\to\CU(\CH_G^\infty)$.
For the rest of the paper, we denote (using the notation from \autoref{rem:quasifree-actions}) by $\gamma: G\curvearrowright\CO_\infty\cong\CO_{\CH_G^{\infty}}$ the quasi-free action determined by $\gamma_g\circ \Fs=\Fs\circ\lambda^\infty_g$ for all $g\in G$.
\end{defi}

\begin{nota} \label{nota:Hilbert-modules}
Let $A$ and $B$ be two \cstar-algebras.
Let $\FZ$ be a Hilbert $A$-$B$-module, i.e., a right Hilbert $B$-module with $B$-valued inner product 
\[
\langle \cdot\mid\cdot\rangle=\langle \cdot\mid\cdot\rangle_B: \FZ\times\FZ\to B
\] 
and a left-action of $A$; cf.\ \cite[Section 13]{BlaKK}.
For any given Hilbert space $\CH$, the external tensor product $\CH\otimes\FZ$ is then a Hilbert $A$-$B$-module after the obvious identifications $A\cong \IC\otimes A$ and $B\cong\IC\otimes B$.
Suppose furthermore that $\alpha: G\curvearrowright A$ and $\beta: G\curvearrowright B$ are two actions and $\delta: G\curvearrowright\FZ$ is a point-norm continuous action by linear isometries turning $(\FZ,\delta)$ into a Hilbert $(A,\alpha)$-$(B,\beta)$-module, i.e., satisfying the formulas
\[
\delta_g(az)=\alpha_g(a)\delta_g(z),\quad \delta_g(zb)=\delta_g(z)\beta_g(b),\quad \langle \delta_g(z_1) \mid \delta_g(z_2)\rangle = \beta_g(\langle z_1\mid z_2\rangle) 
\] 
for all $g\in G$, $a\in A$, $b\in B$, and $z, z_1, z_2\in\FZ$.
If $\sigma: G\to\CU(\CH)$ is an SOT-continuous unitary representation, we may obtain a point-norm continuous action $\sigma\otimes\delta: G\curvearrowright \CH\otimes \FZ$ by linear isometries.
This turns $(\CH\otimes \FZ, \sigma\otimes\delta)$ into a Hilbert $(A,\alpha)$-$(B,\beta)$-module as well.
Of special significance will be the two choices $\CH=\CH_G$ or $\CH=\CH_G^{\infty}$ of Hilbert spaces, where we denote 
\[
L^2(G,\FZ)=\CH_G\otimes \FZ \quad\text{and}\quad L^2_\infty(G,\FZ) = \CH_G^{\infty}\otimes \FZ,
\] 
respectively.
Recall from \cite[Paragraph 1.3]{Kasparov88} that $L^2(G,\FZ)$ contains the compactly supported continuous functions $\CC_{c}(G,\FZ)$ as a dense pre-Hilbert module.
Unless specified otherwise, we will always implicitly equip $\CH_G$ with the left-regular representation of $G$, and $\CH_G^{\infty}$ with its infinite repeat.
In applications, we will usually encounter the case where $A=B=\FZ$, the latter having the obvious $B$-valued inner product.
We then denote, with slight abuse of notation, $\bar{\beta}=\lambda\otimes\beta: G\curvearrowright L^2(G,B)$ as well as $\bar{\beta}=\lambda^\infty\otimes\beta: G\curvearrowright L^2_\infty(G,B)$.
\end{nota}

\begin{prop} \label{prop:induced-bimodule-map}
Let $\beta: G\curvearrowright C$ be an algebraic action on a \cstar-algebra, and let $C_\beta\subseteq C$ be the largest \cstar-subalgebra on which $\beta$ restricts to a point-norm continuous action.
Let $B\subseteq C$ be a non-degenerate $\beta$-invariant \cstar-subalgebra and $A\subseteq\CM(B)$ any $\beta$-invariant \cstar-subalgebra such that both restrictions $\beta|_B$ and $\beta|_A$ are point-norm continuous.
Let $\CH$ be a separable infinite-dimensional Hilbert space with an SOT-continuous unitary representation $\sigma: G\to\CU(\CH)$.
Suppose there is a linear $\sigma$-to-$\beta$-equivariant map
\[
\Fs: \CH\to \CM(C)\cap A'
\]
that satisfies the identity $\Fs(\xi)^*\Fs(\eta)=\langle\xi\mid\eta\rangle\cdot\eins$ for all $\xi,\eta\in\CH$.
Then there exists an equivariant isometric linear left $A$-module and right $B$-module map
\[
\theta: (\CH\otimes \overline{AB}, \sigma\otimes\beta)\to (C_\beta,\beta)
\]
satisfying the identity $\theta(\xi)^*\theta(\eta)=\langle\xi\mid\eta\rangle_{B}$ for all $\xi,\eta\in\CH\otimes \overline{AB}$.
\end{prop}
\begin{proof}
We consider the linear map on the algebraic tensor product
\[
\theta: \CH\odot\overline{AB}\to C,\quad \xi\otimes b \mapsto \Fs(\xi)b.
\]
Since the range of $\Fs$ is in the relative commutant of $A$, we see that $\theta$ is indeed a left $A$-module and right $B$-module map.
It is obvious that $\beta_g\circ\theta=\theta\circ(\sigma\otimes\beta)_g$ for all $g\in G$, hence the range of the map above belongs to $C_\beta$.
If we can show that $\theta$ satisfies the required inner product formula on $\CH\odot AB$, then $\theta$ is automatically isometric with respect to the norm whose completion yields $\CH\otimes \overline{AB}$, in which case its unique continuous extension will be the required map.
Indeed, if we are given elements $a_1^{(i)},\dots,a_n^{(i)}\in A$, $b_1^{(i)},\dots,b_n^{(i)}\in B$ for $i=1,2$ and $\xi_1,\dots,\xi_n,\eta_1,\dots,\eta_n\in\CH$, then we can directly compute
\[
\begin{array}{cl}
\multicolumn{2}{l}{ \dst\theta\Big( \sum_{j=1}^n \xi_j\otimes a_j^{(1)}b_j^{(1)} \Big)^* \theta\Big( \sum_{k=1}^n \eta_k\otimes a_k^{(2)}b_k^{(2)} \Big) }\\
 =& \dst\sum_{j,k=1}^n b_j^{(1)*} a_j^{(1)*} \Fs(\xi_j)^* \Fs(\eta_k) a_k^{(2)} b_k^{(2)} \\
 =& \dst\sum_{j,k=1}^n b_j^{(1)*} a_j^{(1)*} a_k^{(2)} b_k^{(2)}\cdot \langle \xi_j\mid\eta_k\rangle \\
 =& \dst\Big\langle \sum_{j=1}^n \xi_j\otimes a_j^{(1)} b_j^{(1)} \ \Big| \ \sum_{k=1}^n \eta_k\otimes a_k^{(2)}b_k^{(2)} \Big\rangle_{B}.
\end{array}
\]
\end{proof}

The following dynamical condition is one of the key assumptions that enable classification via the techniques presented in this article.

\begin{defi} \label{defi:the-condition}
Let $\beta: G\curvearrowright B$ be an action on a separable \cstar-algebra.
We say $\beta$ is \emph{isometrically shift-absorbing} if there is a linear equivariant map
\[
\Fs: (\CH_G,\lambda)\to \big( F_{\infty,\beta}(B), \tilde{\beta}_\infty \big)
\]
satisfying the identity $\Fs(\xi)^*\Fs(\eta)=\langle \xi\mid\eta\rangle\cdot\eins$ for all $\xi,\eta\in\CH_G$.
\end{defi}

\begin{prop} \label{prop:the-condition}
Let $G$ be a second-countable, locally compact group with more than one element.\footnote{We assume this because the statement would need to be awkwardly adjusted otherwise, due to the fact that being isometrically shift-absorbing is a vacuous condition if $G=\set{1}$.}
Let $\beta: G\curvearrowright B$ be an action on a separable \cstar-algebra, and let $\gamma: G\curvearrowright\CO_\infty$ be the action from \autoref{def:the-model}.
The following are equivalent:
\begin{enumerate}[label=\textup{(\roman*)},leftmargin=*]
\item $\beta$ is isometrically shift-absorbing. \label{prop:the-condition:1}
\item There exists a unital equivariant $*$-homomorphism from $(\CO_\infty,\gamma)$ to $\big( F_{\infty,\beta}(B), \tilde{\beta}_\infty \big)$. \label{prop:the-condition:2}
\item There exists an equivariant linear $B$-bimodule map 
\[
\theta: ( L^2_\infty(G,B), \bar{\beta}) \to ( B_{\infty,\beta} , \beta_\infty)
\] 
satisfying $\theta(\xi)^*\theta(\eta)=\langle\xi\mid\eta\rangle_B$ for all $\xi,\eta\in L^2_\infty(G,B)$. \label{prop:the-condition:3}
\item There exists an equivariant linear $B$-bimodule map 
\[
\theta: ( L^2(G,B), \bar{\beta}) \to ( B_{\infty,\beta} , \beta_\infty)
\] 
satisfying $\theta(\xi)^*\theta(\eta)=\langle\xi\mid\eta\rangle_B$ for all $\xi,\eta\in L^2(G,B)$. \label{prop:the-condition:4}
\end{enumerate}.
\end{prop}
\begin{proof}
We note right away that due to the definition of the model action $\gamma$, it is tautological that condition \ref{prop:the-condition:2} is equivalent to the existence of a linear equivariant map
\begin{equation}\label{eq:sHGinfty}
\Fs: (\CH_G^\infty,\lambda^\infty)\to \big( F_{\infty,\beta}(B), \tilde{\beta}_\infty \big)
\end{equation}
satisfying the identity $\Fs(\xi)^*\Fs(\eta)=\langle \xi\mid\eta\rangle\cdot\eins$ for all $\xi,\eta\in\CH_G^\infty$.
We have a canonical equivariant isomorphism\footnote{This is \cite[Proposition 1.5(ii)]{BarlakSzabo16}, which originally goes back to \cite[Proposition 1.9(4)+(5)]{Kirchberg04}. Since the isomorphism is natural, it is automatically equivariant with respect to any of the maps induced from automorphisms of $B$.} 
\[
F_{\infty}(B)\cong\CM(\overline{B\cdot B_{\infty}\cdot B})\cap B',
\] 
which (together with the last remark in \autoref{def:sequence-algebras}) gives the implications \ref{prop:the-condition:2}$\implies$\ref{prop:the-condition:3} in light of \autoref{prop:induced-bimodule-map}.
The implication \ref{prop:the-condition:3}$\implies$\ref{prop:the-condition:4} is obvious because $( L^2(G,B), \bar{\beta})$ embeds into $(L^2_\infty(G,B), \bar{\beta})$.

Now let us show \ref{prop:the-condition:4}$\implies$\ref{prop:the-condition:1}.
Suppose that a map $\theta$ is given as in \ref{prop:the-condition:4}.
Let $h_n\in B$ be an increasing approximate unit satisfying $\max_{g\in K} \|\beta_g(h_n)-h_n\|\to 0$ for every compact set $K\subseteq G$.
Then we can consider the contractive linear maps $s_n : \CH_G \to B_{\infty,\beta}$ given by $s_n(\xi) = \theta(\xi \otimes h_n)$ for $\xi\in \CH_G$.
Then given any $b\in B$ and $\xi,\eta\in\CH_G$, we can see that
\[
s_n(\xi)^*s_n(\eta)b = \langle \xi\mid\eta\rangle\cdot h_n^2 b \to \langle \xi\mid\eta\rangle\cdot b
\]
and
\[
s_n(\xi)b-bs_n(\xi) = \theta(\xi\otimes (h_nb-bh_n)) \to 0.
\]
If $K\subseteq G$ is a compact set and $\xi\in\CH_G^\infty$ is any unit vector, then
\[
\max_{g\in K} \|\beta_{g,\infty}(s_n(\xi))-s_n(\lambda_g(\xi))\| = \max_{g\in K} \|\theta(\lambda_g(\xi) \otimes (\beta_g(h_n)-h_n))\| \to 0.
\]
By a standard reindexation trick, we may therefore come up with a contractive linear map $s: \CH_G \to B_{\infty}\cap B'$ satisfying the relations
\[
s(\xi)^* s(\eta) b = \langle\xi\mid\eta\rangle b,\quad \beta_{\infty,g}\circ s=s\circ\lambda_g
\]
for all $b\in B$ and $g\in G$. As $s$ is equivariant and contractive, its image is contained in $B_{\infty,\beta} \cap B'$.
The induced linear map $\Fs: \CH_G \to F_{\infty,\beta}(B)$ given by $\Fs(\xi)=s(\xi)+(B_{\infty,\beta}\cap B^\perp)$ then witnesses isometric shift absorption.

We finally show that \ref{prop:the-condition:1} implies \ref{prop:the-condition:2}.
Let $\Fs$ be a map as in the definition of isometric shift-absorption.
It is necessary to argue a bit differently depending on whether $G$ is finite or not.
Assume for now that $G$ has infinitely many elements.
Let $\Fs \times \Fs : \CH_G \odot \CH_G \to F_{\infty, \beta}(B)$ be the product map given on elementary tensors by $(\Fs \times \Fs)(\xi_1 \otimes \xi_2) = \Fs(\xi_1) \Fs(\xi_2)$.
A straightforward computation gives that $(\Fs \times \Fs)(\xi)^* (\Fs \times \Fs)(\eta) = \langle \xi \mid \eta \rangle \cdot \eins$ for $\xi, \eta \in \CH_G \odot \CH_G$, and that $\beta_g \circ (\Fs \times \Fs) = (\Fs \times \Fs) \circ (\lambda_g \times \lambda_g)$ for $g\in G$.
Thus $\Fs \times \Fs$ extends to an equivariant linear map $(\CH_G \hat \otimes \CH_G, \lambda \times \lambda) \to \big( F_{\infty,\beta}(B), \tilde{\beta}_\infty \big)$ with the inner product condition.
By Fell's absorption principle $(\CH_G^\infty, \lambda^\infty) \cong (\CH_G \otimes \CH_G, \lambda \times \lambda)$ since $\CH_G$ is infinite-dimensional.
Hence a map as in \eqref{eq:sHGinfty} exists whenever $G$ has infinitely many elements, thus witnessing \ref{prop:the-condition:2}.

Still assuming \ref{prop:the-condition:1}, suppose now that $G$ is a non-trivial finite group. Fix $h\in G\setminus\{1\}$ and let
\[
s_1 := \tfrac{1}{|G|^{1/2}} \sum_{g\in G}\Fs(\delta_g) \Fs(\delta_g), \quad s_2 := \tfrac{1}{|G|^{1/2}} \sum_{g\in G} \Fs(\delta_g) \Fs (\delta_{gh}) \in F_{\infty, \beta}(B).
\]
It is straight-forward to check that $s_1,s_2$ are $\tilde \beta_\infty$-invariant isometries for which $s_1^\ast s_2 = 0$. Thus the fixed-point algebra $F_{\infty, \beta}(B)^{\tilde \beta_\infty}$ is properly infinite. Equivalently, there is an equivariant unital embedding $(\CO_\infty, \id_{\CO_\infty}) \to \big( F_{\infty,\beta}(B), \tilde{\beta}_\infty \big)$.
Hence there is an equivariant linear map $\hat{\Fs} : (\ell^2(\mathbb N), \eins) \to \big( F_{\infty,\beta}(B), \tilde{\beta}_\infty \big)$ such that $\hat{\Fs}(\xi)^* \hat{\Fs}(\eta) = \langle \xi \mid \eta\rangle \cdot \eins$ for $\xi, \eta\in \ell^2(\mathbb N)$.
Arguing as for infinite groups, the product map $\hat{\Fs} \times \Fs$ realizes a map as in \eqref{eq:sHGinfty} and thus \ref{prop:the-condition:2} is satisfied.
\end{proof}

\begin{prop} \label{prop:isa-implies-absorption}
Let $G$ be a second-countable, locally compact group with more than one element.
Let $\beta: G\curvearrowright B$ be an action on a separable \cstar-algebra.
If $\beta$ is isometrically shift-absorbing and amenable, then $\beta$ is equivariantly $\CO_\infty$-absorbing, i.e., $\beta \cc \beta \otimes \id_{\CO_\infty}$.
\end{prop}
\begin{proof}
Since $G$ is not the trivial group, the Hilbert space $\CH_G$ has at least two orthogonal unit vectors.
If $\Fs: \CH_G\to F_{\infty,\beta}(B)$ is a linear map witnessing that $\beta$ is isometrically shift-absorbing, then the images of two such unit vectors are two isometries in $F_{\infty,\beta}(B)\subseteq F_\infty(B)$ with orthogonal ranges.
Since $B$ is separable, this means that there exist two approximately central sequences of contractions $t^{(1)}_n, t_n^{(2)}\in B$ such that $t_n^{(i)*}t_n^{(j)}$ strictly converges to $\delta_{ij}\cdot\eins$ for $i,j=1,2$.
We let $\theta$ be a map as in condition \autoref{prop:the-condition}\ref{prop:the-condition:4}.
Let $\eps>0$ and let $K\subseteq G$, $\CF\subset B_{\leq 1}$ be two compact sets.
Let $\mu$ be a left-invariant Haar measure on $G$.
Since $\beta$ is amenable, there exists a function $\zeta\in\CC_c(G,B)$ with $\|\zeta\|_2\leq 1$ and
\[
\max_{b\in\CF} \|b\zeta-\zeta b\|_2+\|(\eins-\langle\zeta\mid\zeta\rangle)b\|\leq\eps,\quad \max_{g\in K} \|\zeta-\bar{\beta}_g(\zeta)\|_2\leq\eps.
\]
Let us denote $R=\overline{\supp(\zeta)}\subseteq G$.
Define elements $\xi_n^{(1)},\xi_n^{(2)}\in \CC_c(G,B)$ via
\[
\xi_n^{(i)}(h) = \beta_h(t^{(i)}_n) \zeta(h) ,\quad h\in G,\ i=1,2.
\]
Here we implicitly exploit the fact that $\CC_c(G,B)$ is a $\CC_b(G,B)$-bimodule in an obvious way, such that one has $\|f\eta\|_2^2\leq\|f\|^2\|\eta\|_2^2$ for all $f\in\CC_b(G,B)$ and $\eta\in\CC_c(G,B)$.
Denote by $\Io_\beta: B\to\CC_b(G,B)$ the $*$-homomorphism given via $\Io_\beta(b)(g)=\beta_g(b)$ for all $b\in B$ and $g\in G$.
Using this perspective we can view $\xi_n^{(i)}=\Io_\beta(t^{(i)}_n)\zeta$.

We note that for $h\in G$ and $g\in K$ one has
\[
\bar{\beta}_g(\xi_n^{(i)})(h) = \beta_g(\xi_n^{(i)}(g^{-1} h)) = \beta_g(\beta_{g^{-1} h}(t_n^{(i)})\zeta(g^{-1} h)) = \beta_h(t_n^{(i)}) \cdot \bar {\beta}_g(\zeta)(h)
\]
and therefore
\[
\bar{\beta}_g(\xi_n^{(i)})(h)-\xi_n^{(i)}(h) = \beta_h(t^{(i)}_n)\cdot (\bar{\beta}_g(\zeta)(h)-\zeta(h)).
\]
Thus we observe for every $g\in K$ that
\[
\begin{array}{ccl}
\|\beta_{g,\infty}(\theta(\xi_n^{(i)}))-\theta(\xi_n^{(i)})\| &=& \|\theta\big( \bar{\beta}_g(\xi^{(i)}_n)-\xi^{(i)}_n \big)\| \\
&=& \|\bar{\beta}_g(\xi^{(i)}_n)-\xi^{(i)}_n \|_2 \\
&\leq& \|\bar{\beta}_g(\zeta)-\zeta\|_2 \ \leq \ \eps.
\end{array}
\]
Furthermore we observe for $i,j=1,2$ and $b\in \CF$ that
\[
\def\arraystretch{1.5}
\begin{array}{cl}
\multicolumn{2}{l}{ \| \theta(\xi_n^{(i)})b-b\theta(\xi_n^{(i)})\| } \\
=& \|\theta(\xi_n^{(i)}b-b\xi_n^{(i)})\| \\
=&  \|\xi_n^{(i)}b-b\xi_n^{(i)}\|_2\\
=& \|\Io_\beta(t^{(i)}_n)\zeta b-b\Io_\beta(t^{(i)}_n)\zeta\|_2 \\
\leq & \dst\max_{h\in R} \|[b,\beta_h(t^{(i)}_n)]\| + \|\Io_\beta(t^{(i)}_n)\zeta b-\Io_\beta(t^{(i)}_n)b\zeta\|_2 \\
\leq & \dst\max_{h\in R} \|[b,\beta_h(t^{(i)}_n)]\| + \|\zeta b-b\zeta\|_2 \\
\leq & \dst\max_{h\in R} \|[\beta_{h^{-1}}(b),t^{(i)}_n]\| + \eps \ \to \ \eps,
\end{array}
\]
and
\[
\def\arraystretch{1.5}
\begin{array}{cl}
\multicolumn{2}{l}{ \|\theta(\xi_n^{(i)})^*\theta(\xi_n^{(j)})b-\delta_{ij}b\| } \\
=& \|\langle \xi_n^{(i)}\mid\xi_n^{(j)}\rangle b-\delta_{ij}b\| \\
=& \dst \max_{h\in R} \| \big( \beta_h(t^{(i)}_n t^{(j)}_n)-\delta_{ij} \big) \zeta(h) \| + \delta_{ij}\cdot \| (\eins-\langle\zeta\mid\zeta\rangle)b \| \\
\leq&  \dst \max_{h\in R} \| \big( \beta_h(t^{(i)}_n t^{(j)}_n)-\delta_{ij} \big) \zeta(h) \| + \eps \ \to \ \eps.
\end{array}
\]
Clearly we have $\|\theta(\xi_n^{(i)})\|\leq 1$ for all $n\geq 1$ and $i=1,2$.
Since the triple $(\eps,K,\CF)$ was arbitrary, we can apply a standard diagonal sequence argument to obtain two contractions $w_1, w_2\in (B_{\infty,\beta}\cap B')^{\beta_\infty}$ such that $w_i^*w_jb=\delta_{ij}b$ for all $b\in B$ and $i,j=1,2$.
The resulting isometries $v_i=w_i+(B_{\infty,\beta}\cap B^\perp)$ are then in $F_{\infty,\beta}(B)^{\tilde{\beta}_\infty}$ and have orthogonal ranges.
In particular we may obtain a unital inclusion $\CO_\infty\subset F_{\infty,\beta}(B)^{\tilde{\beta}_\infty}$ (see, e.g., \cite[Proposition 1.1.2]{Rordam}), which yields $\beta\cc\beta\otimes\id_{\CO_\infty}$ by \cite[Corollary 3.8]{Szabo18ssa}.
\end{proof}

\begin{rem}\label{multiplierbimodule}
Note that the $B$-bimodule maps $\theta : L^2_{(\infty)}(G, B) \to B_{\infty, \beta}$ (the notation suggests that inserting the symbol $\infty$ is optional) from \autoref{prop:the-condition} are automatically $\CM(B)$-bimodule maps. One can verify this directly using approximate units, but here is an alternative short argument using Cohen's factorization property (see for instance \cite[Theorem 4.6.4]{BrownOzawa}). In fact, the factorization property implies that any element $\xi\in L^2_{(\infty)}(G,B)$ can be written as $b_1 \xi_0 b_2$ for $b_1,b_2\in B$ and $\xi_0 \in L^2_{(\infty)}(G,B)$. Hence for every $x\in \CM(B)$ one has
\[
x\theta(\xi) = x \theta(b_1 \xi_0 b_2) = (xb_1) \theta(\xi_0 b_2) = \theta((xb_1) \xi_0 b_2) = \theta(x\xi),
\] 
and similarly $\theta(\xi)x = \theta(\xi x)$. 
\end{rem}

Next we include an argument proving a folklore result, which we will use to see that isometric shift-absorption has no $K$-theoretical obstruction for group actions.

\begin{lemma}\label{lem:inductivelimit}
Let $\alpha_n : G \curvearrowright A_n$ be actions on separable \cstar-algebras for $n\in \mathbb N$. 
Suppose $\phi_n : (A_n, \alpha_n) \to (A_{n+1}, \alpha_{n+1})$ are equivariant $*$-homomorphisms and $\psi_n : (A_{n+1}, \alpha_{n+1}) \to (A_n, \alpha_n)$ are equivariant completely positive contractive maps such that $\psi_n \circ \phi_n = \id_{A_n}$.
If each $\phi_n$ is a $KK^G$-equivalence, then the canonical equivariant $*$-homomorphism $(A_1, \alpha_1) \to \varinjlim ((A_n, \alpha_n), \phi_n)$ is a $KK^G$-equivalence.
\end{lemma}
\begin{proof}
By the results of \cite[Section 2.4]{MeyerNest06} the inductive system $((A_n, \alpha_n), \phi_n)$ is admissible (because of the maps $\psi_n$) and therefore there is an induced short exact sequence
\[
0 \to \varprojlim\!{}^1 KK^G_1(\alpha_n, \beta) \to KK^G(\varinjlim \alpha_n , \beta) \to \varprojlim KK^G(\alpha_n, \beta) \to 0
\]
for any action $\beta: G\curvearrowright B$ on a separable \cstar-algebra. 
As each map $\phi_n$ is a $KK^G$-equivalence, the induced maps $KK^G(\alpha_{n+1}, \beta) \to KK^G(\alpha_n, \beta)$ are isomorphisms.
Hence the $\varprojlim^1$-term vanishes, and $\varprojlim KK^G(\alpha_n, \beta) \cong KK^G(\alpha_1, \beta)$ canonically.
Therefore the canonical equivariant map $(A_1, \alpha_1) \to \varinjlim(A_n, \alpha_n)$ induces an isomorphism $KK^G(\varinjlim \alpha_n, \beta) \to KK^G(\alpha_1, \beta)$.
Applying this to $\beta = \alpha_1$ and $\beta = \varinjlim \alpha_n$ it easily follows that $A_1 \to \varinjlim A_n$ is a $KK^G$-equivalence.
\end{proof}

\begin{cor}\label{cor:modelKKtrivial}
Let $\gamma: G\curvearrowright\CO_\infty$ be the action from \autoref{def:the-model}. Then the canonical unital inclusion $(\IC, \id_{\IC}) \to (\CO_\infty^{\otimes \infty}, \gamma^{\otimes \infty})$ is a $KK^G$-equivalence.
\end{cor}
\begin{proof}
Let $(A_n, \alpha_n) = (\CO_\infty^{\otimes n}, \gamma^{\otimes n})$. By \autoref{prop:quasifree-outer} the vacuum state on $\CO_{\CH_G^\infty} \cong \CO_\infty$ is a $\gamma$-invariant state. Hence the slice map with respect to this state gives an equivariant conditional expectation $\psi_n : (\CO_\infty^{\otimes (n+1)}, \gamma^{\otimes (n+1)}) \to (\CO_\infty^{\otimes n}, \gamma^{\otimes n})$. By \autoref{rem:KK-class-quasifree} and \autoref{lem:inductivelimit} the result follows.
\end{proof}

The following result demonstrates that the main classification theorem cannot be extended (using only $KK^G$) to cover more general amenable actions on nuclear \cstar-algebras, since the actions we classify form a skeleton in the $KK^G$-category.
The version given here is slightly more general than \autoref{thmi:range} given in the introduction.
The proof (which amounts to a combination of literature sources) uses a Cuntz-Pimsner algebra construction which is due to Kumjian \cite{Kumjian04} in the non-equivariant case, and Meyer \cite{Meyer21} in the equivariant case.
That amenability is preserved by this construction is due to Ozawa--Suzuki \cite{OzawaSuzuki21}.

\begin{theorem}[cf.\ Pimsner, Kumjian, Meyer, Ozawa--Suzuki] \label{thm:range}
Let $G$ be a second-countable locally compact group and let $\alpha : G \curvearrowright A$ be an amenable action on a separable nuclear \cstar-algebra.
Then there exists an amenable and isometrically shift-absorbing action $\beta : G \curvearrowright B$ on a stable Kirchberg algebra and an equivariant embedding $(A,\alpha)\to (B, \beta)$ that induces a $KK^G$-equivalence.
If $A$ is unital (in which case $G$ must be exact), then $B$ can instead be chosen unital such that the embedding is unital.
\end{theorem}
\begin{proof}
\cite[Theorem 6.1]{OzawaSuzuki21} implies the above result except for the part about $\beta$ being isometrically shift-absorbing. By \autoref{cor:modelKKtrivial} we may tensor $(B, \beta)$ with $(\mathcal O_\infty^{\otimes \infty}, \gamma^{\otimes \infty})$ and obtain an action which additionally is isometrically shift-absorbing by \autoref{prop:the-condition}.
\end{proof}

\begin{rem} \label{rem:ISA-KK}
If $G$ has the Haagerup property (see \cite{CCJJV}), then \autoref{thm:range} holds even without assuming that $\alpha$ is amenable.
In fact, by a theorem of Higson--Kasparov \cite{HigsonKasparov01} there exists an proper (and therefore amenable) action on a separable type I \cstar-algebra which is $KK^G$-equivalent to $\mathbb C$.
Hence we may tensor this onto any action and obtain an amenable action with the same $KK^G$-equivalence class.
In the case where $G$ is exact, the unital subcase can be obtained by additionally applying \autoref{thm:exact-HP} or \cite[Theorem B]{Suzuki23}.

On the other hand, if $G$ admits an amenable action on any \cstar-algebra which is $KK^G$-equivalent to $\IC$, then the quotient map $C^\ast(G) \to C^\ast_\lambda(G)$ from the full to the reduced group \cstar-algebra is a $KK$-equivalence by \cite[Proposition 6.5]{OzawaSuzuki21}.
If $G$ is non-compact and has property (T) (as opposed to the Haagerup property), then this quotient map is not a $KK$-equivalence since the canonical Kazhdan projection has non-trivial class in $K_0(C^*(G))$ but it vanishes in $C^\ast_\lambda(G)$ (cf.~\cite[Remark 2.7]{Cuntz83_2} and \cite[Corollary 3.7]{JulgValette84}). Thus no amenable action is $KK^G$-equivalent to $\mathbb C$ for such groups.
\end{rem}

Suppose for now that $G$ is a countable discrete group and that $\beta: G\curvearrowright B$ is an action on a separable \cstar-algebra.
Then, if we keep in mind that the left-regular representation on $\ell^2(G)$ is cyclic with respect to the charactistic function over the neutral element, it follows that $\beta$ is isometrically shift-absorbing if and only if there exists an isometry $s\in F_\infty(B)$ such that $ss^*\perp\tilde{\beta}_{\infty,g}(ss^*)$ for all $g\neq 1$.

With this observation at hand, we identify isometric shift-absorption for actions by discrete groups on Kirchberg algebras in terms of outerness. In the unital case this was observed by Izumi--Matui in \cite[Lemma 3.4]{IzumiMatui10}.
For completion, we fill in a proof that also works in the non-unital case using results of the second author.
These are based on substantially deeper results of Kishimoto \cite{Kishimoto81} and Kirchberg--Phillips \cite{KirchbergPhillips00}.
We remark that this is the only place in the paper where ultrafilters are used.\footnote{One could modify the results from  \cite{Szabo18kp} and obtain a proof without using ultrafilters, but this would be a major digression from our main objective.}

\begin{prop} \label{rem:outer-actions}
Suppose that $G$ is a countable discrete group, and let $\beta: G\curvearrowright B$ be an action on a Kirchberg algebra.
Then $\beta$ is isometrically shift-absorbing if and only if it is pointwise outer.
\end{prop}
\begin{proof}
Any inner automorphism on $B$ induces the trivial automorphism on $F_{\infty}(B)$, and thus isometrically shift-absorbing actions must be pointwise outer. Conversely, suppose $\beta$ is pointwise outer.
Let $\CF \subset B$ and $K \subseteq G\setminus\{1\}$ be finite subsets and $\eps>0$. By \cite[Propositions 2.2, 3.2, and Theorem 3.1]{Szabo18kp}, the (ultrapower) central sequence algebra $F_\omega(B)$ is purely infinite and simple, and there exists a non-zero projection $p\in F_\omega(B)$ such that $p \tilde{\beta}_{\infty, g}(p) =0$ for $g\in K$. As $F_\omega(B)$ is unital, purely infinite and simple, there exists an isometry $v\in F_\omega(B)$ such that $vv^* \leq p$. By picking a contractive representing sequence for $v$, and choosing a suitable entry from this sequence, we obtain a contraction $d\in B$ satisfying
\[
\max_{b\in \CF} \| (\eins - d^* d) b\| \leq \eps, \quad \max_{b\in \CF}\| db - bd \| \leq \eps, \quad  \max_{b\in \CF, g\in K} \| d^* \beta_g(d) b\| \leq \eps.
\]
By a standard diagonal argument we obtain an isometry $s\in F_\infty(B)$ such that $ss^* \perp \tilde{\beta}_{\infty,g}(ss^*)$ for $g\neq 1$ as desired.
\end{proof}

Next to using our recent results from \cite{GabeSzabo22}, the following technical observation can be seen as the driving force behind our classification theory.

\begin{lemma} \label{lem:key-lemma}
Let $\alpha: G\curvearrowright A$ and $\beta: G\curvearrowright B$ be two actions on separable \cstar-algebras, and assume that $\beta$ is amenable and isometrically shift-absorbing.
Let $(\phi,\Iu), (\psi,\Iv): (A,\alpha) \to (\CM(B),\beta)$ be two cocycle representations.
Suppose that $\phi$ weakly contains $\psi$ when they are viewed as cocycle representations with respect to the trivial group.
Then it follows that $(\phi,\Iu)$ approximately 1-dominates $(\psi,\Iv)$.
\end{lemma}
\begin{proof}
As before, we choose a Haar measure $\mu$ on $G$.
Let compact sets $1\in K\subseteq G$, $\CF\subset A_{\leq 1}$, and $\eps> 0$ be given.
Let $b\in B$ be any contraction.
Since $\beta$ is amenable, so is $\beta^\Iv$, and there exists a function $\zeta\in\CC_c(G,B)$ with $\|\zeta\|_2\leq 1$ such that
\begin{equation} \label{eq:key-1}
\max_{a\in\CF} \big\| \big( \langle\zeta\mid \psi(a)\zeta\rangle - \psi(a) \big) b \big\|\leq\eps
\end{equation}
\begin{equation} \label{eq:key-2}
\max_{g\in K} \big\| \big(\eins-\langle\zeta\mid \bar{\beta}^\Iv_g(\zeta) \rangle\big)\Iv_g\beta_g(b) \|\leq\eps.
\end{equation}
Choose a positive contraction $e\in B$ (apply \autoref{lem:Kasparov} with $\beta^\Iv$ in place of $\beta$) with the following properties:
\begin{equation} \label{eq:key-3}
\max_{g\in R} \|\beta^\Iv_g(e)-e\|\leq\eps ,\quad\text{where } R=\overline{\supp(\zeta)}.
\end{equation}
\begin{equation} \label{eq:key-4}
\max_{g\in K} \| \langle \zeta \mid e^2\Iv_g \bar{\beta}_g(\zeta) \rangle - e\langle \zeta \mid \Iv_g\bar{\beta}_g(\zeta) \rangle e \|\leq\eps.
\end{equation}
\begin{equation} \label{eq:key-5}
\max_{a\in \CF} \| \langle \zeta \mid e \psi(a) e \zeta \rangle - e \langle \zeta \mid \psi(a) \zeta \rangle e \|\leq\eps.
\end{equation}
\begin{equation}  \label{eq:key-6}
\max_{g\in K}\|(\eins-e) \beta_g(b) \|\leq\eps.
\end{equation}
Using that $\phi$ weakly contains $\psi$ (as an ordinary $*$-homomorphism), we may choose a collection of elements $\{c_{k} \mid k=1,\dots,N\}\subset B$ satisfying
\begin{equation} \label{eq:key-7}
\Big\| e^2-\sum_{k=1}^{N} c_{k}^*c_{k} \Big\| \leq \eps %/\mu(R)
\end{equation}
and
\begin{equation} \label{eq:key-8}
\max_{a\in\CF}\ \max_{h\in R} \Big\| e\psi(\alpha_{h^{-1}}(a))e - \sum_{k=1}^N c_{k}^*\phi(\alpha_{h^{-1}}(a))c_{k} \Big\| \leq \eps  %/\mu(R).
\end{equation}
%Without loss of generality we may assume that $\|\sum_{k=1}^{N} c_{k}^*c_{k}\|\leq 1$.
Since $\beta$ is isometrically shift-absorbing, it follows from \autoref{prop:the-condition} that there exists an equivariant linear $B$-bimodule map 
\[
\theta: ( L^2_\infty(G,B), \bar{\beta}) \to ( B_{\infty,\beta} , \beta_\infty)
\] 
satisfying $\theta(\xi)^*\theta(\eta)=\langle\xi\mid\eta\rangle_B$ for all $\xi,\eta\in L^2_\infty(G,B)$.
For every $k=1,\dots,N$ we consider $\xi_k\in \CC_c(G,B)$ via
\[
\xi_k(h)= \Iu_h\beta_h(c_k)\Iv_h^*\zeta(h) ,\quad h\in G.
\]
We set $\xi=(\xi_1,\xi_2,\dots,\xi_N,0,0,\dots)\in L^2_\infty(G,B)$.
Recall that $\theta$ is a $\CM(B)$-bimodule map by \autoref{multiplierbimodule}. Using this, we compute for every $g\in K$ that
\[
\begin{array}{cl}
\multicolumn{2}{l}{ \dst\theta(\xi)^* \Iu_g \beta_{g,\infty}(\theta(\xi)) } \\ 
=& \theta(\xi)^*\theta(\Iu_g\bar{\beta}_g(\xi)) \\
=& \dst\langle \xi \mid \Iu_g\bar{\beta}_g(\xi)\rangle \\
=& \dst \sum_{k=1}^N \int_G \zeta(h)^* \Iv_h \beta_h(c_k^*)\Iu_h^*\cdot  \Iu_g \beta_g\Big( \Iu_{g^{-1}h}\beta_{g^{-1}h}(c_k)\Iv_{g^{-1}h}^* \zeta(g^{-1}h) \Big) ~d\mu(h) \\
=& \dst \sum_{k=1}^N \int_{G} \zeta(h) \Iv_h \beta_h(c_k^*c_k)\beta_g(\Iv_{g^{-1}h})^* \bar{\beta}_g(\zeta)(h) ~d\mu(h) \\ 
=& \dst \sum_{k=1}^N \int_{G} \zeta(h) \beta^\Iv_h(c_k^*c_k) \Iv_g \bar{\beta}_g(\zeta)(h) ~d\mu(h) \\
\stackrel{\eqref{eq:key-3},\eqref{eq:key-7}}{=}_{\makebox[0pt]{\footnotesize\hspace{-9mm} $2\eps$}} & \dst \int_G \zeta(h) e^2 \Iv_g \bar{\beta}_g(\zeta)(h)  ~d\mu(h) \\
=& \dst \big\langle \zeta \mid e^2\Iv_g\bar{\beta}_g(\zeta) \big\rangle \\
\stackrel{\eqref{eq:key-4}}{=}_{\makebox[0pt]{\footnotesize\hspace{-2mm} $\eps$}} & e \big\langle \zeta\mid \Iv_g \bar{\beta}_g(\zeta) \big\rangle e. 
\end{array}
\]
Using how we chose $e$ and $\zeta$, we use this to observe that
\[
\def\arraystretch{1.25}
\begin{array}{ccl}
\theta(\xi b)^*\Iu_g\beta_{g,\infty}(\theta(\xi b)) &=_{\makebox[0pt]{\footnotesize\hspace{3mm} $3\eps$}}&  b^* e \big\langle \zeta\mid \Iv_g \bar{\beta}_g(\zeta) \big\rangle e \beta_g(b) \\
&\stackrel{\eqref{eq:key-6}}{=}_{\makebox[0pt]{\footnotesize\hspace{-1mm} $2\eps$}}& b^* \big\langle \zeta\mid \bar{\beta}^\Iv_g(\zeta) \big\rangle \Iv_g \beta_g(b) \\
&\stackrel{\eqref{eq:key-2}}{=}_{\makebox[0pt]{\footnotesize\hspace{-3mm} $\eps$}}& b^*\Iv_g\beta_g(b).
\end{array}
\]
Moreover, again using that $\theta$ is a $\CM(B)$-bimodule map, we compute for every $a\in\CF$ that
\begin{longtable}{cl}
\multicolumn{2}{l}{ $\theta(\xi b)^*\phi(a)\theta(\xi b)$  } \\
$=$ & $b^*\theta(\xi)^*\theta(\phi(a)\xi)b$ \\
$=$ & $\dst b^* \sum_{k=1}^N \int_{G} \zeta(h)^* \Iv_h \beta_h(c_k)^*\Iu_h^*\phi(a)\Iu_h\beta_h(c_k)\Iv_h^* \zeta(h) ~d\mu(h) \cdot b$ \\
$=$ & $\dst b^* \sum_{k=1}^N \int_{G} \zeta(h)^* \Iv_h \beta_h(c_k)^*\beta_h(\phi( \alpha_{h^{-1}}(a) ))\beta_h(c_k)\Iv_h^* \zeta(h) ~d\mu(h) \cdot b$ \\
$=$ & $\dst b^* \int_{G} \zeta(h)^* \beta_h^\Iv\Big( \sum_{k=1}^N c_k^*\phi( \alpha_{h^{-1}}(a) ) c_k \Big) \zeta(h) ~d\mu(h)\cdot b$ \\
$\stackrel{\eqref{eq:key-8}}{=}_{\makebox[0pt]{\footnotesize\hspace{-2mm} $\eps$}}$ & $\dst b^* \int_{G} \zeta(h)^* \beta_h^\Iv\Big( e \psi( \alpha_{h^{-1}}(a) ) e \Big) \zeta(h) ~d\mu(h) \cdot b$ \\
$\stackrel{\eqref{eq:key-3}}{=}_{\makebox[0pt]{\footnotesize\hspace{-1mm} $2\eps$}}$ & $\dst b^* \int_{G} \zeta(h)^* e \beta^\Iv_h\Big( \psi( \alpha_{h^{-1}}(a)) \Big) e \zeta(h)~d\mu(h)\cdot b$ \\
$=$ & $b^* \big\langle \zeta \mid e\psi(a)e \zeta\rangle b$ \\
$\stackrel{\eqref{eq:key-5}}{=}_{\makebox[0pt]{\footnotesize\hspace{-2mm} $\eps$}}$ & $b^* e \big\langle \zeta \mid \psi(a) \zeta\rangle e b$ \\
$\stackrel{\eqref{eq:key-6}}{=}_{\makebox[0pt]{\footnotesize\hspace{-1mm} $2\eps$}}$ & $b^* \big\langle \zeta \mid \psi(a) \zeta\rangle b$ \\
$\stackrel{\eqref{eq:key-1}}{=}_{\makebox[0pt]{\footnotesize\hspace{-2mm} $\eps$}}$ & $b^*\psi(a)b.$
\end{longtable}
\noindent
By lifting the element $\theta(\xi b)\in B_{\infty,\beta}$ to a sequence of contractions in $B$, the computations above imply that we can obtain a contraction $v\in B$ such that
\[
\max_{g\in K} \|b^*\Iv_g\beta_g(b) - v^*\Iu_g\beta_g(v)\|\leq 7\eps
\]
and
\[
\max_{a\in\CF} \|b^*\psi(a)b-v^*\phi(a)v\|\leq 8\eps.
\]
Since $K,\CF,\eps$ and $b$ were arbitrary, this implies that $(\phi,\Iu)$ approximately 1-dominates $(\psi,\Iv)$.
\end{proof}

\begin{cor} \label{cor:inf-repeats-are-absorbing}
Let $\alpha: G\curvearrowright A$ and $\beta: G\curvearrowright B$ be two actions on separable \cstar-algebras, and assume that $\beta$ is amenable, isometrically shift-absorbing and strongly stable.
Suppose that $A$ or $B$ is nuclear.
Let $(\phi,\Iu): (A,\alpha)\to (B,\beta)$ be a cocycle morphism such that $\phi$ is full.
Then the infinite repeat $(\phi^\infty,\Iu^\infty): (A,\alpha)\to(\CM(B),\beta)$ is an absorbing cocycle representation.
\end{cor}
\begin{proof}
By \cite[Corollary 3.12]{GabeSzabo22}, the claim is true if $(\phi,\Iu)$ weakly contains every cocycle representation.
By \autoref{lem:key-lemma}, this is true if $\phi$ weakly contains every $*$-homomorphism $A\to\CM(B)$, which follows from \autoref{rem:ordinary-dom}.
\end{proof}

%%%%

\section{The dynamical $\CO_2$-embedding theorem}

In this section we prove a dynamical version of the $\CO_2$-embedding theorem, namely \autoref{thmi:O2-embedding}.
Our strategy follows the broad strokes of the classical strategy of proving the known $\CO_2$-embedding theorem, but requires some specific technical setup in order to be adapted.
In particular we need access to a uniqueness theorem for equivariant maps of the form $(A,\alpha)\to (B_{\infty,\beta},\beta_\infty)$, where the results from the previous section may not immediately apply because for certain choices of $G$ the action $\beta_\infty$ is never amenable.\footnote{We note, however, that this issue and the resulting technical setup becomes somewhat redundant when $G$ is exact by virtue of the results in \cite{OzawaSuzuki21}}
Some of the arguments that follow are nevertheless similar in spirit to the arguments in the previous sections, but work based on the amenability of $\alpha$ rather than $\beta$, at the expense of having a rather narrow range of applicability (though sufficient for the goal of this section).

\begin{lemma} \label{lem:reduce-1-dom}
Let $\alpha: G\curvearrowright A$ and $\beta: G\curvearrowright B$ be two actions on  \cstar-algebras.
Suppose that $A$ is separable and $\beta$ is strongly stable.
Let $\phi, \psi: (A,\alpha) \to (B_{\infty, \beta},\beta_\infty)$ be two equivariant $*$-homomorphisms.
Suppose that for every contraction $d\in A$ there exists a contraction $s\in B_{\infty}$ with
\[
s^*\beta_{\infty,g}(s) = \psi(d^*\alpha_g(d)) \qquad s^* \phi(a) s = \psi(d^*ad)
\]
for all $a\in A$ and $g\in G$. Then it follows that $(\phi,\eins)$ approximately 1-dominates $(\psi,\eins)$.
\end{lemma}
\begin{proof}
Let $b\in B_{\infty, \beta}$ be a contraction. 
Choose an approximately $\alpha$-invariant approximate unit $e_n\in A$.
For each $n\geq 1$, apply the assumption for $e_n$ in place of $d$ and choose a corresponding element $s_n$.
Given that $e_n$ is an approximate unit, we see that the second condition implies $s_n^*\phi(a)s_n\to\psi(a)$ for all $a\in A$.
Furthermore, the first condition yields for all $g\in G$ that
\[
\begin{array}{ccl}
\|s_n-\beta_{\infty,g}(s_n)\|^2 &=& \| (s_n-\beta_{\infty,g}(s_n))^*(s_n-\beta_{\infty,g}(s_n))\| \\
&\leq& \|e_n^2-e_n\alpha_g(e_n)-\alpha_g(e_n)e_n+\alpha_g(e_n^2)\| \ \to \ 0.
\end{array}
\]
Note that the intermediate inequality yields $s_n\in B_{\infty,\beta}$ for all $n$, and this computation implies that the convergence is uniform over compact sets.
By \autoref{lem:local-stability}, we may choose a separable $\beta_{\infty}$-invariant \cstar-subalgebra $D\subset B_{\infty,\beta}$ containing $\phi(A)\cup\psi(A)\cup\set{s_n}_{n\geq 1}\cup\set{b}$ such that $\beta_\infty|_D$ is strongly stable.
If we view $\phi$ and $\psi$ as equivariant maps into $D$, we see that the sequence $s_n\in D$ satisfies the requirements of \autoref{lem:absorption-step-1}.
Thus there exists a sequence of isometries $S_n\in\CM^{\beta_\infty}(D)$ such that
\[
S_n^*\phi(a)S_n\to \psi(a) \quad\text{and}\quad \max_{g\in K} \|S_n-\beta_{\infty,g}(S_n)\|\to 0
\]
for all $a\in A$ and compact sets $K\subseteq G$.
If we set $c_n=S_nb\in D$, then 
\[
c_n^*\psi(a)c_n\to b^*\psi(a)b \quad\text{and}\quad \max_{g\in K} \| c_n^*\beta_{\infty,g}(c_n) - b^*\beta_{\infty,g}(b)\|\to 0
\]
for all $a\in A$ and every compact set $K\subseteq G$.
Since $b\in B_{\infty,\beta}$ was arbitrary, this shows the claim.
\end{proof}

\begin{lemma}\label{lem:1domequicont}
Let $\alpha: G\curvearrowright A$ and $\beta: G\curvearrowright B$ be two actions on separable \cstar-algebras.
Assume that $\alpha$ is amenable and $\beta$ is strongly stable and isometrically shift-absorbing.
Let $\phi, \psi: (A,\alpha) \to (B_{\infty, \beta},\beta_\infty)$ be two equivariant $*$-homomorphisms.
Suppose that as (ordinary) maps into $B_\infty$, $\phi$ and $\psi$ are nuclear and $\phi$ approximately 1-dominates $\psi$. 
Then it follows that $(\phi, \eins)$ approximately 1-dominates $(\psi, \eins)$ as maps into $B_{\infty, \beta}$.
\end{lemma}
\begin{proof}
Let $d\in A$ be a contraction (we consider this independently to the other parameters that will now be fixed).
Let compact sets $1\in K\subseteq G$, $\CF\subset A_{\leq 1}$, and $\eps> 0$ be given. 
Since $\alpha$ is amenable, there exists a function $\zeta\in\CC_c(G,A)$ with $\|\zeta\|_2\leq 1$ such that
\begin{equation} \label{eq:aa-key-1}
\max_{a\in\CF} \big\| \langle\zeta\mid a\zeta \rangle - a \big\|\leq\eps;
\end{equation}
\begin{equation} \label{eq:aa-key-2}
\max_{g\in K} \big\| d^* \big( \eins-\langle \zeta \mid \bar{\alpha}_g(\zeta) \rangle \big) \|\leq\eps.
\end{equation}
By applying \autoref{lem:Kasparov} with $\alpha$ in place of $\beta$, we choose a positive contraction $e\in A$ with the following properties:
\begin{equation} \label{eq:aa-key-3}
\max_{g\in R} \|\alpha_g(e)-e\|\leq\eps , \qquad \textrm{where } R = \overline{\supp(\zeta)};
\end{equation}
\begin{equation} \label{eq:aa-key-4}
\max_{g\in K} \| \langle \zeta \mid e^2 \bar{\alpha}_g(\zeta) \rangle - e \langle \zeta \mid \bar{\alpha}_g(\zeta) \rangle e \|\leq\eps;
\end{equation}
\begin{equation} \label{eq:aa-key-5}
\max_{a\in \CF} \| \langle \zeta \mid e a e \zeta \rangle - e \langle \zeta \mid a \zeta \rangle e \|\leq\eps;
\end{equation}
\begin{equation}  \label{eq:aa-key-6}
\max_{g\in K}\|(\eins-e) \alpha_g(d) \|\leq\eps.
\end{equation}
As $\phi$ and $\psi$ are nuclear as maps into $B_\infty$, we can use the Choi--Effros lifting theorem \cite{ChoiEffros76} to pick completely positive contractive maps
\[
(\phi_k)_k, (\psi_k)_k : A \to \ell^\infty(B)
\]
that lift $\phi$ and $\psi$ respectively.
Since $\phi$ and $\psi$ take values in $B_{\infty, \beta}$ it follows that $(\phi_k)_k$ and $(\psi_k)_k$ take values in $\ell^\infty_\beta(B)$.  
Consider
\[
\CG_0 = \{ d,d^*, e, e^2\} \cup \mathcal F \cup \zeta(G) \cup\{ \langle \zeta, a \overline{\alpha}_g(\zeta)\rangle: a\in \CF, \, g\in K\}.
\]
Set $\CG_1 = \bigcup_{h\in K \cup R \cup R^{-1}} \alpha_h(\CG_0)$ and let $\CG$ be the set of products of 5 or less elements from $\CG_1$.
Then $\CG\subset A$ is compact.
Set $M=1+ \mu(R) (\| \zeta\|_\infty + \| \zeta\|_\infty^2)$. 
Using that $(\phi_k)_k$ and $(\psi_k)_k$ are pointwise $G$-equicontinuous, approximately equivariant approximate $*$-homo\-morphisms, we may pick $N\in \mathbb N$ such that 
\begin{equation}\label{eq:aa-new-1}
\max_{a\in \CG} \max_{g\in K \cup R \cup R^{-1}} \sup_{k\geq N}\| \beta_g(\phi_k(a)) - \phi_k(\alpha_g(a))\| \leq \eps/M;
\end{equation}
\begin{equation}\label{eq:aa-new-2}
\max_{a_1, a_2\in \CG} \sup_{k\geq N} \| \phi_k(a_1 a_2) - \phi_k(a_1)\phi_k(a_2)\| \leq \eps/M.
\end{equation}
We also assume that the above hold with $\psi_k$ in place $\phi_k$.
Note that for any $f\in \CC_c(G)$ one has the inequality $\| f\|_2 \leq \mu(\overline{\supp}(f)) \| f\|_\infty$.
Hence the estimates above with $\eps /M$ imply that
\begin{equation}\label{eq:aa-new-3}
\max_{g\in K} \sup_{k\geq N} \| \beta_g \circ \psi_k \circ f - \psi_k \circ \alpha_g \circ f\|_2 \leq \eps
\end{equation}
and
\begin{equation}\label{eq:aa-new-4}
\sup_{k\geq N} \| \psi_k \circ (F \cdot f) - (\psi_k \circ F) \cdot (\psi_k \circ f)\|_2 \leq \eps
\end{equation}
for every $f\in \CC_c(G, A)$ and $F \in \CC_b(G, A)$ with $f(G), F(G) \subseteq \CG$, $\mu(\overline{\supp} f) \leq \mu(R)$ and $\| f\|_\infty \leq \| \zeta\|_\infty$ and $\|F\|_\infty \leq \min\{1, \| \zeta\|_\infty\}$. 

Moreover, as $\phi$ approximately 1-dominates $\psi$ as maps into $B_\infty$, we may find a contraction $(c_k)_k \in \ell^\infty(B)$ such that after possibly increasing $N$ we have
\begin{equation}\label{eq:aa-ckphick}
\max_{a\in \CG} \sup_{k\geq N} \| c_k^\ast \phi_k(a) c_k - \psi_k(a)\| \leq \eps.
\end{equation}
Since the maps $\psi_k$ are completely positive contractive, we get from Kadison's inequality that for every $f\in \CC_c(G, A)$ one has
\[
\|\psi_k\circ f\|_2^2 = \Big\| \int_G \psi_k(f(h))^\ast \psi_k(f(h)) ~d\mu(h) \Big\| \leq \|\psi_k(\langle f, f \rangle )\| \leq \| f\|_2^2. 
\]
Since we assumed $\beta$ to be isometrically shift-absorbing, it follows from \autoref{prop:the-condition} that there exists an equivariant linear $B$-bimodule map 
\[
\theta: ( L^2(G,B), \bar{\beta}) \to ( B_{\infty,\beta} , \beta_\infty)
\] 
satisfying $\theta(\xi)^*\theta(\eta)=\langle\xi\mid\eta\rangle_B$ for all $\xi,\eta\in L^2(G,B)$.
We consider $\xi_k\in \CC_c(G,B)\subseteq L^2(G,B)$ via
\[
\xi_k(h)= \beta_h(\phi_k(e)c_k)\psi_k(\zeta(h)) ,\quad h\in G.
\]
Recall that we consider $B\subseteq B_{\infty, \beta}$ as constant sequences. We compute for every $g\in K$ and $k\geq N$ that %(reminding the reader that $\phi_k$ and $\psi_k$ have norm at most 2)
\[
\renewcommand{\arraystretch}{1.5}
\begin{array}{cl}
\multicolumn{2}{l}{ \dst\theta(\xi_k)^* \beta_{g,\infty}(\theta(\xi_k)) } \\ 
=& \theta(\xi_k)^*\theta(\bar{\beta}_g(\xi_k)) \\
=& \dst\langle \xi_k \mid \bar{\beta}_g(\xi_k)\rangle \\
=& \dst \int_G \psi_k(\zeta(h))^* \beta_h(c_k^* \phi_k(e)) \beta_g\Big( \beta_{g^{-1}h}(\phi_k(e) c_k) \psi_k(\zeta(g^{-1}h)) \Big) ~d\mu(h) \\
=& \dst \int_{G} \psi_k(\zeta(h))^* \beta_h(c_k^*\phi_k(e)^2 c_k) \beta_{g}(\psi_k(\zeta(g^{-1}h))) ~d\mu(h) \\ 
\stackrel{\eqref{eq:aa-new-3}}{=}_{\makebox[0pt]{\footnotesize\hspace{-2mm} $\eps$}}& \dst \int_{G} \psi_k(\zeta(h))^* \beta_h(c_k^*\phi_k(e)^2 c_k) \psi_k(\bar{\alpha}_g(\zeta)(h))) ~d\mu(h) \\
\stackrel{\eqref{eq:aa-new-2}}{=}_{\makebox[0pt]{\footnotesize\hspace{-2mm} $\eps$}}& \dst \int_{G} \psi_k(\zeta(h))^* \beta_h(c_k^*\phi_k(e^2) c_k) \psi_k(\bar{\alpha}_g(\zeta)(h))) ~d\mu(h) \\
\stackrel{\eqref{eq:aa-ckphick}}{=}_{\makebox[0pt]{\footnotesize\hspace{-3mm} $\eps$}}& \dst \int_{G} \psi_k(\zeta(h))^* \beta_h(\psi_k(e^2)) \psi_k(\bar{\alpha}_g(\zeta)(h)) ~d\mu(h)  \\
\stackrel{\eqref{eq:aa-new-1}}{=}_{\makebox[0pt]{\footnotesize\hspace{-2mm} $\eps$}}& \dst \int_{G} \psi_k(\zeta(h))^* \psi_k(\alpha_h(e^2)) \psi_k(\bar{\alpha}_g(\zeta)(h)) ~d\mu(h)  \\
\stackrel{\eqref{eq:aa-new-4}}{=}_{\makebox[0pt]{\footnotesize\hspace{-2mm} $2\eps$}}& \dst \int_{G} \psi_k\Big( \zeta(h)^* \alpha_h(e^2) \bar{\alpha}_g(\zeta)(h)\Big) ~d\mu(h)  \\
=& \dst \psi_k\Big( \int_{G} \zeta(h)^* \alpha_h(e^2) \bar{\alpha}_g(\zeta)(h) ~d\mu(h) \Big) \\
\stackrel{\eqref{eq:aa-key-3}}{=}_{\makebox[0pt]{\footnotesize\hspace{-2mm} $2\eps$}} & \dst \psi_k\Big( \int_{G} \zeta(h)^* e^2 \bar{\alpha}_g(\zeta)(h) ~d\mu(h) \Big) \\
=& \dst \psi_k\big( \big\langle \zeta \mid e^2\bar{\alpha}_g(\zeta) \big\rangle \big) \\
\stackrel{\eqref{eq:aa-key-4}}{=}_{\makebox[0pt]{\footnotesize\hspace{-2mm} $\eps$}} & \psi_k( e \big\langle \zeta\mid \bar{\alpha}_g(\zeta) \big\rangle e ). 
\end{array}
\]
Hence for $k\geq N$ and $g\in K$ we get
\[
\begin{array}{cl}
\multicolumn{2}{l}{ \dst \psi_k(d)^* \theta(\xi_k)^* \beta_{g,\infty}(\theta(\xi_k)\psi_k(d)) } \\
=& \dst \psi_k(d)^* \theta(\xi_k)^* \beta_{g,\infty}(\theta(\xi_k)) \beta_g(\psi_k(d)) \\
{=}_{\makebox[0pt]{\footnotesize\hspace{3mm} $9\eps$}} & \psi_k(d)^* \psi_k( e \big\langle \zeta\mid \bar{\alpha}_g(\zeta) \big\rangle e ) \beta_g(\psi_k(d)) \\
\stackrel{\eqref{eq:aa-new-1}}{=}_{\makebox[0pt]{\footnotesize\hspace{-2mm} $\eps$}} & \psi_k(d)^* \psi_k( e \big\langle \zeta\mid \bar{\alpha}_g(\zeta) \big\rangle e ) \psi_k(\alpha_g(d)) \\
\stackrel{\eqref{eq:aa-new-2}}{=}_{\makebox[0pt]{\footnotesize\hspace{-2mm} $2\eps$}} & \psi_k(d^*e \big\langle \zeta\mid \bar{\alpha}_g(\zeta) \big\rangle e \alpha_g(d)) \\
\stackrel{\eqref{eq:aa-key-6}}{=}_{\makebox[0pt]{\footnotesize\hspace{-1mm} $2\eps$}} & \psi_k\big( d^* \big\langle \zeta\mid \bar{\alpha}_g(\zeta) \big\rangle \alpha_g(d) \Big) \\
\stackrel{\eqref{eq:aa-key-2}}{=}_{\makebox[0pt]{\footnotesize\hspace{-3mm} $\eps$}} & \psi_k(d^*\alpha_g(d)).
\end{array}
\]
Moreover, for $a\in \CF$ and $k\geq N$ we have
\begin{longtable}{cl}
\multicolumn{2}{l}{ $\theta(\xi_k)^*\phi_k(a)\theta(\xi_k)$  } \\
$=$ & $\theta(\xi_k )^*\theta(\phi_k(a)\xi_k)$ \\
$=$ & $\dst \int_{G} \psi_k(\zeta(h))^* \beta_h(c_k^* \phi_k(e) \beta_{h^{-1}}(\phi_k(a)) \phi_k(e)c_k) \psi_k(\zeta(h)) ~d\mu(h) $ \\
$\stackrel{\eqref{eq:aa-new-1}}{=}_{\makebox[0pt]{\footnotesize\hspace{-2mm} $\eps$}}$ & $\dst \int_{G} \psi_k(\zeta(h))^* \beta_h(c_k^* \phi_k(e) \phi_k(\alpha_{h^{-1}}(a)) \phi_k(e)c_k) \psi_k(\zeta(h)) ~d\mu(h) $ \\
$\stackrel{\eqref{eq:aa-new-2}}{=}_{\makebox[0pt]{\footnotesize\hspace{-2mm} $2\eps$}}$ & $\dst \int_{G} \psi_k(\zeta(h))^* \beta_h\Big( c_k^*\phi_k( e\alpha_{h^{-1}}(a)e ) c_k \Big) \psi_k(\zeta(h)) ~d\mu(h) $ \\
$\stackrel{\eqref{eq:aa-ckphick}}{=}_{\makebox[0pt]{\footnotesize\hspace{-3mm} $\eps$}}$ & $\dst  \int_{G} \psi_k(\zeta(h))^* \beta_h\Big( \psi_k( e\alpha_{h^{-1}}(a)e )  \Big) \psi_k(\zeta(h)) ~d\mu(h) $ \\
$\stackrel{\eqref{eq:aa-key-3}}{=}_{\makebox[0pt]{\footnotesize\hspace{-1mm} $2\eps$}}$ & $\dst  \int_{G} \psi_k(\zeta(h))^* \beta_h(\psi_k(\alpha_{h^{-1}}(eae))) \psi_k(\zeta(h))~d\mu(h)$ \\
$\stackrel{\eqref{eq:aa-new-1}}{=}_{\makebox[0pt]{\footnotesize\hspace{-1mm} $\eps$}}$ & $\dst \int_{G} \psi_k(\zeta(h))^* \psi_k(eae) \psi_k(\zeta(h))~d\mu(h) $ \\
$\stackrel{\eqref{eq:aa-new-4}}{=}_{\makebox[0pt]{\footnotesize\hspace{-1mm} $2\eps$}}$ & $\dst \int_{G} \psi_k\Big( \zeta(h)^* eae \zeta(h)\Big)~d\mu(h) $ \\
$=$ & $\psi_k\Big( \big\langle \zeta \mid eae \zeta\rangle  \Big)$ \\
$\stackrel{\eqref{eq:aa-key-5}}{=}_{\makebox[0pt]{\footnotesize\hspace{-2mm} $\eps$}}$ & $\psi_k\Big( e \big\langle \zeta \mid a \zeta\rangle e  \Big).$
\end{longtable}
So for $a\in \CF$ and $k\geq N$ we get
\begin{longtable}{cl}
\multicolumn{2}{l}{ $ \psi_k(d)^*\theta(\xi_k)^*\phi_k(a)\theta(\xi_k)\psi_k(d)$  } \\%
${=}_{\makebox[0pt]{\footnotesize\hspace{3mm} $10\eps$}}$ & $\psi_k(d)^*\psi_k\Big( e \big\langle \zeta \mid a \zeta\rangle e  \Big) \psi_k(d) $ \\
$\stackrel{\eqref{eq:aa-new-2}}{=}_{\makebox[0pt]{\footnotesize\hspace{-1mm} $2\eps$}}$ & $\psi_k\Big( d^*e \big\langle \zeta \mid a \zeta\rangle e d \Big)$ \\
$\stackrel{\eqref{eq:aa-key-6}}{=}_{\makebox[0pt]{\footnotesize\hspace{-1mm} $2\eps$}}$ & $\psi_k\Big( d^* \big\langle \zeta \mid a \zeta\rangle d \Big)$ \\
$\stackrel{\eqref{eq:aa-key-1}}{=}_{\makebox[0pt]{\footnotesize\hspace{-2mm} $\eps$}}$ & $\psi_k(d^*ad).$
\end{longtable}
Hence we may for each $k\geq N$ lift $\theta(\xi_k)\psi_k(d)$ to a contraction $\ell^\infty_\beta(B)$ and pick an entry $z_k$ such that
\[
\sup_{k\geq N} \max_{g\in K} | \| z_k^\ast \beta_g(z_k) - \psi_k(d^\ast \alpha_g(d))\| | \leq 16 \eps
\]
and
\[
\sup_{k\geq N}\max_{a\in \CF} \| z_k^\ast \phi_k(a) z_k - \psi_k(d^*ad)\| \leq 16 \eps.
\]
By a diagonal argument (with respect to $\eps, \CF$ and $K$), we may find a contraction $s\in B_\infty$ such that 
\[
s^\ast \phi(a) s = \psi(d^*ad) ,\qquad  s^* \beta_{\infty, g}(s)  = \psi(d^* \alpha_g(d) )
\]
for all $a\in A$ and $g\in G$. 
As $d\in A$ was an arbitrary contraction, it follows from \autoref{lem:reduce-1-dom} that $(\phi, \eins)$ approximately 1-dominates $(\psi, \eins)$ 
\end{proof}

\begin{lemma} \label{lem:O2-absorbing-uniqueness}
Let $\alpha: G\curvearrowright A$ and $\beta: G\curvearrowright B$ be two actions on separable \cstar-algebras.
Suppose that $\alpha$ is amenable and $\beta$ is isometrically shift-absorbing, strongly stable and that $\beta\cc\beta\otimes\id_{\CO_2}$.
Let
\[
\phi, \psi: (A,\alpha)\to (B_{\infty,\beta},\beta_\infty)
\]
be two equivariant $*$-homomorphisms that are full and nuclear when considered as $*$-homomorphisms into $B_\infty$.
Then $(\phi,\eins)$ and $(\psi,\eins)$ are properly unitarily equivalent.
\end{lemma}
\begin{proof}
Since $B$ is $\CO_\infty$-stable it follows that $B_\infty$ is strongly purely infinite by \cite[Proposition 5.12, Theorem 8.6]{KirchbergRordam02}. 
The $*$-homomorphisms $\phi, \psi : A \to B_\infty$ are full and nuclear, and thus they approximately 1-dominate each other by \cite[Corollary 3.13, Theorem 4.8, Proposition 9.4]{Gabe21}.

By \autoref{lem:1domequicont} it follows that $(\phi, \eins)$ and $(\psi, \eins)$ approximately 1-dominate each other as maps into $B_{\infty , \beta}$. 
Pick a separable, $\beta_\infty$-invariant \cstar-subalgebra $D\subseteq B_{\infty, \beta}$ containing the images of $\phi$ and $\psi$, and such that $(\phi, \eins)$ and $(\psi, \eins)$ approximately 1-dominate each other when corestricted to $D$.
By \autoref{lem:local-stability} we may assume that $\beta_\infty|_D$ is strongly stable. 
Moreover, as $\beta \cc \beta\otimes \id_{\CO_2}$ it follows (see \cite[Lemma 2.12]{Szabo18ssa2}) that $\CO_2$ unitally embeds into $F(D,B_{\infty,\beta})^{\tilde \beta_\infty}$.\footnote{Here $F(D,B_{\infty,\beta})=(B_{\infty,\beta}\cap D')/(B_{\infty,\beta}\cap D^\perp)$ and $\tilde{\beta}_\infty$ is the action induced by $\beta$.}
There is a canonical commutative diagram of equivariant $*$-homomorphisms given by
\[
\xymatrix{
(D,\beta_\infty) \ar[rd]^{\id_D\otimes\eins} \ar[rr] && ( B_{\infty,\beta}, \beta_\infty ) \\
& (D\otimes_{\max} F(D,B_{\infty,\beta}), \beta_\infty \otimes \tilde{\beta}_\infty ) \ar[ur] &
}
\]
Hence we may enlarge $D$ and thus assume that there is a unital inclusion $\CO_2 \to \CM(D)^{\beta_\infty}$ commuting with the images of $\phi$ and $\psi$.  
By applying \autoref{lem:strong-sum-absorption} twice, we get that $(\phi,\eins)$ and $(\psi,\eins)$ are strongly asymptotically unitarily equivalent. 
As the cocycles involved are trivial, the unitary path implementing this equivalence is asymptotically $\beta_\infty|_D$-invariant. 
By performing a standard diagonal sequence argument within $B_{\infty}$, it is a routine consequence that $(\phi,\eins)$ and $(\psi,\eins)$ are properly unitarily equivalent as maps into $B_\infty$ implemented by a $\beta_\infty$-invariant unitary, and thus they are properly unitarily equivalent as equivariant maps into $B_{\infty, \beta}$. 
\end{proof}

\begin{rem}\label{rem:fullseq}
For the purpose of subsequent applications of \autoref{lem:O2-absorbing-uniqueness}, we shall point out in more explicit terms when $\ast$-homomorphisms $\phi : A \to B_\infty$ are full and nuclear, for $A$ exact and $B$ simple and purely infinite.
Assuming that $A$ is exact, it follows from \cite[Proposition 3.3]{Dadarlat97} and the Choi--Effros lifting theorem that $\phi$ is nuclear if and only if it can be represented by a sequence of maps $(\phi_n)_n$ where each $\phi_n : A \to B$ is a nuclear c.p.c.\ map.

To describe full maps, we first describe full elements of $B_\infty$.
It is obvious that for $x = [(x_n)_n] \in B_\infty$ to be full, it must satisfy $\liminf_{n\to \infty} \| x_n\| >0$.
Moreover, when $B$ is simple and purely infinite, this characterizes full elements in $B_\infty$. This follows easily by the same argument as \cite[Proposition 6.2.6]{Rordam}, which showed that the ultrapower $B_\omega$ is simple and purely infinite.

In particular, if $A$ is exact and $B$ is a Kirchberg algebra, then a $\ast$-homomorphism $\phi : A \to B_\infty$ is full and nuclear if and only if it can be represented by a sequence of c.p.c.~maps $\phi_n : A \to B$ for $n\in \mathbb N$ such that $\liminf_{n\to \infty} \|\phi_n(a)\| >0$ for every non-zero $a\in A$. 
In a situation where the domain $A$ happens to be given as a minimal tensor product of two \cstar-algebras, it follows from \cite[Lemma 4.1.9]{Rordam} that it suffices to check fullness on elementary tensors.
We apply this observation in the proof of \autoref{lem:equi-O2-embedding-step1}.
\end{rem}

\begin{rem} \label{rem:embed-compacts-into-gamma}
We claim that there exists an equivariant embedding 
\[
(\CK(\CH_G^\infty),\ad(\lambda^\infty))\to(\CO_\infty,\gamma).
\]
Indeed, let $\Fs: \CH_G^\infty\to\CO_\infty$ be the universal linear map whose range generates $\CO_\infty$, and which satisfies $\Fs\circ\lambda^\infty_g=\gamma_g\circ \Fs$ for all $g\in G$.
Then we obtain a $*$-homomorphism $\iota: \CK(\CH_G^\infty)\to\CO_\infty$ that is uniquely determined on the set of rank one operators via the formula
\[
\iota(E_{\xi,\eta})=\Fs(\xi)\Fs(\eta)^*,\quad\text{where } \xi,\eta\in\CH_G^\infty,\ \xi, \eta \neq 0, \text{ and } E_{\xi,\eta}(\nu):=\xi \cdot \langle \eta\mid\nu\rangle.
\]
Evidently $\iota$ is equivariant and hence defines the desired inclusion.
\end{rem}

\begin{rem} \label{rem:suspension-crossed-product}
Let $\sigma$ be the shift automorphism on $\CC_0(\IR)$ given by $\sigma(f)(t)=f(t+1)$.
It is well-known that $\CC_0(\IR)\rtimes_\sigma\IZ$ is isomorphic to $\CC(\IT)\otimes\CK$.
If $A$ is a \cstar-algebra, we denote for brevity $SA=\CC_0(\IR)\otimes A$.
Then there is a natural isomorphism 
\[
SA\rtimes_{\sigma\otimes\id_A}\IZ\cong A\otimes (\CC_0(\IR)\rtimes_\sigma\IZ)\cong A\otimes\CC(\IT)\otimes\CK.
\]
If $\alpha: G\curvearrowright A$ is an action, then $S\alpha=\id_{\CC_0(\IR)}\otimes\alpha: G\curvearrowright SA$ commutes with $\sigma\otimes\id_A$, so naturality of this isomorphism entails that the action $\alpha\otimes\id$ on the right-hand side becomes the unique $G$-action on the left-hand side that extends $\id_{\CC_0(\IR)}\otimes\alpha$ by acting trivially on the copy of $\IZ$.
In particular, $A$ embeds equivariantly as a full corner into both sides of this isomorphism, which we will use below.
\end{rem}

\begin{lemma} \label{lem:equi-O2-embedding-step1}
Let $A$ be a separable exact \cstar-algebra with an action $\alpha: G\curvearrowright A$.
Let $\beta: G\curvearrowright B$ be an isometrically shift-absorbing action on a Kirchberg algebra with $\beta\cc\beta\otimes\id_{\CO_\infty}$. 
Then there exists an equivariant $*$-homomorphism from $(SA,S\alpha)$ to $(B_{\infty,\beta},\beta_\infty)$ which is full and nuclear as a map into $B_\infty$.
\end{lemma}
\begin{proof}
Since $A$ is separable, we find a faithful representation $\pi: A\to\CB(\ell^2(\IN))$.
The canonically induced covariant representation of $(A,\alpha)$ then consists of the unitary representation $\lambda^\infty=\eins_{\ell^2(\IN)}\otimes\lambda: G\to \CU(\ell^2(\IN)\hat{\otimes}\CH_G)$ and the representation $\pi^\alpha: A\to \CB(\ell^2(\IN)\hat{\otimes}\CH_G)=\CB(L^2(G,\ell^2(\IN)))$ given by $\pi^\alpha(a)(\xi)(g)=\pi(\alpha_{g^{-1}}(a))\xi(g)$ for all $\xi\in L^2(G,\ell^2(\IN))$ and $g\in G$.
We may apply \autoref{lem:Kasparov} and choose an approximate unit $h_n^{(1)}\in\CK := \CK(\ell^2(\IN)\hat{\otimes}\CH_G)$ such that
\[
\lim_{n\to\infty}\max_{g\in K} \|\ad(\lambda^\infty_g)(h_n^{(1)})-h_n^{(1)}\|=0,\quad \lim_{n\to\infty} \|[h_n^{(1)},\pi^\alpha(a)]\|=0
\]
for every compact set $K\subseteq G$ and all $a\in A$.
Using that $\beta\cc\beta\otimes\id_{\CO_\infty}$ and that $(B,\beta)$ also has an approximately $\beta$-invariant approximate unit, we may find a sequence of positive elements $h_n^{(2)}\in B$ with full spectrum $[0,1]$ and $\lim_{n\to\infty}\max_{g\in K} \|\beta_g(h_n^{(2)})-h_n^{(2)}\|=0$ for every compact set $K\subseteq G$. Set $B^s=B\otimes\CK$ and $\beta^s=\beta\otimes\ad(\lambda^\infty)$.
We define the sequence $h_n=h_n^{(2)}\otimes h_n^{(1)}\in B^s$ and consider the element $h=(h_n)_n$ in the sequence algebra of $B^s$.
It is then fixed under the induced action $\beta^s_\infty$ and commutes with the range of $\eins_{\CM(B)}\otimes\pi^\alpha$ if we view the latter as an equivariant $*$-homomorphism into $\CM(B^s)$.
We obtain an equivariant $*$-homomorphism
\[
\psi: (SA,S\alpha) \to (B^s_{\infty, \beta^s}, \beta^s_\infty),\quad \psi(f\otimes a)=f(h)\cdot (\eins_{\CM(B)}\otimes\pi^\alpha(a))
\]
for all $f\in\CC_0(0,1)\cong\CC_0(\IR)$ and $a\in A$.
For any non-zero $a\in A$ and $f\in \CC_0(0,1)$, we may pick a contraction $x\in \CK$ such that $\pi^\alpha(a) x \neq 0$. Since $h_n^{(1)}$ approximately acts as a unit on $\pi^\alpha(a) x$, we get that 
\[
 \psi(f\otimes a) \cdot (\eins_{\CM(B)}\otimes x) = [(f(h_n^{(2)}) \otimes \pi^{\alpha}(a) x )_{n\in \mathbb N}] \in B_\infty^s.
\] 
Because 
\[
\liminf_{n\to \infty} \| f(h_n^{(2)}) \otimes \pi^{\alpha}(a) x \| = \liminf_{n\to \infty} \| f(h_n^{(2)}) \| \| \pi^{\alpha}(a)x\| >0,
\]
it follows from \autoref{rem:fullseq} that $\psi$ is a full map into $B^s_\infty$ (since we verified fullness on elementary tensors). 
Since $A$ is exact, the representation $\pi^\alpha$ is nuclear, so it follows by \cite[Lemma 6.9]{Gabe20} that $\psi$ is also nuclear. Hence it can be represented by a sequence $(\psi_n : SA \to B^s)$ of c.p.c.~maps such that $\liminf_{n\to \infty} \| \psi_n (c)\| >0$ for all non-zero $c\in SA$. 

Since $\beta$ is assumed to be isometrically shift-absorbing, we can conclude from \autoref{prop:the-condition} that there exists an equivariant $*$-homomorphism $\Theta : (B\otimes\CO_\infty,\beta\otimes\gamma) \to (B_{\infty,\beta},\beta_\infty)$, which is (automatically) nuclear and full as a map into $B_\infty$. 
By \autoref{rem:embed-compacts-into-gamma} $(B\otimes\CO_\infty,\beta\otimes\gamma)$ contains $(B^s, \beta^s)$ as a subsystem. Use \autoref{rem:fullseq} to represent $\Theta|_{B^s}$ by a sequence of c.p.c.~maps $\theta_m: B^s \to B$ such that $\liminf_{m\to \infty} \|\theta_m(b) \| >0$ for all non-zero $b\in B^s$.

By a standard diagonal argument applied to the c.p.c. maps $\theta_m \circ \psi_n : SA \to B$ we obtain an equivariant $*$-homomorphism from $(SA,S\alpha)$ to $(B_{\infty,\beta},\beta_\infty)$ which is full and nuclear as a map into $B_\infty$ by \autoref{rem:fullseq}.
\end{proof}

\begin{theorem} \label{thm:equi-O2-embedding}
Let $A$ be a separable exact \cstar-algebra with an amenable action $\alpha: G\curvearrowright A$.
Suppose that $\beta: G\curvearrowright B$ is an isometrically shift-absorbing action on a Kirchberg algebra with $\beta\cc\beta\otimes\id_{\CO_2}$.
Then there exists a proper cocycle embedding from $(A,\alpha)$ to $(B,\beta)$.
\end{theorem}
\begin{proof}
We note first that we may assume without loss of generality that $\beta$ is strongly stable.
This is because $\beta$ being cocycle conjugate to $\beta\otimes\id_{\CO_2}$ automatically implies that it is properly cocycle conjugate to $\beta\otimes\id_{\CO_2}$ by \autoref{thm:strong-ssa-abs}.
Since one can easily construct an inclusion $\CK\subseteq\CO_2$, this provides a proper cocycle embedding from $\beta\otimes\id_{\CK}$ to $\beta$, hence we assume from now on that $\beta$ is strongly stable.

We apply \autoref{lem:equi-O2-embedding-step1} and choose an equivariant $*$-homom\-orphism $\psi: (SA,S\alpha) \to (B_{\infty,\beta},\beta_\infty)$ that is full and nuclear as a map into $B_\infty$.
Recall the notation introduced in \autoref{rem:suspension-crossed-product}.
Due to our assumptions on $\alpha$ and $\beta$, the assumptions of \autoref{lem:O2-absorbing-uniqueness} are satisfied, so it follows that $\psi$ and $\psi\circ(\sigma\otimes\id_A)$ are properly unitarily equivalent.
In other words, there exists a unitary in the fixed point algebra $U\in\CU(\eins+(B_{\infty,\beta})^{\beta_\infty})$ with $\ad(U)\circ\psi=\psi\circ(\sigma\otimes\id_A)$.
By the universal property of the crossed product we obtain a $*$-homomorphism
\[
\bar{\psi}: SA\rtimes_{\sigma\otimes\id_A}\IZ \to B_{\infty,\beta} \quad\text{with}\quad \bar{\psi}|_{SA}=\psi.
\]
By construction $\bar{\psi}$ is also equivariant with respect to the obvious extension of $\id_{\CC_0(\IR)}\otimes\alpha$ on the left side and $\beta_\infty$ on the right.
We have that $\bar{\psi}$ is nuclear into $B_\infty$ by \cite[Lemma 6.10]{Gabe20}.
In light of \autoref{rem:suspension-crossed-product}, we can restrict $\bar{\psi}$ to a full corner that is equivariantly isomorphic to $(A,\alpha)$, and thereby obtain an equivariant $*$-homomor\-phism
\[
\phi: (A,\alpha) \to (B_{\infty,\beta},\beta_\infty)
\]
that is nuclear as a map into $B_\infty$.
Since $\phi(a)$ generates a larger closed ideal than $\psi(f\otimes a)$ for any $f\in \CC_0(\IR)$, it follows that $\phi$ is full as a map into $B_\infty$.
Let $\kappa: \IN\to\IN$ be an arbitrary map with $\lim_{n\to\infty}\kappa(n)=\infty$.
Then we obtain an equivariant endomorphism $\kappa^*$ of $B_{\infty,\beta}$ given at the level of representing sequences by $\kappa^*[ (b_n)_n] = [(b_{\kappa(n)})_n]$.
By \autoref{rem:fullseq} it follows that $\kappa^*$ has the property that it maps full elements in $B_\infty$ to full elements.
Therefore, we have that the composition $\kappa^*\circ\phi$ is also an equivariant $*$-homomorphism from $(A,\alpha)$ to $(B_{\infty,\beta},\beta_\infty)$ that is full and nuclear as a map into $B_\infty$.
By the assumptions on $\alpha$ and $\beta$, the assumptions of \autoref{lem:O2-absorbing-uniqueness} are satisfied and we can conclude that $\phi$ and $\kappa^*\circ\phi$ are properly unitarily equivalent.
Since $\kappa$ was arbitrary, the existence of the claimed proper cocycle embedding follows directly from the one-sided intertwining result \cite[Theorem 4.10]{Szabo21cc}.
\end{proof}

\begin{rem} \label{rem:exact-case}
We note that if $G$ is exact, then \autoref{thm:equi-O2-embedding} holds without the assumption that $\alpha$ is amenable.
This is because by exactness, we may find some amenable action $\delta: G\curvearrowright D$ on a separable unital \cstar-algebra, and hence embed $\alpha\otimes\delta$ in place of $\alpha$.
In contrast, if $G$ is not exact, then \autoref{thm:equi-O2-embedding} fails even for $\alpha$ being the trivial action on $\IC$.
Namely, if we exploit \cite{OzawaSuzuki21}, we can find an action $\beta$ as in the statement that is also amenable, which rules out the possibility that any of its cocycle perturbations fix a non-zero projection.
\end{rem}

%%%%%%%%%%%%%%%%%%%%%%%%%%%%%%%%%

\section{Existence and uniqueness theorems}

The following observation by the first author is central to both the existence and uniqueness theorem proved in this section.

\begin{lemma}[see {\cite[Lemma 7.1]{Gabe21}}] \label{lem:special-unitary-path}
There exists a continuous map $u: [0,\infty)\to\CU(\eins+\CO_2\otimes\CK)$ with $u_0=\eins$ such that
\begin{enumerate}[label=\textup{(\roman*)}]
\item $u_t^*(\eins\otimes e_{1,1})u_t\to \eins$ in the strict topology as $t\to\infty$;
\item for all $x\in\CO_2\otimes\CK$, one has that $u_tx$ converges in norm as $t\to\infty$.
\end{enumerate}
\end{lemma}

\begin{lemma} \label{lem:non-standard-embedding}
Let $B$ be a \cstar-algebra and $\beta: G\curvearrowright B$ a strongly stable action.
Then there exists a non-degenerate equivariant embedding from $(B\otimes\CO_2,\beta\otimes\id_{\CO_2})$ to $(B\otimes\CO_\infty,\beta\otimes\id_{\CO_\infty})$
\end{lemma}
\begin{proof}
Since $\beta$ is (genuinely) conjugate to $\beta \otimes \id_{\CK}$, it suffices to show that there exists a non-degenerate embedding $\CO_2\otimes \CK \to \CO_\infty\otimes \mathcal K$.
Such an embedding is known to exist, for instance as a consequence of Brown's stable isomorphism theorem \cite{Brown77}.
\end{proof}

\begin{cor} \label{cor:special-relative-O2-embedding}
Let $\alpha: G\curvearrowright A$ and $\beta: G\curvearrowright B$ be actions on \cstar-algebras
Suppose that $\beta$ is strongly stable and conjugate to $\beta\otimes\id_{\CO_\infty}$.
Suppose that there exists a proper cocycle embedding from $(A,\alpha)$ into $(B\otimes\CO_2,\beta\otimes\id_{\CO_2})$.
Then there exists a proper cocycle embedding $(\theta,\Iy): (A,\alpha)\to (B,\beta)$ along with a unital embedding $\iota_0: \CO_2\to \CM(B)^{\beta}$ whose range commutes with the ranges of $\theta$ and $\Iy$.
\end{cor}
\begin{proof}
By assumption there exists a proper cocycle embedding from $(A,\alpha)$ to $(B\otimes\CO_2,\beta\otimes\id_{\CO_2})$.
If we use an isomorphism $\CO_2\cong\CO_2\otimes\CO_2$, we may choose one whose range commutes pointwise with the range of some unital embedding $\iota_0: \CO_2\to \CM(B\otimes\CO_2)^{\beta\otimes\id_{\CO_2}}$.
By \autoref{lem:non-standard-embedding} and the assumption that $\beta$ is conjugate to $\beta\otimes\id_{\CO_\infty}$, there exists a non-degenerate equivariant embedding from $(B\otimes\CO_2,\beta\otimes\id_{\CO_2})$ to $(B,\beta)$, so the claim follows.
\end{proof}

The next remark slightly extends some terminology from the first section.

\begin{rem}[cf.\ \autoref{def:KKG}] \label{rem:KKGEG}
Suppose $\beta$ is strongly stable.
Given some $x\in\IE^G(\alpha,\beta)$, we denote its homotopy class by $[x]_h$.
More specifically, if we are given a proper cocycle morphism $(\phi,\Iu): (A,\alpha)\to (B,\beta)$, then we naturally associate to it the Cuntz pair $\big( (\phi,\Iu), (0,\eins) \big)$.
Its homotopy class in $\IE^G(\alpha,\beta)/{\sim_h}$ is also denoted $[(\phi,\Iu)]_h$.
We call it \emph{anchored} when $(\Iu,\eins)\sim_h (\eins,\eins)$.

Note that any anchored proper cocycle conjugacy $(\psi, \Iv) : (B, \beta) \to (C, \gamma)$ induces an isomorphism $(\psi, \Iv)_* : \IE^G(\alpha,\beta)/{\sim_h} \to \IE^G(\alpha,\gamma)/{\sim_h}$ such that $(\psi, \Iv)_* ([(\phi, \Iu)]_h) = [(\psi, \Iv)\circ (\phi, \Iu)]_h$ for all anchored proper cocycle morphisms $(\phi,\Iu): (A,\alpha)\to (B,\beta)$.
\end{rem}

We shall now prove the existence and uniqueness theorems underpinning our classification theory.
We note that the combination of \autoref{thm:existence} and \autoref{thm:uniqueness} proves \autoref{intro:existence-uniqueness}.

\subsection{Existence}

The following is the existence theorem underpinning our classification theory.
Its strategy of proof can be regarded as the dynamical generalization of \cite[Lemma 7.3]{Gabe21}.

\begin{theorem} \label{thm:existence}
Let $A$ be a separable exact \cstar-algebra with an action $\alpha: G\curvearrowright A$.
Suppose that $G$ is exact or that $\alpha$ is amenable.
Let $B$ be a Kirchberg algebra and $\beta: G\curvearrowright B$ a strongly stable, amenable and isometrically shift-absorbing action.
Then:
	\begin{enumerate}[leftmargin=*,label=\textup{(\roman*)}]
	\item For every $z\in \IE^G(\alpha,\beta)/{\sim_h}$, there exists a proper cocycle embedding $(\phi,\Iu): (A,\alpha)\to (B,\beta)$ such that $[(\phi,\Iu)]_h=z$. \label{thm:existence:1}
	\item For every $z\in KK^G(\alpha,\beta)$, there exists an anchored proper cocycle embedding $(\phi,\Iu): (A,\alpha)\to (B,\beta)$ with $KK^G(\phi,\Iu)=z$.\label{thm:existence:2}
	\end{enumerate}
\end{theorem}
\begin{proof}
We note that \ref{thm:existence:2} can be viewed as a special case of \ref{thm:existence:1} in light of \autoref{def:KKG}, hence we only need to prove the first part.

We know from \autoref{prop:isa-implies-absorption} that $\beta\cc\beta\otimes\id_{\CO_\infty}$.
By \autoref{thm:strong-ssa-abs}, it follows in fact that the first factor embedding $(B,\beta) \to (B\otimes \CO_\infty, \beta \otimes \id_{\CO_\infty})$ is strongly asympotically unitarily equivalent to a properly cocycle conjugacy, which is anchored since the first factor embedding is trivially anchored.
We may therefore, without any loss of generality, replace $\beta$ by $\beta\otimes\id_{\CO_\infty}$ in our claim and assume that $\beta$ is conjugate to $\beta\otimes\id_{\CO_\infty}$. 
Depending on whether $G$ is exact or not, we argue as in \autoref{rem:exact-case}, and apply \autoref{cor:special-relative-O2-embedding} and \autoref{thm:equi-O2-embedding} with $\beta\otimes\id_{\CO_2}$ in place of $\beta$.
This allows us to find a proper cocycle embedding $(\theta,\Iy): (A,\alpha)\to (B,\beta)$ and a unital embedding $\iota_0: \CO_2\to \CM(B)^{\beta}$ whose range commutes with $\theta$ and $\Iy$.

Using that $\beta$ is strongly stable, let us choose a sequence of isometries $r_n\in\CM(B)^\beta$ such that $\sum_{n=1}^\infty r_nr_n^*=\eins$ in the strict topology.
As before we let $\set{e_{k,\ell}\mid k,\ell\geq 1}$ be a set of matrix units generating $\CK$.
We then consider the non-degenerate embedding
\[
\iota: \CO_2\otimes\CK\to\CM(B)^\beta,\quad \iota(a\otimes e_{k,\ell})=r_k\iota_0(a)r_\ell^*,\quad k,\ell\geq 1.
\]
Using \autoref{lem:special-unitary-path}, we may find a continuous unitary path $u: [0,\infty)\to\CU(\eins+\iota(\CO_2\otimes\CK))\subseteq\CU(\CM(B)^\beta)$ with $u_0=\eins$ such that 
\begin{equation}\label{eq:utr1r1ut}
\eins=\lim_{t\to\infty} u_t^*\iota(\eins\otimes e_{1,1})u_t=\lim_{t\to\infty} u_t^* r_1r_1^* u_t
\end{equation}
holds in the strict topology as $t\to\infty$, and moreover $u_tb$ converges in norm as $t\to\infty$ for all $b\in B$.

By \autoref{cor:inf-repeats-are-absorbing}, we have that $(\theta^\infty,\Iy^\infty)$ is an absorbing cocycle representation, where we note that the infinite repeat is meant to be formed via the sequence $r_n$.
This ensures that the range of $\iota$, and therefore also the range of $u$, commutes pointwise with the range of $\theta^\infty$ and $\Iy^\infty$.
It follows from \cite[Corollary 3.17]{GabeSzabo22} that we can find some cocycle representation $(\psi,\Iv): (A,\alpha)\to(\CM(B),\beta)$ that forms an $(\alpha,\beta)$-Cuntz pair together with $(\theta^\infty,\Iy^\infty)$ such that one has $z=\big[ (\psi,\Iv), (\theta^\infty,\Iy^\infty) \big]_h$.

By the properties of the unitary path $u$ constructed above, we have for all $a\in A$ that
\[
u_t\psi(a)u_t^* = u_t(\underbrace{\psi(a)-\theta^\infty(a)}_{\in B})u_t^*+\theta^\infty(a)
\]
converges in norm as $t\to\infty$.
Let $\phi'=\lim_{t\to\infty} \ad(u_t)\circ\psi$ be the $*$-homomorphism arising as the point-norm limit.
We observe for every $a\in A$ 
\[
\begin{array}{ccl}
\|(\eins-r_1r_1^*)(\phi'(a)-\theta^\infty(a))\| &=& \dst\lim_{t\to\infty}\|(\eins-r_1r_1^*)(u_t\psi(a)u_t^*-\theta^\infty(a))\| \\
&=& \dst\lim_{t\to\infty} \|(\eins-u_t^*r_1r_1^*u_t)(\psi(a)-\theta^\infty(a))\| \\
&\stackrel{\eqref{eq:utr1r1ut}}{=}& 0.
\end{array}
\]
Let us consider the isometry $r_\infty=\sum_{n=1}^\infty r_{n+1}r_n^*\in\CM(B)^\beta$, which has the property that $r_1 r_1^* + r_\infty r_\infty^*=\eins$.
Since all partial isometries of the form $r_kr_\ell^*$ commute with the range of $\theta^\infty$, we can conclude that the projection $r_1r_1^*$ commutes with the range of $\phi'$.
Hence it follows for all $a\in A$ that 
\[
\begin{array}{ccl}
\phi'(a) &=& r_1r_1^*\phi'(a)+(\eins-r_1r_1^*)\phi'(a) \\
&=& r_1r_1^*\phi'(a)+(\eins-r_1r_1^*)\theta^\infty(a) \\
&=& \phi(a)\oplus_{r_1,r_\infty}\theta^\infty(a),
\end{array}
\]
where $\phi: A\to\CM(B)$ is the $*$-homomorphism defined as $\phi(a)=r_1^*\phi'(a)r_1$.
Appealing to the properties of the unitary path $u$ once more, we have for all $g\in G$ that
\[
u_t\Iv_gu_t^*=u_t(\underbrace{\Iv_g-\Iy^\infty_g}_{\in B})u_t^*+\Iy_g^\infty
\]
converges in norm as $t\to\infty$.
Since the pointwise difference $\Iv-\Iy^\infty$ is norm-continuous (see \cite[Proposition 6.9]{Szabo21cc}), this convergence is uniform over compact subsets of $G$.
Let $\Iu'_\bullet=\lim_{t\to\infty} u_t\Iv_\bullet u_t^*$ be the $\beta$-cocycle arising as the pointwise limit in norm.
We observe for every $g\in G$ that
\[
\begin{array}{ccl}
\|(\eins-r_1r_1^*)(\Iu'_g-\Iy^\infty_g)\| &=& \dst\lim_{t\to\infty}\|(\eins-r_1r_1^*)(u_t\Iv_gu_t^*-\Iy^\infty_g)\| \\
&=& \dst\lim_{t\to\infty} \|(\eins-u_t^*r_1r_1^*u_t)(\Iv_g-\Iy^\infty_g)\| \\
&\stackrel{\eqref{eq:utr1r1ut}}{=}& 0.
\end{array}
\]
Since all partial isometries of the form $r_kr_\ell^*$ commute commute with the range of $\Iy^\infty$, we can conclude that the projection $r_1r_1^*$ commutes with the range of $\Iu'$.
Hence it follows for all $g\in G$ that 
\[
\renewcommand\arraystretch{1.25}
\begin{array}{ccl}
\Iu'_g &=& r_1r_1^*\Iu'_g + (\eins-r_1r_1^*)\Iu'_g \\
&=& r_1r_1^*\Iu'_g + (\eins-r_1r_1^*)\Iy^\infty_g \\
&=& \Iu_g\oplus_{r_1,r_\infty}\Iy^\infty_g,
\end{array}
\]
where $\Iu: G\to\CU(\CM(B))$ is the $\beta$-cocycle defined as $\Iu_g=r_1^*\Iu'_gr_1$.
In conclusion, we have constructed a cocycle representation $(\phi,\Iu): (A,\alpha)\to (\CM(B),\beta)$ such that the two $(\alpha,\beta)$-Cuntz pairs
\[
\big( (\phi,\Iu)\oplus_{r_1,r_\infty}(\theta^\infty,\Iy^\infty), (\theta^\infty,\Iy^\infty) \big) \quad\text{and}\quad \big( (\psi,\Iv), (\theta^\infty,\Iy^\infty) \big)
\]
are homotopic.
Note that our choice of isometries to define the Cuntz sum leads to the equation $(\theta^\infty,\Iy^\infty)=(\theta,\Iy)\oplus_{r_1,r_\infty}(\theta^\infty,\Iy^\infty)$.
In particular, we see that $(\phi,\Iu)$ and $(\theta,\Iy)$ also necessarily form an $(\alpha,\beta)$-Cuntz pair representing the class $z$.
Since $(\theta,\Iy)$ was a proper cocycle morphism, so is hence $(\phi,\Iu)$.
By construction (and \autoref{def:KKG}) it is clear that $[(\theta,\Iy)]_h=0$, so we may conlude that
\[ 
z=\big[ (\phi,\Iu), (\theta,\Iy) \big]_h=[(\phi,\Iu)]_h-[(\theta,\Iy)]_h=[(\phi,\Iu)]_h.
\]
Finally, if $\phi$ is not an embedding, we may replace $(\phi,\Iu)$ by $(\phi,\Iu)\oplus (\theta,\Iy)$, the homotopy class of which also equals $z$.
This finishes the proof.
\end{proof}

\begin{theorem} \label{thm:unital-existence}
Suppose $G$ is exact.\footnote{When $G$ is non-exact, an action $\beta$ as in this theorem cannot exist.}
Let $A$ be a separable exact unital \cstar-algebra with an action $\alpha: G\curvearrowright A$.
Let $B$ be a unital Kirchberg algebra and $\beta: G\curvearrowright B$ an amenable and  isometrically shift-absorbing action.
Then for every $x\in KK^G(\alpha,\beta)$ with $[\eins_A]_0\otimes x=[\eins_B]_0\in K_0(B)$\footnote{To make sense of this formula, we can view $x\in KK(A,B)$ via the forgetful map. The canonical identification $K_0(\_)\cong KK(\IC,\_)$ and the Kasparov product allow us to make sense of this compatibility formula of $x$ with the $K_0$-classes of the unit elements.},
there exists a unital cocycle embedding $(\psi,\Iv): (A,\alpha)\to (B,\beta)$ such that $KK^G(\psi,\Iv)=x$.
\end{theorem}
\begin{proof}
Denote $B^s=B\otimes\CK$ and $\beta^s=\beta\otimes\id_\CK$.
Let us consider the invertible element $\kappa\in KK^G(\beta,\beta^s)$ given by the canonical inclusion $B\subseteq B^s$, $b\mapsto b\otimes e_{1,1}$.
By \autoref{thm:existence}, we can find a proper cocycle embedding $(\phi,\Iu): (A,\alpha)\to (B^s,\beta^s)$ such that $KK^G(\phi,\Iu)=x\otimes\kappa$.
By the assumptions on $x$, it follows that the projections $\phi(\eins_A)$ and $\eins_B\otimes e_{1,1}$ represent the same $K_0$-class.
Since $B^s$ is a stable Kirchberg algebra, these projections are unitarily equivalent, so we find a unitary $U\in\CU(\eins+B^s)$ with $U\phi(\eins)U^*=\eins_B\otimes e_{1,1}$.
Thus $\ad(U)\circ(\phi,\Iu): (A,\alpha)\to(B^s,\beta^s)$ is of the form
$(\psi\otimes e_{1,1},\Iv\otimes e_{1,1}+\Iv')$
for a unital cocycle embedding $(\psi,\Iv): (A,\alpha)\to (B,\beta)$ and a $\beta$-cocycle $\Iv'$ with values in $(\eins-e_{1,1})+(\eins-e_{1,1})B^s(\eins-e_{1,1})$.
Using \cite[Proposition 6.14]{Szabo21cc}, it follows that indeed
\[
KK^G(\psi,\Iv)=KK^G(\ad(U)\circ(\phi,\Iu))\otimes\kappa^{-1} = KK^G(\phi,\Iu)\otimes \kappa^{-1} = x.
\]
\end{proof}

\subsection{Uniqueness}

The following is the uniqueness theorem underpinning our classification theory.
Its strategy of proof can be regarded as the dynamical generalization of \cite[Lemma 7.4]{Gabe21}.

\begin{theorem} \label{thm:uniqueness}
Let $A$ be a separable exact \cstar-algebra with an action $\alpha: G\curvearrowright A$.
Let $B$ be a Kirchberg algebra and $\beta: G\curvearrowright B$ a strongly stable, amenable and isometrically shift-absorbing action.
Let $(\phi,\Iu), (\psi,\Iv): (A,\alpha)\to(B,\beta)$ be two proper cocycle embeddings that form an anchored Cuntz pair.
Then $KK^G(\phi,\Iu)=KK^G(\psi,\Iv)$ if and only if $(\phi,\Iu)$ and $(\psi,\Iv)$ are strongly asymptotically unitarily equivalent.
\end{theorem}
\begin{proof}
Since the ``if'' part is clear, we prove the ``only if'' part.
By \autoref{prop:isa-implies-absorption}, we have $\beta\cc\beta\otimes\id_{\CO_\infty}$, so with \autoref{thm:strong-ssa-abs}, it follows that there exists a proper cocycle conjugacy
\[
(\kappa,\Ix): (B,\beta)\to (B\otimes\CO_\infty,\beta\otimes\id_{\CO_\infty})
\]
that is strongly asymptotically unitarily equivalent to the equivariant first-factor embedding $\id_B\otimes\eins_{\CO_\infty}$.
We may in particular conclude that the proper cocycle morphism $(\kappa,\Ix)^{-1}\circ(\id_B\otimes\eins_{\CO_\infty},\eins)$ is strongly asymptotically inner.
This way we see that in order to prove the claim, it suffices to show that the two proper cocycle morphisms
\[
(\phi\otimes\eins_{\CO_\infty},\Iu\otimes\eins_{\CO_\infty}), (\psi\otimes\eins_{\CO_\infty},\Iv\otimes\eins_{\CO_\infty}): (A,\alpha)\to (B\otimes\CO_\infty,\beta\otimes\id_{\CO_\infty})
\]
are strongly asymptotically unitarily equivalent.
By appealing to \autoref{cor:special-relative-O2-embedding}, we may hence assume without loss of generality that there exists a unital inclusion $\CO_\infty\subset\CM(B)^\beta$ that commutes pointwise with the ranges of the maps $\phi,\psi,\Iu,\Iv$, and moreover a proper cocycle embedding $(\theta,\Iy): (A,\alpha)\to (B,\beta)$ and a unital inclusion $\iota_0: \CO_2\to\CM(B)^\beta$ whose range commutes with the range of $\theta$ and $\Iy$.

Combining \autoref{rem:ordinary-dom} and \autoref{lem:key-lemma}, it follows that both $(\phi,\Iu)$ and $(\psi,\Iv)$ approximately 1-dominate $(\theta,\Iy)$.
By \autoref{lem:strong-sum-absorption}, it follows that $(\phi,\Iu)$ is strongly asymptotically unitarily equivalent to $(\phi\oplus\theta,\Iu\oplus\Iy)$ and that $(\psi,\Iv)$ is strongly asymptotically unitarily equivalent to $(\psi\oplus\theta,\Iv\oplus\Iy)$.
In particular it suffices to show that the two proper cocycle morphisms
\[
(\phi\oplus\theta,\Iu\oplus\Iy), (\psi\oplus\theta,\Iv\oplus\Iy): (A,\alpha)\to (B,\beta)
\]
are strongly asymptotically unitarily equivalent, which we are about to do.

As before, we pick a sequence of isometries $r_n\in\CM(B)^\beta$ such that $\sum_{n=1}^\infty r_n r_n^*=\eins$ holds in the strict topology, and construct all infinite repeats by using this sequence.
Let us also consider the isometry $r_\infty=\sum_{n=1}^\infty r_{n+1}r_n^*$, which fits into the equation $r_1r_1^*+r_\infty r_\infty^*=\eins$.
Furthermore we let $\set{e_{k,\ell}\mid k,\ell\geq 1}$ be a set of matrix units generating $\CK$.

By \autoref{cor:inf-repeats-are-absorbing}, the infinite repeat $(\theta^\infty,\Iy^\infty)$ is an absorbing cocycle representation.
We then consider the non-degenerate embedding
\[
\iota: \CO_2\otimes\CK\to\CM(B)^\beta,\quad \iota(a\otimes e_{k,\ell})=r_k\iota_0(a)r_\ell^*,\quad k,\ell\geq 1.
\]
Using \autoref{lem:special-unitary-path}, we may find a continuous unitary path $u: [0,\infty)\to\CU(\eins+r_\infty\iota(\CO_2\otimes\CK)r_\infty^*)\subseteq\CU(\CM(B)^\beta)$ %with $u_0=\eins$ 
such that 
\[
r_\infty r_\infty^* = \lim_{t\to\infty} u_t^*r_\infty\iota(\eins\otimes e_{1,1})r_\infty^*u_t=\lim_{t\to\infty} u_t^* r_2r_2^* u_t
\] 
in the strict topology as $t\to\infty$. %, and moreover $u_tb$ converges in norm as $t\to\infty$ for all $b\in B$. 
Note that by construction, the range of $\iota$ commutes pointwise with the range of $\theta^\infty$ and $\Iy^\infty$, and therefore the range of $u$ commutes pointwise with the range of $0\oplus_{r_1,r_\infty}\theta^\infty$ and $0\oplus_{r_1,r_\infty}\Iy^\infty$.
Since the range of $u$ clearly acts like a unit on $r_1\CM(B)r_1^*$, we also observe the strict convergence
\[
\eins  =\lim_{t\to\infty} u_t^* (r_1r_1^* + r_2r_2^*) u_t.
\] 
By assumption, we have an equality of classes
\[
0=KK^G(\phi,\Iu)-KK^G(\psi,\Iv)=\big[ (\phi,\Iu), (\psi,\Iv) \big] \quad\text{in } KK^G(\alpha,\beta).
\]
Since we assumed the pair of proper cocycle morphisms to be anchored, it thus follows from the stable uniqueness theorem \cite[Theorem 5.4]{GabeSzabo22} that the cocycle representations $(\phi,\Iu)\oplus_{r_1,r_\infty}(\theta^\infty,\Iy^\infty)$ and $(\psi,\Iv)\oplus_{r_1,r_\infty}(\theta^\infty,\Iy^\infty)$ are strongly asymptotically unitarily equivalent.
In other words, we find a norm-continuous path of unitaries $w: [0,\infty)\to\CU(\eins+B)$ with $w_0=\eins$ such that
\[
\lim_{t\to\infty} w_t(\phi(a)\oplus_{r_1,r_\infty} \theta^\infty(a))w_t^*=\psi(a)\oplus_{r_1,r_\infty} \theta^\infty(a)
\]
for all $a\in A$, and
\[
\lim_{t\to\infty} \max_{g\in K} \| w_t(\Iu_g\oplus_{r_1,r_\infty}\Iy^\infty_g)\beta_g(w_t)^* - (\Iv_g\oplus_{r_1,r_\infty}\Iy^\infty_g) \| = 0
\]
for every compact set $K\subseteq G$. %\footnote{Here the Cuntz sum is meant to be formed with $r_1$ and $r_\infty$ each time.}
For ease of notation we shall denote $p_2=r_1r_1^* + r_2r_2^*$.
By reparameterizing and/or cutting off an initial segment of $u$, if necessary, we may additionally assume
\[
\lim_{t\to\infty} \| \big( u_t^* p_2 u_t - \eins \big) (w_t-\eins)  \| = 0 = \lim_{t\to\infty} \| \big( p_2 - \eins \big) (u_tw_tu_t^*-\eins)  \|,
\]
and that these norms are uniformly bounded above by $1/4$ over all $t\geq 0$.
Let us consider the norm-continuous path of elements $z': [0,\infty)\to \eins+p_2Bp_2$ given by
\[
z'_t = p_2u_tw_tu_t^*p_2+(\eins-p_2).
\]
Then $z'_0=\eins$ and $\sup_{t\geq 0} \|u_tw_tu_t^*-z'_t\|<1$, so we may define the continuous unitary path $z: [0,\infty)\to\CU(\eins+p_2Bp_2)$ via $z_t=z_t'|z_t'|^{-1}$.
By construction, we can see that $\lim_{t\to\infty} \|z_t-u_tw_tu_t^*\|=0$.
By choice of the unitary paths $u$ and $w$, we can now conclude
\[
\begin{array}{cl}
\multicolumn{2}{l}{ \dst \lim_{t\to\infty} z_t(\phi(a)\oplus_{r_1,r_\infty} \theta^\infty(a))z_t^* } \\
=& \dst \lim_{t\to\infty} u_tw_tu_t^*(\phi(a)\oplus_{r_1,r_\infty} \theta^\infty(a))u_tw_t^*u_t^* \\
=& \dst \lim_{t\to\infty} u_tw_t(\phi(a)\oplus_{r_1,r_\infty} \theta^\infty(a))w_t^*u_t^* \\
=& \dst \lim_{t\to\infty} u_t (\psi(a)\oplus_{r_1,r_\infty} \theta^\infty(a)) u_t^* \\
=& \psi(a)\oplus_{r_1,r_\infty} \theta^\infty(a)
\end{array}
\]
for all $a\in A$, and likewise
\[
\lim_{t\to\infty} \max_{g\in K} \| z_t(\Iu_g\oplus_{r_1,r_\infty}\Iy^\infty_g)\beta_g(z_t)^* - (\Iv_g\oplus_{r_1,r_\infty}\Iy^\infty_g) \| = 0
\]
for every compact set $K\subseteq G$.
By multiplying all the involved elements with $p_2$, we can see
\[
\lim_{t\to\infty} z_t(r_1\phi(a)r_1^*+ r_2\theta(a)r_2^*)z_t^*=r_1\psi(a)r_1^* + r_2\theta(a)r_2^*
\]
for all $a\in A$, and
\[
\lim_{t\to\infty} \max_{g\in K} \| z_t(r_1\Iu_gr_1^*+r_2\Iy_gr_2^*)\beta_g(z_t)^* - (r_1\Iv_gr_1^* + r_2\Iy_gr_2^*) \| = 0
\]
for every compact set $K\subseteq G$.
Consider the isometry $R=r_1r_1^*+r_2r_\infty^*\in\CM(B)^\beta$, which satisfies $RR^*=p_2$.
Then $v: [0,\infty)\to\CU(\eins+B)$ given by $v_t=R^*z_tR$ is a unitary path with $v_0=\eins$.
Since $R^*r_1=r_1$ and $R^*r_2=r_\infty$, we can consider the above limit properties and conjugate the terms via $R^*(\dots)R$, in order to finally arrive at
\[
\lim_{t\to\infty} v_t(\phi(a)\oplus_{r_1,r_\infty} \theta(a))v_t^*=\psi(a)\oplus_{r_1,r_\infty} \theta(a)
\]
for all $a\in A$, and
\[
\lim_{t\to\infty} \max_{g\in K} \| v_t(\Iu_g\oplus_{r_1,r_\infty}\Iy_g)\beta_g(v_t)^* - (\Iv_g\oplus_{r_1,r_\infty}\Iy_g) \| = 0
\]
for every compact set $K\subseteq G$.
In particular, we have just shown that $(\phi\oplus\theta,\Iu\oplus\Iy)$ and $(\psi\oplus\theta,\Iv\oplus\Iy)$ are indeed strongly asymptotically unitarily equivalent.
This finishes the proof.
\end{proof}

\begin{theorem} \label{thm:unital-uniqueness}
Suppose $G$ is exact.
Let $A$ be a separable unital exact \cstar-algebra with an action $\alpha: G\curvearrowright A$.
Let $B$ be a unital Kirchberg algebra and $\beta: G\curvearrowright B$ an amenable and isometrically shift-absorbing action.
Let $(\phi,\Iu), (\psi,\Iv): (A,\alpha)\to(B,\beta)$ be two unital cocycle embeddings.
Then $KK^G(\phi,\Iu)=KK^G(\psi,\Iv)$ if and only if $(\phi,\Iu)$ and $(\psi,\Iv)$ are asymptotically unitarily equivalent.
\end{theorem}
\begin{proof}
As before, the ``if'' part is clear, so from now on assume $KK^G(\phi,\Iu)=KK^G(\psi,\Iv)$ holds.

Denote $B^s=B\otimes\CK$ and $\beta^s=\beta\otimes\id_\CK$.
Let us consider the canonical inclusion $\iota: B\to B^s$ via $\iota(b)=b\otimes e_{1,1}$.
Consider the $\beta^s$-cocycles $\Iu'_g=\iota(\Iu_g)+\eins-e_{1,1}$ and $\Iv'_g=\iota(\Iv_g)+\eins-e_{1,1}$.
Applying \autoref{thm:existence}\ref{thm:existence:1} to $A=0$, we obtain a norm-continuous $\beta^s$-cocycle $\Ix: G\to\CU(\eins+B^s)$ such that $[(\Iu',\Iv')]_h=[(\Ix,\eins)]_h$.

Choose a pair of isometries $r_1,r_2\in\CM(\CK)\subseteq\CM(B^s)^{\beta^s}$ with $r_1r_1^*+r_2r_2^*=\eins$ and $r_1e_{1,1}=e_{1,1}$.
This leads in particular to the equation $\Iu'=\Iu'\oplus_{r_1,r_2}\eins$.
Then we have an equality (see \cite[Proposition 6.14]{Szabo21cc})
\[
KK^G(\iota\circ\phi,\Iu') = KK^G(\iota\circ\psi,\Iv'\oplus_{r_1,r_2}\Ix) 
\]
in $KK^G(\alpha,\beta^s)$.
By construction, we have 
\[
[(\Iu',\Iv\oplus\Ix)]_h=[(\Iu'\oplus_{r_1,r_2}\eins,\Iv'\oplus_{r_1,r_2}\Ix)]_h=[(\Iu',\Iv')]_h-[(\Ix,\eins)]_h=0.
\]
Thus we can apply \autoref{thm:uniqueness} and find a norm-continuous path of unitaries $v: [0,\infty)\to\CU(\eins+B^s)$ witnessing that $(\iota\circ\phi,\Iu')$ and $(\iota\circ\psi,\Iv'\oplus_{r_1,r_2}\Ix)$ are strongly asymptitically unitarily equivalent.
Then as $t\to\infty$, the unitaries $v_t$ approximately commute with $\eins_B\otimes e_{1,1}$.
Thus, after cutting away an initial segment of $v$, if necessary,\footnote{This is the reason why it is not always possible to arrange strong asymptotic unitary equivalence in the claim here.} we can define a norm-continuous path of unitaries
\[
u: [0,\infty)\to\CU(B),\ u_t=\iota^{-1}\Big( (\eins_B\otimes e_{1,1})v_t(\eins_B\otimes e_{1,1}) |(\eins_B\otimes e_{1,1})v_t(\eins_B\otimes e_{1,1})|^{-1} \Big)
\]
which then satisfies
\[
\lim_{t\to\infty} u_t\phi(a)u_t^*=\psi(a) \quad\text{and}\quad \lim_{t\to\infty} \max_{g\in K} \| u_t\Iu_g\beta_g(u_t)^* - \Iv_g \| = 0
\]
for all $a\in A$ and every compact set $K\subseteq G$.
This finishes the proof.
\end{proof}

%%%
\subsection{Characterizing asymptotic coboundaries}

We finish this section by applying the uniqueness theorem to determine in $K$-theoretic terms when cocycles on unital Kirchberg algebras can be realized as continuous limits of coboundaries.

\begin{defi}[cf.\ {\cite[Definition 1.4]{Szabo17ssa3}}]
Let $G$ be a second-countable, locally compact group.
Let $B$ be a \cstar-algebra and $\beta: G\curvearrowright B$ an action.
We say that a norm-continuous cocycle $\Iu: G\to\CU(\eins+B)$ is an \emph{asymptotic coboundary}, if there exists a continuous path of unitaries $v: [0,\infty)\to\CU(\eins+B)$ such that
\[
\lim_{t\to\infty} \max_{g\in K} \| \Iu_g-v_t\beta_g(v_t)^* \| =0
\]
for every compact set $K\subseteq G$.
If $\alpha: G\curvearrowright A$ is another action on a \cstar-algebra, then a proper cocycle conjugacy $(\phi,\Iu): (A,\alpha)\to (B,\beta)$ is called a very strong cocycle conjugacy, if $\Iu$ is an asymptotic coboundary.
\end{defi}

We recall the following observation due to Izumi:

\begin{prop}[see {\cite[Lemma 2.4]{Izumi04}}]
Let $G$ be a compact group and $\beta: G\curvearrowright B$ an action on a \cstar-algebra.
If $\Iu: G\to\CU(\eins+B)$ is a norm-continuous $\beta$-cocycle with $\max_{g\in G} \|\Iu_g-\eins\|<1$, then $\Iu$ is a coboundary, i.e., there exists $v\in\CU(\eins+B)$ with $\Iu_g=v\beta_g(v)^*$ for all $g\in G$.
\end{prop}

\begin{cor} \label{cor:vscc-for-compact}
Let $G$ be a compact group and $\beta: G\curvearrowright B$ an action on a \cstar-algebra.
If a norm-continuous $\beta$-cocycle $\Iu: G\to\CU(\eins+B)$ is an asymptotic coboundary, then it is a coboundary.
Moreover, if $\alpha: G\curvearrowright A$ is another action on a \cstar-algebra, then a very strong cocycle conjugacy $(\phi,\Iu): (A,\alpha)\to (B,\beta)$ is properly unitarily equivalent to a conjugacy.
\end{cor}

\begin{rem} \label{rem:vscc-comp}
We recall (see \cite[Section 1]{Szabo21cc}) that for any two proper cocycle morphisms $(\phi,\Iu): (A,\alpha)\to (B,\beta)$ and $(\psi,\Iv): (B,\beta)\to (C,\gamma)$, their composition is given by $(\psi,\Iv)\circ(\phi,\Iu)=(\psi\circ\phi,\psi(\Iu_\bullet)\Iv_\bullet)$.
Suppose that $\Iu$ is an asymptotic $\beta$-coboundary witnessed by a path $y: [0,\infty)\to\CU(\eins+B)$ and that $\Iv$ is an asymptotic $\gamma$-coboundary witnessed by the path $z: [0,\infty)\to\CU(\eins+C)$.
Then it follows that also $\psi(\Iu_\bullet)\Iv_\bullet$ is an asymptotic $\gamma$-coboundary, since we can compute for all $g\in G$ that
\[
\begin{array}{ccl}
\psi(\Iu_g)\Iv_g &=& \dst \lim_{t\to\infty} \psi(y_t\beta_g(y_t)^*)\Iv_g \\
&=& \dst \lim_{t\to\infty} \psi(y_t) \Iv_g \gamma_g(\psi(y_t))^* \\
&=& \dst \lim_{t\to\infty} \psi(y_t)z_t\gamma_g(z_t^*\psi(y_t)^*).
\end{array}
\]
\end{rem}

\begin{nota}
Given a unital \cstar-algebra $A$, we will denote by $\iota_A: \IC\to A$ the canonical unital inclusion.
If $A$ carries a group action, we will equip $\IC$ with the trivial action so that $\iota_A$ can be viewed as an equivariant inclusion.
\end{nota}

\begin{theorem} \label{thm:compact-first-cohomology}
Let $\beta: G\curvearrowright B$ be an amenable and isometrically shift-absorbing action on a Kirchberg algebra and let $\Iu: G\to\CU(\eins+B)$ be a $\beta$-cocycle.
	\begin{enumerate}[label=\textup{(\roman*)},leftmargin=*]
	\item \label{thm:compact-first-cohomology:1}
	Suppose $\beta$ is strongly stable.
	If $(\Iu,\eins)\sim_h (\eins,\eins)$ in the sense of \autoref{def:KKG-Thomsen}, then $\Iu$ is an asymptotic coboundary, and in fact there exists a norm-continuous path $y: [0,\infty)\to\CU(\eins+B)$ with $y_0=\eins$ and
\[
\lim_{t\to\infty} \max_{g\in K} \|\Iu_g-y_t\beta_g(y_t)^*\|=0
\]
for every compact set $K\subseteq G$.
	\item \label{thm:compact-first-cohomology:2}
	Suppose $G$ is exact and $B$ is unital.
	Consider the element $\Iu^\sharp=KK^G(\iota_B,\Iu)\in KK^G(\id_\IC,\beta)$ associated to $\Iu$.
Then $\Iu$ is an asymptotic coboundary if and only if $\Iu^\sharp=KK^G(\iota_B,\eins)$.
	\end{enumerate}
\end{theorem}
\begin{proof}
\ref{thm:compact-first-cohomology:1}: This is a direct consequence of \autoref{thm:uniqueness} applied to the special case $A=0$ and the two proper cocycle embeddings $(0,\Iu)$ and $(0,\eins)$.

\ref{thm:compact-first-cohomology:2}: We observe by comparing definitions that $\Iu$ is an asymptotic coboundary if and only if the two unital cocycle embeddings
\[
(\iota_B,\eins), (\iota_B,\Iu): (\IC,\id)\to (B,\beta)
\]
are asymptotically unitarily equivalent.
This is seen to be equivalent to $\Iu^\sharp=KK^G(\iota_B)$ by applying \autoref{thm:unital-uniqueness} to $A=\IC$.
\end{proof}

%%%%%%%%%%%%%%%%%%%%%%%%%%%%%%%%%%%

\section{The classification theorem and some applications}

We recall the following intertwining result from \cite[Corollary 4.6]{Szabo21cc}, which will be the last piece towards our main classification theorem.

\begin{theorem} \label{thm:two-sided-intertwining}
Let $\alpha: G\curvearrowright A$ and $\beta: G\curvearrowright B$ be two actions on separable \cstar-algebras.
Suppose that
\[
(\phi,\Iu): (A,\alpha) \to (B,\beta) \quad\text{and}\quad  (\psi,\Iv): (B,\beta)\to (A,\alpha)
\]
are two proper cocycle morphisms such that both
\[
(\psi,\Iv)\circ(\phi,\Iu)  \quad\text{and}\quad (\phi,\Iu)\circ(\psi,\Iv) 
\]
are properly asymptotically inner.
Then $(\phi,\Iu)$ is strongly asymptotically unitarily equivalent to a proper cocycle conjugacy.
\end{theorem}

The following is our main classification result and proves \autoref{intro:main-result}, which includes \autoref{intro:main-result-discrete} as a special case.
Before starting it, we remind the reader of Zhang's dichotomy \cite{Zhang92}, which implies that every Kirchberg algebra is either stable or unital.

\begin{theorem} \label{thm:main-theorem}
Let $G$ be a second-countable locally compact group.
Let $\alpha: G\curvearrowright A$ and $\beta: G\curvearrowright B$ be two amenable and isometrically shift-absorbing actions on Kirchberg algebras.\footnote{Let us point out once more that when $G$ is discrete, $\alpha$ and $\beta$ are isometrically shift-absorbing if and only if they are outer by \autoref{rem:outer-actions}.}
	\begin{enumerate}[label=\textup{(\roman*)},leftmargin=*]
	\item \label{thm:main-theorem:1}
	If both $A$ and $B$ are stable, then every invertible element $x\in KK^G(\alpha,\beta)$ lifts to a cocycle conjugacy $(A,\alpha)\to (B,\beta)$.
	\item \label{thm:main-theorem:2}
	If $\alpha$ and $\beta$ are strongly stable, then every invertible element $x\in KK^G(\alpha,\beta)$ lifts to a very strong cocycle conjugacy $(A,\alpha)\to (B,\beta)$.
	\item \label{thm:main-theorem:3}
	Suppose $G$ is exact.
	If $A$ and $B$ are unital, then every invertible element $x\in KK^G(\alpha,\beta)$ with $[\eins_A]_0\otimes x=[\eins_B]_0$ lifts to a cocycle conjugacy $(A,\alpha)\to (B,\beta)$.
	Moreover, such an element $x\in KK^G(\alpha,\beta)$ lifts to a very strong cocycle conjugacy if and only $KK^G(\iota_A)\otimes x=KK^G(\iota_B)$.
	\end{enumerate}
%
%If $G$ is compact, then every instance of ``very strong cocycle conjugacy'' can be replaced by ``conjugacy''.
\end{theorem}
\begin{proof}
In light of \autoref{prop:cc-stability}, we see that \ref{thm:main-theorem:1} follows directly from \ref{thm:main-theorem:2}.

\ref{thm:main-theorem:2}: By applying \autoref{thm:existence} twice, we may find two anchored proper cocycle embeddings $(\phi,\Iu): (A,\alpha)\to (B,\beta)$ and $(\psi,\Iv): (B,\beta)\to (A,\alpha)$ with $KK^G(\phi,\Iu)=x$ and $KK^G(\psi,\Iv)=x^{-1}$.
By \autoref{thm:compact-first-cohomology}\ref{thm:compact-first-cohomology:1}, there exists a norm-continuous path $y: [0,\infty)\to\CU(\eins+B)$ with $y_0=\eins$ and
\[
\lim_{t\to\infty} \max_{g\in K} \|\Iu_g-y_t\beta_g(y_t)^*\|=0
\]
for every compact set $K\subseteq G$.
Likewise, we can choose such a path $z: [0,\infty)\to\CU(\eins+A)$ for the $\alpha$-cocycle $\Iv$.
In light of \autoref{rem:vscc-comp}, we can observe that the $\alpha$-cocycle $\psi(\Iu_\bullet)\Iv_\bullet$ is norm-homotopic to the trivial $\alpha$-cocycle, hence $[(\psi(\Iu_\bullet)\Iv_\bullet,\eins)]_h=0$ in $\IE^G(\beta,\alpha)/{\sim_h}$.
Analogously we have $[(\phi(\Iv_\bullet)\Iu_\bullet,\eins)]_h=0$ in $\IE^G(\alpha,\beta)/{\sim_h}$.
Hence the compositions $(\psi,\Iv)\circ(\phi,\Iu)$ and $(\phi,\Iu)\circ(\psi,\Iv)$ are both anchored.
This allows us to apply \autoref{thm:uniqueness} and conclude that they are both strongly asymptotically inner.
We conclude from \autoref{thm:two-sided-intertwining} that $(\phi,\Iu)$ is strongly asymptotically unitarily equivalent to a proper cocycle conjugacy $(\Phi,\IU)$, which necessarily also represents $x$.
Since $\Iu$ was an asymptotic coboundary, so is $\IU$ and $(\Phi,\IU)$ is in fact a very strong cocycle conjugacy.

\ref{thm:main-theorem:3}: 
For the first part of the claim, carry out the analogous argument we used to prove \ref{thm:main-theorem:2} above, but use \autoref{thm:unital-existence} in place of \autoref{thm:existence} and \autoref{thm:unital-uniqueness} in place of \autoref{thm:uniqueness}.
Let us hence prove the ``Moreover'' part, i.e., let an invertible element $x\in KK^G(\alpha,\beta)$ with $[\eins_A]_0\otimes x=[\eins_B]_0$ be given.
We already know that $x$ lifts to a cocycle conjugacy $(\phi,\Iu): (A,\alpha)\to (B,\beta)$.
By the compatibility of the Kasparov product with respect to compositions, we observe that
\[
KK^G(\iota_A)\otimes x = KK^G(\iota_A)\otimes KK^G(\phi,\Iu)=KK^G(\iota_B,\Iu)=\Iu^\sharp.
\]
If $(\phi,\Iu)$ can be chosen to be a very strong cocycle conjugacy, then it follows by the ``only if'' part in \autoref{thm:compact-first-cohomology}\ref{thm:compact-first-cohomology:2} that this class is equal to $KK^G(\iota_B)$.
Conversely, if we assume $KK^G(\iota_A)\otimes x=KK^G(\iota_B)$ and choose $(\phi,\Iu)$ arbitrarily, then it evidently follows that $\Iu^\sharp=KK^G(\iota_B)$.
By the ``if'' part from \autoref{thm:compact-first-cohomology}\ref{thm:compact-first-cohomology:2}, it follows that $\Iu$ is an asymptotic coboundary, so $(\phi,\Iu)$ is automatically a very strong cocycle conjugacy.
%
%Finally, if $G$ is compact, then the last part of the statement is a direct consequence of \autoref{cor:vscc-for-compact}.
\end{proof}

In the case of compact groups, we can further improve the conclusion above by observing that one is always in one of the situations described by the two conditions \ref{thm:main-theorem:2} and \ref{thm:main-theorem:3}.

\begin{prop} \label{prop:compact-groups}
Let $G$ be a second-countable compact group and $\beta: G\curvearrowright B$ an isometrically shift-absorbing action on a separable stable \cstar-algebra.
Then $\beta$ is strongly stable.
\end{prop}
\begin{proof}
We assume $B\neq 0$.
The inclusion $B^\beta\subseteq B$ from the fixed point algebra is non-degenerate.
Since this induces a canonical unital inclusion $\CM(B^\beta)\subseteq\CM(B)^\beta$ that is strictly continuous on the unit ball, it follows directly from \cite[Remark 1.4]{GabeSzabo22} (see the paragraph after \autoref{nota:strongly-stable}) that $\beta$ is strongly stable if $B^\beta$ is stable.
By the main result of \cite{HjelmborgRordam98}, any given $\sigma$-unital \cstar-algebra $A$ is stable if and only if for every $\eps>0$ and positive element $e\in A$ of norm one, there exists a contraction $x\in A$ with $e=_\eps x^*x$ and $\|exx^*\|\leq\eps$.
We shall verify this condition for $A=B^\beta$.

Let $\eps>0$ and $e\in B^\beta$ be given.
Since $B$ is stable, we can find a contraction $y\in B$ with $e=_\eps y^*y$ and $\|e-y^*y\| + \|eyy^*\|\leq\eps$.
Using that $e$ is a fixed point, we get for all $g\in G$ that
\[
\eps \geq \|e-y^*y\| + \|eyy^*\| = \|e-\beta_g(y^*y)\| + \|e\beta_g(yy^*)\|.
\]
We consider $L^2(G,B)$ constructed in the sense of \autoref{nota:Hilbert-modules} with the normalized Haar measure $\mu$ on $G$ and define a contraction $\zeta\in\CC(G,B)\subseteq L^2(G,B)$ via $\zeta(h)=\beta_h(y)$.
Then $\zeta$ is fixed under $\bar{\beta}$ and moreover we observe the two properties
\[
\| e-\langle \zeta\mid \zeta\rangle\| = \Big\| \int_G \big( e-\beta_g(y^*y) \big)~d\mu \Big\| \leq \eps.
\]
and
\[
\| e\zeta \|_2^2 = \Big\| \int_G e\beta_g(yy^*)e ~d\mu \Big\| \leq \int_G \|e\beta_g(yy^*)e\|~d\mu \leq \eps.
\]
By \autoref{prop:the-condition}\ref{prop:the-condition:4}, there exists an equivariant linear $B$-bimodule map 
\[
\theta: ( L^2(G,B), \bar{\beta}) \to ( B_{\infty,\beta} , \beta_\infty)
\] 
satisfying $\theta(\xi)^*\theta(\eta)=\langle\xi\mid\eta\rangle_B$ for all $\xi,\eta\in L^2(G,B)$.
By what we have computed above, $z=\theta(\zeta)\in (B_{\infty,\beta})^{\beta_\infty}=(B^\beta)_\infty$ is a contraction satisfying $\|e-z^*z\|\leq\eps$ and $\|ez\|^2\leq\eps$.
Since $\eps>0$ was arbitrary, we can apply a standard reindexation trick and find a sequence of contractions $x_n\in B^\beta$ satisfying $x_n^*x_n\to e$ and $ex_n\to 0$.
Since $e$ was arbitrary, it follows that $B^\beta$ is indeed stable.  
\end{proof}

For compact groups $G$, very strong cocycle conjugacy coincides with conjugacy by \autoref{cor:vscc-for-compact}. Hence we may combine the above proposition with \autoref{thm:main-theorem} and obtain classification up to genuine conjugacy for compact groups.

\begin{cor} \label{cor:main-compact}
Let $G$ be a second-countable compact group.
Let $\alpha: G\curvearrowright A$ and $\beta: G\curvearrowright B$ be two isometrically shift-absorbing actions on Kirchberg algebras.	\begin{enumerate}[label=\textup{(\roman*)},leftmargin=*]
	\item \label{cor:main-compact:1}
	If $A$ and $B$ are stable, then every invertible element $x\in KK^G(\alpha,\beta)$ lifts to a conjugacy $(A,\alpha)\to (B,\beta)$.
	\item \label{cor:main-compact:2}
	If $A$ and $B$ are unital, then every invertible element $x\in KK^G(\alpha,\beta)$ with $KK^G(\iota_A)\otimes x=KK^G(\iota_B)$ lifts to a conjugacy $(A,\alpha)\to (B,\beta)$.
	\end{enumerate}
\end{cor}

\begin{rem}
The extra condition appearing in both \autoref{thm:main-theorem}\ref{thm:main-theorem:3} and \autoref{cor:main-compact}\ref{cor:main-compact:2} on the $KK^G$-class $x$ is not redundant.
Even for $G=\IZ/2\IZ$, there are known examples of outer actions $\alpha,\beta: G\curvearrowright B$ for $B=\CO_2$ that are cocycle conjugate, but not conjugate; see \cite[Corollary 5.5]{BarlakLi17}.
If we choose these actions and define $x$ to be the class associated to a cocycle conjugacy $(\phi,\Iu): (B,\alpha)\to (B,\beta)$, then it follows that $x$ cannot satisfy the extra condition $KK^G(\iota_B)\otimes x=KK^G(\iota_B)$.
More generally, let $G$ be any compact group and $\beta: G\curvearrowright B$ any action as in \autoref{thm:main-theorem}\ref{thm:main-theorem:3}. 
Given a $\beta$-cocycle $\Iu: G\curvearrowright\CU(B)$, we consider the $KK^G$-equivalence in $KK^G(\beta^\Iu,\beta)$ induced by the exterior equivalence $(\id_B,\Iu): (B,\beta^\Iu)\to (B,\beta)$. 
By \autoref{thm:compact-first-cohomology}, the extra condition $\Iu^\sharp=KK^G(\iota_B)$ holds if and only if $\Iu$ is a coboundary, which holds if and only if this element can also be represented by a conjugacy.
In some very special cases, such as when $\beta$ has the Rokhlin property, it may happen that all cocycles are coboundaries, but in general this provides plenty of examples demonstrating that conjugacy between actions on unital Kirchberg algebras is indeed stronger than cocycle conjugacy.
\end{rem}

We now demonstrate how a special case of our main results provides a positive solution to a conjecture of Izumi; see \cite[Conjecture 1]{Izumi10}, \cite[Conjecture 1.2]{IzumiMatui21}, and \cite[Conjecture 1.1]{IzumiMatui21_2}.
The last of these conjectures is the precise one that we verify below, which has been deemed ``the most optimistic version of conjectures of similar kind'' in \cite{IzumiMatui21_2}.
In order to set up the statement of the theorem we choose to directly quote the relevant paragraph from the introduction of \cite{IzumiMatui21_2}:
\renewenvironment{quote}{%
   \list{}{%
     \leftmargin 2mm   % this is the adjusting screw
     \rightmargin 0mm
   }
   \item\relax
}
{\endlist}
	\begin{quote}
	\it ``We recall the notion of classifying spaces in algebraic topology first.
	For any topological group $G$, there exists a universal principal $G$-bundle $EG \to BG$ satisfying the following property: 
	every numerable principal $G$-bundle $\CP \to X$ is isomorphic to the pullback bundle $f^*EG$ of a continuous map $f : X \to BG$ so that the set of isomorphism classes of numerable principal $G$-bundles over $X$ is in one-to-one correspondence with the homotopy set $[X,BG]$; see \cite[Chapter 4]{Husemoller94}. 
	The space $BG$, which is unique up to homotopy equivalence by universality, is called the classifying space of $G$. 
	Since the Milnor construction of $BG$ is functorial, a continuous group homomorphism $h : G_1 \to G_2$ induces a continuous map $Bh : BG_1 \to BG_2$.
	If moreover $G_1$ and $G_2$ are discrete groups, the map
	\[
	\Hom(G_1,G_2)/\text{conjugacy} \ \ni \ [h] \mapsto [Bh] \ \in \ [BG_1,BG_2]
	\]
is a bijection, which follows from the classification of regular covering spaces over $BG_1$; see for example \cite[Section 3.7, 3.8]{May99}.''
	\end{quote}

The following theorem, when restricted to actions of poly-$\IZ$ groups, recovers and generalizes the main result of \cite{IzumiMatui21_2}.

\begin{theorem} \label{thm:Izumi-conjecture}
Let $G$ be a countable discrete amenable torsion-free group, and let $A$ be a stable Kirchberg algebra. Then the map
\[
\CO\CA(G,A)/\text{cocycle conjugacy} \ \ni \ [\alpha] \mapsto [B\alpha] \ \in \ [BG,B\Aut(A)]
\]
is a bijection, where $\CO\CA(G,A)$ denotes the set of outer actions of $G$ on $A$.
\end{theorem}
\begin{proof}
Meyer showed in \cite[Theorem 3.10]{Meyer21}\footnote{Meyer's theorem relies on Kirchberg's unpublished classification of non-simple purely infinite \cstar-algebras. The proof was supposed to appear in his book which unfortunately was not completed before his passing. The first named author provided an alternative proof of this classification theorem in \cite{Gabe21}.} that $KK^G$-classes of (outer) $G$-actions on $A$ are in a natural bijective correspondence with $[BG,B\Aut(A)]$.
As a consequence, the claim holds precisely when outer $G$-actions on $A$ are classified by $KK^G$.
This is what is accomplished in \autoref{thm:main-theorem} and therefore the claim is proved. 
\end{proof}

Next, we observe that the classification result implies that in many cases, isometric shift-absorption for an action is expressible in terms of (very strongly) tensorially absorbing the canonical quasi-free action on $\CO_\infty$.

\begin{defi}[cf.\ {\cite[Definition 5.1]{Szabo21cc}}]
Let $\alpha: G\curvearrowright A$ and $\delta: G\curvearrowright D$ be actions on separable \cstar-algebras, and suppose $D$ is unital.
We say that $\alpha$ \emph{very strongly absorbs} $\delta$, if the equivariant embedding
\[
\id_A\otimes\eins_D: (A,\alpha)\to (A\otimes D, \alpha\otimes\delta)
\]
is strongly asymptotically unitarily equivalent to a cocycle conjugacy.
\end{defi}

\begin{cor} \label{cor:model-action-ssa}
Let $\gamma: G\curvearrowright\CO_\infty$ be the model action from \autoref{def:the-model}.
Let $\beta: G\curvearrowright B$ be an amenable action on a Kirchberg algebra.
Consider the following conditions:
	\begin{enumerate}[label=\textup{(\roman*)},leftmargin=*]
	\item $\beta$ very strongly absorbs $\gamma$; \label{cor:model-action-ssa:1}
	\item $\beta$ is cocycle conjugate to $\beta\otimes\gamma^{\otimes\infty}$; \label{cor:model-action-ssa:2}
	\item $\beta$ is isometrically shift-absorbing. \label{cor:model-action-ssa:3}
	\end{enumerate}
Then \ref{cor:model-action-ssa:1}$\implies$\ref{cor:model-action-ssa:2}$\iff$\ref{cor:model-action-ssa:3}.
If $B$ is unital or $\beta$ is strongly stable, then all three conditions are equivalent.
Furthermore, if $G$ is amenable, then $\gamma^{\otimes\infty}: G\curvearrowright\CO_\infty^{\otimes\infty}$ is strongly self-absorbing and all three conditions are equivalent.	
If $G$ is both discrete and amenable, then in fact $\gamma$ is strongly self-absorbing.\footnote{Whether this is true when $G$ is non-discrete remains an open problem.}
\end{cor}
\begin{proof}
The implication \ref{cor:model-action-ssa:1}$\implies$\ref{cor:model-action-ssa:2} holds in general due to \cite[Proposition 5.2]{Szabo21cc}.
The implication \ref{cor:model-action-ssa:2}$\implies$\ref{cor:model-action-ssa:3} is clear by \autoref{prop:the-condition}.
Let us for a moment assume that $B$ is either unital or $\beta$ is strongly stable and prove \ref{cor:model-action-ssa:3}$\implies$\ref{cor:model-action-ssa:1}.
By \autoref{rem:KK-class-quasifree}, we see that the first-factor embedding
\[
\id_B\otimes\eins: (B,\beta)\to (B\otimes \CO_\infty, \beta\otimes\gamma)
\]
is a $KK^G$-equivalence.
Since both the domain and codomain actions are amenable and isometrically shift-absorbing, it follows from \autoref{thm:main-theorem} and either \autoref{thm:unital-uniqueness} or \autoref{thm:uniqueness} that this embedding is strongly asympotically unitarily equivalent to a cocycle conjugacy.
In particular, $\beta$ very strongly absorbs $\gamma$.
The only remaining implication is \ref{cor:model-action-ssa:3}$\implies$\ref{cor:model-action-ssa:2} if $B$ is non-unital and $\beta$ is not strongly stable.
However, condition \ref{cor:model-action-ssa:2} is clearly invariant under cocycle conjugacy.
Since $B$ is stable, $\beta$ is therefore cocycle conjugate to a strongly stable action by \autoref{prop:cc-stability}, so the above completes the circle of implications.

Now let us assume that $G$ is amenable.
Then the above applies in particular to $\beta=\gamma^{\otimes\infty}$ and yields the strong self-absorption of this action.
In this case, the implication \ref{cor:model-action-ssa:2}$\implies$\ref{cor:model-action-ssa:1} follows from \autoref{thm:strong-ssa-abs}.
Finally, let us assume that $G$ is discrete and amenable.
By \autoref{rem:outer-actions}, every outer $G$-action on a Kirchberg algebra is isometrically shift-absorbing, so the claim follows from the fact that every faithful quasi-free action on $\CO_\infty$ is outer; see \autoref{prop:quasifree-outer}. 
\end{proof}

Next, we verify another conjecture of Izumi and hence generalize the main result of \cite{GoldsteinIzumi11} to infinite amenable groups:

\begin{theorem} \label{thm:Izumi-other-conjecture}
Let $G$ be a countable discrete amenable group.
Then all faithful quasi-free actions $G\curvearrowright\CO_\infty$ are mutually very strongly cocycle conjugate.
\end{theorem}
\begin{proof}
We know by \autoref{prop:quasifree-outer} that every faithful quasi-free action is outer.
Moreover we know by \autoref{rem:KK-class-quasifree} that the canonical unital inclusion $\IC\subset\CO_\infty$ is a $KK^G$-equivalence with respect to any quasi-free action on $\CO_\infty$.
In particular, we see that any pair of faithful quasi-free actions $G\curvearrowright\CO_\infty$ has a canonical $KK^G$-equivalence between them as required by \autoref{thm:main-theorem}\ref{thm:main-theorem:3}, which shows the claim.
\end{proof}

As a consequence of our existence theorem, we can, via abstract means, construct new examples of amenable actions of non-amenable groups with the Haagerup property, which so far are only known for free groups; see \cite[Corollary 6.4]{OzawaSuzuki21}.
In particular this partially resolves an issue raised in the concluding remarks of \cite{Suzuki21}.
We note that shortly before the submission of this article, Suzuki \cite{Suzuki23} has constructed such examples by a different method.

\begin{theorem} \label{thm:exact-HP}
Let $G$ be a second-countable, locally compact group.
Suppose $G$ is exact and satisfies the Haagerup property.
Then there exists an amenable and isometrically shift-absorbing action $\alpha: G\curvearrowright\CO_\infty$ such that the unital inclusion $\IC\subset\CO_\infty$ is a $KK^G$-equivalence.
\end{theorem}
\begin{proof}
By \cite[Corollary 6.3]{OzawaSuzuki21}, we can find an amenable action $\beta: G\curvearrowright\CO_\infty\otimes\CK$ and an invertible element $x\in KK^G(\id_\IC,\beta)$.
Due to \autoref{rem:ISA-KK}, we may assume $\beta$ to be isometrically shift-absorbing.
With \autoref{thm:existence} we can find a proper cocycle embedding $(\psi,\Iv): (\IC,\id_\IC)\to (\CO_\infty\otimes\CK,\beta)$ with $KK^G(\psi,\Iv)=x$.
In particular we have that $p=\psi(\eins)$ is a projection fixed by the action $\beta^\Iv$.
We may hence define $\alpha$ to be the restriction of $\beta^\Iv$ to the corner spanned by $p$.
Since $\psi$ has to also induce an ordinary $KK$-equivalence, we deduce that $p$ must be a generator of the $K_0$-group, which in this case means that $p(\CO_\infty\otimes\CK)p\cong\CO_\infty$.
By the stability properties of $KK^G$, the inclusion map $\iota: (\CO_\infty,\alpha)\to (\CO_\infty\otimes\CK,\beta^\Iv)$ is a $KK^G$-equivalence.
If $\iota_1: \IC\to\CO_\infty$ denotes the unital inclusion, then it fits into the obvious equality of proper cocycle morphisms $(\id,\Iv)\circ\iota\circ\iota_1=(\psi,\Iv)$.
This implies\footnote{Keep in mind that Kasparov product is compatible with compositions in the reverse order.} that 
\[
KK^G(\iota_1)= x\otimes KK^G(\id,\Iv^*)\otimes KK^G(\iota)^{-1} \ \in \ KK^G(\id_\IC,\alpha)
\]
is invertible.
Since both amenability and isometric shift-absorption are properties passing to cocycle perturbations and hereditary subsystems (cf.\ \autoref{prop:the-condition} and \cite[Proposition 3.7]{OzawaSuzuki21}), $\alpha$ shares these properties with $\beta$.
\end{proof}

We also exhibit the following interesting phenomenon (along with some consequences), which can be obtained from our main result in conjunction with the Baum--Connes machinery of Meyer--Nest \cite{MeyerNest06}.
It is rather striking that such a statement can be obtained as a consequence of the deep homological algebra techniques that are applicable to the structure of equivariant $KK$-groups, especially because it seems impossible to prove more directly.

\begin{theorem} \label{thm:Meyer-Nest-consequence}
Let $G$ be a second-countable locally compact group with the Haagerup property.
Let $\alpha: G\curvearrowright A$ and $\beta: G\curvearrowright B$ be two amenable and isometrically shift-absorbing actions on Kirchberg algebras.
Suppose that either both $A$ and $B$ are unital or both $\alpha$ and $\beta$ are strongly stable.
Let $(\phi,\Iu): (A,\alpha)\to (B,\beta)$ be a proper cocycle morphism.
Then $(\phi,\Iu)$ is strongly asymptotically unitarily equivalent to a proper cocycle conjugacy if and only if for every compact subgroup $H\subseteq G$, the proper cocycle embedding of restricted $H$-actions $(\phi,\Iu|_H): (A,\alpha|_H)\to (B,\beta|_H)$ is strongly asymptotically unitarily equivalent to a proper cocycle conjugacy (of $H$-actions).
\end{theorem}
\begin{proof}
The ``only if'' part is tautological.
For the ``if'' part, keep in mind that the case $H=\set{1}$ already implies that $\phi$ must be an embedding, which is necessarily unital if $A$ and $B$ are unital.
We are in a position to apply \autoref{thm:main-theorem} and either \autoref{thm:uniqueness} or \autoref{thm:unital-uniqueness}.
Hence we observe that $(\phi,\Iu)$ is strongly asymptotically unitarily equivalent to a proper cocycle conjugacy if and only if $KK^G(\phi,\Iu)\in KK^G(\alpha,\beta)$ is invertible.
Given any compact subgroup $H\subseteq G$, the same equivalence holds for the restricted $H$-actions if we insert $H$ in place of $G$.
In other words, the claim amounts to saying that if $KK^H(\phi,\Iu|_H)\in KK^H(\alpha|_H,\beta|_H)$ is invertible for every compact subgroup $H\subseteq G$, then $KK^G(\phi,\Iu)\in KK^G(\alpha,\beta)$ is invertible.
But this follows directly from \cite[Theorem 8.5]{MeyerNest06}.
\end{proof}

\begin{cor} \label{cor:ssa-abs-MN}
Let $G$ be a second-countable locally compact group with the Haagerup property.
Let $D$ be a separable unital simple nuclear \cstar-algebra with actions $\delta,\delta^{(1)},\delta^{(2)}: G\curvearrowright D$.
Let $\alpha: G\curvearrowright A$ be an amenable and isometrically shift-absorbing action on a Kirchberg algebra.
Suppose that $A$ is either unital or $\alpha$ is strongly stable.
Then:
\begin{enumerate}[leftmargin=*,label=\textup{(\roman*)}]
\item \label{cor:ssa-abs-MN:1}
$\alpha$ very strongly absorbs $\delta$ if and only if for every compact subgroup $H\subseteq G$, the restricted $H$-action $\alpha|_H$ very strongly absorbs $\delta|_H$.
\item \label{cor:ssa-abs-MN:2}
Suppose $\delta$ is amenable and isometrically shift-absorbing.
Then $\delta$ is strongly self-absorbing if and only if for every compact subgroup $H\subseteq G$, the restricted action $\delta|_H$ is strongly self-absorbing.
\item \label{cor:ssa-abs-MN:3}
Suppose that $\delta^{(1)}$ and $\delta^{(2)}$ are amenable, isometrically shift-absorbing, and strongly self-absorbing.
Then $\delta^{(1)}\cc\delta^{(2)}$ if and only if for every compact subgroup $H\subseteq G$, one has $\delta^{(1)}|_H\cc\delta^{(2)}|_H$ as $H$-actions. 
\end{enumerate}
\end{cor}
\begin{proof}
We first note that by \autoref{prop:the-condition}, every amenable and isometrically shift-absorbing action also absorbs the trivial action on $\CO_\infty$.

\ref{cor:ssa-abs-MN:1}:
This is a special case of \autoref{thm:Meyer-Nest-consequence} if we insert $\alpha\otimes\delta$ in place of $\beta$.

\ref{cor:ssa-abs-MN:2}: 
Since $\delta$ is equivariantly Jiang--Su stable, it follows from \autoref{thm:strong-ssa-abs} that $\delta$ is strongly self-absorbing if and only if $\delta$ very strongly absorbs $\delta$.
Hence this becomes a special case of \ref{cor:ssa-abs-MN:1}.

\ref{cor:ssa-abs-MN:3}:
Since we assume $\delta^{(1)}$ and $\delta^{(2)}$ to be strongly self-absorbing, they are cocycle conjugate if and only if they absorb each other.
Since they are both equivariantly Jiang--Su stable, it follows from \autoref{thm:strong-ssa-abs} that this is the case if and only if they absorb each other very strongly.
Hence the claim follows from applying part \ref{cor:ssa-abs-MN:1} twice.
\end{proof}

As a consequence, we can partially verify and extend on a conjecture by the second author, which was originally posed for actions of amenable groups; see \cite{Szabo17ssa3} or \cite[Conjecture A]{Szabo19ssa4}.

\begin{cor} \label{cor:ssa-conj}
Let $G$ be a countable discrete torsion-free group with the Haagerup property.
Let $\CD$ be a strongly self-absorbing Kirchberg algebra.
%Then:
\begin{enumerate}[leftmargin=*,label=\textup{(\roman*)}]
\item \label{cor:ssa-conj:1}
An amenable action $\alpha: G\curvearrowright A$ on a $\CD$-stable Kirchberg algebra is outer if and only if $\alpha$ very strongly absorbs every action $\delta: G\curvearrowright\CD$.
\item \label{cor:ssa-conj:2}
Suppose $G$ is exact.
Then up to (very strong) cocycle conjugacy, there exists a unique amenable outer action $\delta: G\curvearrowright\CD$, which is automatically strongly self-absorbing.
\end{enumerate}
\end{cor}
\begin{proof}
\ref{cor:ssa-conj:1} follows from \autoref{cor:ssa-abs-MN}\ref{cor:ssa-abs-MN:1} for $H=\set{1}$.

\ref{cor:ssa-conj:2}: From \autoref{thm:exact-HP} we can conclude that there exists such an action in the first place.
The fact that it is strongly self-absorbing and unique follows directly from \autoref{cor:ssa-abs-MN}\ref{cor:ssa-abs-MN:2}+\ref{cor:ssa-abs-MN:3} for $H=\set{1}$.
\end{proof}

For the subsequent applications of our main results to flows, we recall the statement below as a special case of \cite[Theorem 20.7.3]{BlaKK}, which can be viewed as a strengthening of the Connes--Thom isomorphism theorem \cite{Connes81, FackSkandalis81}.

\begin{theorem} \label{thm:Connes-Thom}
Let $k\geq 1$ be given and let $\alpha: \IR^k\curvearrowright A$ and $\beta: \IR^k\curvearrowright B$ be two actions on separable \cstar-algebras.
Then the canonical forgetful map $KK^{\IR^k}(\alpha,\beta)\to KK(A,B)$, which takes a $KK^{\IR^k}$-class and sends it to the associated ordinary $KK$-class between \cstar-algebras, is an isomorphism of abelian groups.
\end{theorem}
%%%
\begin{comment}
\begin{proof}
Upon reviewing the construction of equivariant $KK$-theory, it is evident (especially through the original Kasparov picture) that the statement is true for $\alpha$ and $\beta$ being the trivial actions.

For general $\alpha$, we consider the action $\bar{\alpha}: \IR\curvearrowright A[0,1]$ given by $\bar{\alpha}_{\vec{r}}(f)(t)=\alpha_{t\vec{r}}(f(t))$ for all $f\in A[0,1]$, $\vec{r}\in\IR^k$ and $t\in [0,1]$.
Then the equivariant evaluation maps $\ev_0: (A[0,1],\bar{\alpha})\to(A,\id_A)$ and $\ev_1: (A[0,1],\bar{\alpha})\to(A,\alpha)$ are both $KK^{\IR^k}$-equivalences by \cite[Theorem 8.5]{MeyerNest06}.
Hence 
\[
\kappa_\alpha=KK^{\IR^k}(\ev_1)^{-1}\otimes KK^{\IR^k}(\ev_0) \ \in \ KK^{\IR^k}(\alpha,\id_A)
\]
is invertible.
Clearly $\Ff(\kappa_\alpha)=KK(\id_A)$ in $KK(A,A)$.
Then $x\mapsto\kappa_\alpha^{-1}\otimes x\otimes\kappa_\beta$ defines an isomorphism $KK^{\IR^k}(\alpha,\beta)\cong KK^{\IR^k}(\id_A,\id_B)$, and we see that upon further composing with the forgetful map, we have an isomorphism $KK^{\IR^k}(\alpha,\beta)\cong KK(A,B)$ which agrees with the original forgetful map $\Ff$.
\end{proof}
\end{comment}
%%%

For what follows, the reader should briefly recall the definition of the Rokhlin property for actions of $\IR^k$; see \cite[Definition 6.2]{Szabo19ssa4}.
The argument below provides an alternative proof of and generalizes \cite[Theorem A]{Szabo21R}.

\begin{cor} \label{cor:Rokhlin-multiflows}
Let $\gamma: \IR^k\curvearrowright\CO_\infty$ be the canonical quasi-free action from \autoref{def:the-model}.
Let $\beta: \IR^k\curvearrowright B$ be an action on a separable $\CO_\infty$-absorbing \cstar-algebra.
	\begin{enumerate}[leftmargin=*,label=\textup{(\roman*)}]
	\item $\beta$ has the Rokhlin property if and only if $\beta$ is isometrically shift-absorbing. \label{cor:Rokhlin-multiflows:1}
	\item Suppose $B$ is a Kirchberg algebra.
	Then $\beta$ has the Rokhlin property if and only if $\beta$ is cocycle conjugate to $\id_B\otimes\gamma^{\otimes\infty}: \IR^k\curvearrowright B\otimes\CO_\infty^{\otimes\infty}$.
	If either $B$ is unital or $\beta$ is strongly stable, then these two actions are in fact very strongly cocycle conjugate. \label{cor:Rokhlin-multiflows:2}
	\end{enumerate}
\end{cor}
\begin{proof}
\ref{cor:Rokhlin-multiflows:1}:
Note that the existence of a Rokhlin action on $\CO_\infty$ is ensured by \cite[Example 6.8]{Szabo19rd}, which uses \cite{BratteliKishimotoRobinson07}.
For the ``only if'' part, we note that $\gamma^{\otimes\infty}$ is a strongly self-absorbing action by \autoref{cor:model-action-ssa}. 
Assuming $\beta$ has the Rokhlin property, it follows by \cite[Corollary B]{Szabo19rd} that $\beta\cc\beta\otimes\gamma^{\otimes\infty}$, so in particular $\beta$ is isometrically shift-absorbing.
Conversely, if we know that $\beta$ is isometrically shift-absorbing, then it suffices to show by \autoref{prop:the-condition} that $\gamma^{\otimes\infty}$ has the Rokhlin property.
By \autoref{thm:Connes-Thom} and \autoref{thm:main-theorem}\ref{thm:main-theorem:3}, we can obtain a cocycle conjugacy between $\gamma^{\otimes\infty}$ and its tensor product with any Rokhlin action on $\CO_\infty$, which shows the claim.

\ref{cor:Rokhlin-multiflows:2}:
With \ref{cor:Rokhlin-multiflows:1} the ``if'' part is clear, so we show the ``only if'' part.
By \autoref{thm:Connes-Thom}, we have canonical (if $B$ is unital, in particular unit-preserving) $KK^{\IR^k}$-equivalences of actions $\beta\sim\id_B\sim\id_B\otimes\gamma^{\otimes\infty}$.
By \ref{cor:Rokhlin-multiflows:1} we know that $\beta$ is isometrically shift-absorbing.
So with \autoref{thm:main-theorem}, it follows that $\beta$ and $\id_B\otimes\gamma^{\otimes\infty}$ are indeed (very strongly) cocycle conjugate. 
\end{proof}

%%%%%%%%%%%%%%%%%%%%%%%%%%%%%%%%%%%%%%%%%%%%%%%%%%%%%%%%%%%%%%%%%%

\textbf{Acknowledgements.} 
The first named author has been supported by the Independent Research Fund Denmark through the Sapere Aude:\ DFF-Starting Grant 1054-00094B, and by the Australian Research Council grant DP180100595. 
The second named author has been supported by the start-up grant STG/18/019 of KU Leuven, the research project C14/19/088 funded by the research council of KU Leuven, and the project G085020N funded by the Research Foundation Flanders (FWO).

We would like to thank Stuart White and Yuhei Suzuki for some discussions that had a positive impact on this article. In addition our gratitude goes out to Sergio Giron Pacheco,  Robert Neagu and Shanshan Hua for dedicating a reading seminar in Oxford to read our work carefully, which benefited us and the readers in the form of numerous small corrections and improvements. We would lastly like to thank the referees for a careful reading and valuable suggestions.

%%%%%%%%%%%%%%%%%%%%%%%%%%%%%%%%%%%%%%%%%%%%%%%%%%%%%%%%%%%%%%%%%%

\bibliographystyle{gabor}
\bibliography{master}

\end{document}